\documentclass[reqno, 11pt, twoside]{amsart}

\makeatletter
\def\specialsection{\@startsection{section}{1}%
  \z@{\linespacing\@plus\linespacing}{.5\linespacing}%
  {\normalfont}}
\def\section{\@startsection{section}{1}%
  \z@{.7\linespacing\@plus\linespacing}{.5\linespacing}%
  {\normalfont\scshape}}
\makeatother

\usepackage{etoolbox}
\patchcmd{\section}{\scshape}{\bfseries}{}{}
\makeatletter
\renewcommand{\@secnumfont}{}
\makeatother

\makeatletter
\def\subsection{\@startsection{subsection}{1}%
  \z@{.5\linespacing\@plus.7\linespacing}{-.5em}%
  {\normalfont\itshape}}
	\def\@sect#1#2#3#4#5#6[#7]#8{%
  	\edef\@toclevel{\ifnum#2=\@m 0\else\number#2\fi}%
  	\ifnum #2>\c@secnumdepth \let\@secnumber\@empty
  	\else \@xp\let\@xp\@secnumber\csname the#1\endcsname\fi
  	\@tempskipa #5\relax
  	\ifnum #2>\c@secnumdepth
  	  \let\@svsec\@empty
  	\else
  	  \refstepcounter{#1}%
    \edef\@secnumpunct{%
      \ifdim\@tempskipa>\z@ 
        \@ifnotempty{#8}{.\@nx\enspace}%
      \else
        \@ifempty{#8}{.}{.\@nx\enspace}%
      \fi
    }%
    \@ifempty{#8}{%
      \ifnum #2=\tw@ \def\@secnumfont{\bfseries}\fi}{}%
    \protected@edef\@svsec{%
      \ifnum#2<\@m
        \@ifundefined{#1name}{}{%
          \ignorespaces\csname #1name\endcsname\space
        }%
      \fi
      \@seccntformat{#1}%
    }%
  \fi
  \ifdim \@tempskipa>\z@ 
    \begingroup #6\relax
    \@hangfrom{\hskip #3\relax\@svsec}{\interlinepenalty\@M #8\par}%
    \endgroup
    \ifnum#2>\@m \else \@tocwrite{#1}{#8}\fi
  \else
  \def\@svsechd{#6\hskip #3\@svsec
    \@ifnotempty{#8}{\ignorespaces#8\unskip
       }%
    \ifnum#2>\@m \else \@tocwrite{#1}{#8}\fi
  }%
  \fi
  \global\@nobreaktrue
  \@xsect{#5}}
\makeatother

\makeatletter
\def\abstract#1{ \gdef\@abstract{#1}}
\def\@abstract{\@latex@error{No \noexpand\abstract given}\@ehc}
\makeatother

\makeatletter
\def\keywords#1{ \gdef\@keywords{#1}}
\def\@keywords{\@latex@error{No \noexpand\keywords given}\@ehc}
\makeatother

\makeatletter
\def\MSC#1{ \gdef\@MSC{#1}}
\def\@MSC{\@latex@error{No \noexpand\MSC given}\@ehc}
\makeatother

\makeatletter
\def\abstract#1{ \gdef\@abstract{#1}}
\def\@abstract{\@latex@error{No \noexpand\abstract given}\@ehc}
\makeatother

\usepackage{verbatim}
\usepackage[letterspace=150]{microtype}
\usepackage{multicol}
\usepackage[affil-it]{authblk}
\usepackage{titling}
\usepackage[latin1]{inputenc}
\usepackage{calligra}
\usepackage[OT1]{fontenc}
\usepackage[english]{babel}
\usepackage{amsfonts}
\usepackage{amsmath}
\usepackage{amssymb}
\usepackage{amsthm}
\usepackage{faktor}
\usepackage{float}
\usepackage{enumerate}
\usepackage{color}
\usepackage{pdfpages}
\usepackage{esint}
\usepackage{cite}
\usepackage{fancyhdr}
\usepackage{mathtools}
\usepackage{mathrsfs}

\DeclareMathOperator*{\Id}{Id}

\newcommand{\Div}{\mathrm{div}}

\newcommand{\Pp}{\mathcal{P}}
\newcommand{\ee}{\varepsilon}

\newcommand{\Aa}{\mathcal{A}}

\newcommand{\Bb}{\mathcal{B}}

\newcommand{\dd}{\mathrm{d}}

\newcommand{\pre}{ \mathrm{p}}
\newcommand{\uu}{ \mathsf{u}}

\newcommand{\Y}{\mathfrak{Y}}

\newcommand{\I}{\mathcal{I}}

\newcommand{\FF}{\mathcal{F}}

\newcommand{\RR}{\mathbb{R}}
\newcommand{\ZZ}{\mathbb{Z}}

\newcommand{\NN}{\mathbb{N}}

\newcommand{\BB}{\dot{B}}
\newcommand{\Hh}{\dot{H}}
\newcommand{\Dd}{\dot{\Delta}}
\newcommand{\Sd}{\dot{S}}

\newtheorem{theorem}{Theorem}[section]

\newtheorem{prop}[theorem]{Proposition}
\newtheorem{lemma}[theorem]{Lemma}
\newtheorem{definition}[theorem]{Definition}
\newtheorem{remark}[theorem]{Remark}

\DeclareFontFamily{OT1}{rsfs}{}
\DeclareFontShape{OT1}{rsfs}{m}{n}{ <-7> rsfs5 <7-10> rsfs7 <10-> rsfs10}{}
\DeclareMathAlphabet{\mycal}{OT1}{rsfs}{m}{n}

\def\q{{\bf q}}

\def\m{{\mathbf{m}}}

\DeclareSymbolFont{letters}{OML}{ztmcm}{m}{it}

\begin{document}
 \begin{center}%
  \noindent\rule{185mm}{0.01cm}
  \let \footnote \thanks
  {\LARGE  {\sc The Fujita-Kato theorem for some Oldroyd-B model} \par}%
  \noindent\rule{185mm}{0.01cm}
  \vskip 1.em
  \today	
 \end{center}
 \vskip 1.em%
 {\let \footnote \thanks  Francesco De Anna \par}
 {\textsl{\footnotesize{Institut f\"ur Mathematik, Universit\"at W\"urzburg, Germany\\
			e-mail: francesco.deanna@mathematik.uni-wuerzburg.de} } }
 \vskip 0.1cm\hspace{-0.5cm}
 {\noindent \let \footnote \thanks Marius Paicu \par}
 {\noindent \textsl{\footnotesize{{Institut de Math{\'e}matiques de Bordeaux, Universit\'e de Bordeaux, France\\
 \noindent  e-mail: Marius.Paicu@math.u-bordeaux.fr} } }
 \par
 \vskip 1.em
 \noindent\rule{185mm}{0.01cm} 
 \vskip 1.em

{\small
  \begin{minipage}{0.3\linewidth}
  \noindent\hspace{-0.6cm} \textsc{article info}
  
  \noindent\hspace{-0.4cm}\rule{4.3cm}{0.01cm} 
  
  \end{minipage}
  \begin{minipage}{0.6\linewidth}
   \noindent \noindent \textsc{ abstract}
   
   \noindent\rule{12.42cm}{0.01cm} 
  
  \end{minipage}
	
  \smallskip
  \hspace{-0.4cm}
  \begin{minipage}{0.29\linewidth}
  
  \noindent \hspace{-0.0cm}\textsl{Keywords:}
  Oldroyd-B model,  finite energy solutions, Lipschitz flow, critical regularities.
  
  \vspace{0.4cm}
  \noindent \hspace{-0.0cm}\textsl{MSC:}
  35Q35, 35B65, 76D05, 76N10.
  \end{minipage}
  \hspace{0.34cm}
  \begin{minipage}{0.6\linewidth}$\,$
  
  		\noindent In this paper, we investigate the Cauchy problem associated to a system of PDE's of Oldroyd type. The considered model describes the evolution of certain viscoelastic fluids within a corotational framework.  The non corotational setting is also addressed in dimension two.
      
      \noindent We show that some widespread results concerning the incompressible Navier-Stokes equations can be extended to the considered systems.  
      
      \noindent In particular we show the existence and uniqueness of global-in-time classical solutions for large data in dimension two. This result is supported  by suitable condition on the initial data to provide a global-in-time Lipschitz  regularity for the flow, which allows  to overcome specific challenging due to the non time decay of the main forcing terms. 
  							
  	  \noindent Secondly, we address the global-in-time well posedness in dimension $\dd\geq 3$. We prove the propagation 
  	  of Lipschitz regularities for the flow. For this result, we just assume the initial data to be sufficiently 
  	  small in a critical Lorentz space.
  
  		$\,$
  \end{minipage}
  
  \noindent\rule{185mm}{0.01cm} 
  } 	
\makeatother
\allowdisplaybreaks{} 
 
\section{Introduction}
\noindent The modeling and analysis of the hydrodynamics  of viscoelastic fluids has attracted much attention over the last decades \cite{BM, CM, LLZ, FCGO, M, LM}. Generally, the physical state of matter of a material  can be determined by the degree of freedom for the movement of its constitutive molecules. Increasing this degree of freedom, the most of materials evolves from a solid state to a liquid phase and eventually to a gas form. Nevertheless, there exist in nature several materials that present characteristics in the between of an isotropic fluid and a crystallized solid. These materials are usually classified as viscoelastic fluids, since they generally behave as a viscous fluids as well as they share some properties with elastic materials. The most common ones, for instance, are related to memory effects. We refer the reader to \cite{CaraGuillOrtega, MRenardy, BrisLel, Oswald, Saut} for an overview of the Physics behind the modeling of these complex fluids.

\noindent
This article is devoted to the analysis of the following evolutionary system of PDE's, describing the hydrodynamics of specific incompressible  viscoelastic fluids:
\begin{equation}\label{main_system}
\left\{\hspace{0.2cm}
	\begin{alignedat}{2}
		&\,\partial_t \tau\,+\,\uu\cdot \nabla \tau 	- \,\omega\, \tau \,+\,\tau\,\omega \,=\,\mu \mathbb{D}
		\hspace{3cm}
		&&\RR_+\times \RR^\dd, \vspace{0.1cm}	\\		
		&\,\partial_t \uu  + \uu\cdot \nabla  \uu\, - \, \nu \Delta  \uu  
		\,+\, \nabla \pre  =
		\Div\,\tau
									\hspace{3cm}									&& \RR_+\times \RR^\dd, \vspace{0.1cm}\\
		&\,\Div\,\uu\,=\,0			
		&&\RR_+\times \RR^\dd, \vspace{0.1cm}	\\			
		& (\uu,\,\tau)_{|t=0}	\,=\,(\uu_0,\,\tau_0)			
		&&\hspace{1.02cm} \RR^\dd.\vspace{0.1cm}																							
	\end{alignedat}
	\right.
\end{equation}
Here $\mu\geq 0$, the constant $\nu>0$ stands for the viscosity of the fluids, while  the state variables correspond to $\uu= \uu(t,x)$, the velocity field of a particle $x\in \RR^n$ at a time $t\in\RR$, and $\tau=(\tau(t,x)_{ij})_{i,j=1,\dots, \dd}$, the conformation tensor in $\RR^{\dd\times\dd}$ describing the internal elastic forces that the constitutive molecules exert on each other. To simplify our analysis, we assume our non-Newtonian fluid to occupy the entire  whole space $\RR^\dd$, with dimension $d\geq 2$. 
The evolutionary equation for the conformation tensor $\tau$ is then driven by the vorticity tensor $\omega= (\omega(t,x))_{i,j=1,\dots \dd}$, which stands for the skew-adjoint part of the deformation tensor $\nabla \uu$:
\begin{equation*}
	\omega \,=\, \frac{\nabla \uu\,-\,^t\nabla \uu}{2}.
\end{equation*}
The system can be seen as a simplified version of the more general Oldroyd-B model (cf. \cite{CM}):
\begin{equation}\label{Oldroyd}
\left\{\hspace{0.2cm}
	\begin{alignedat}{2}
		&\,\partial_t \tau\,+\,\uu\cdot \nabla \tau\,+\,a\,\tau\, 	+\,\mathcal{Q}(\uu,\,\tau) \,=\,\mu_2\, \mathbb{D}
		\hspace{3cm}
		&&\RR_+\times \RR^\dd, \vspace{0.1cm}	\\		
		&\,\partial_t \uu  + \uu\cdot \nabla  \uu\, - \, \nu \Delta  \uu  
		\,+\, \nabla \pre  =
		\mu_1\Div\,\tau
									\hspace{3cm}									&& \RR_+\times \RR^\dd, \vspace{0.1cm}\\
		&\,\Div\,\uu\,=\,0			
		&&\RR_+\times \RR^\dd, \vspace{0.1cm}	\\			
		& (\uu,\,\tau)_{|t=0}	\,=\,(\uu_0,\,\tau_0)			
		&&\hspace{1.02cm} \RR^\dd.\vspace{0.1cm}																							
	\end{alignedat}
	\right.
\end{equation}
The main parameters $\nu, \,a,\, \mu_1,\, \mu_2$ are assumed to be non-negative and they are specific to the characteristic of the considered material. In particular  \cite{LM} $\nu$, $a$ and $b$ correspond respectively to $\theta/{\rm Re}$, $1/{\rm We}$ and $2(1-\theta)/({\rm We}\, {\rm Re})$, where ${\rm Re}$ is the Reynolds number, ${ \theta}$ is the ratio between the so called relaxation and retardation times and ${\rm We}$ is the Weissenberg number (see also \cite{FeGuOr}, equations $(1.5)$--$(1.6)$). The bilinear term $\mathcal{Q}(\uu,\,\tau)$ reads as follows:
\begin{equation}\label{bilinear-term}
	\mathcal{Q}(\uu,\,\tau)\,=\,\tau\,\omega -\omega \tau\,+\,b (\mathbb{D}\,\tau + \tau\,\mathbb{D}),
\end{equation}
where the so-called slip parameter $b$ is a constant value between $[-1, 1]$ and $\mathbb{D}=(\nabla \uu + \nabla \uu^T)/2$ is the symmetric contribution of the deformation tensor $\nabla \uu$. The corotational term $\tau\,\omega -\omega \tau$ describes how molecules are twisted by the underlying flow, while the term depending on $b$ describes how molecules are stretched and deformed by the flow itself.

\noindent For the sake of our analysis, in this paper we mainly impose the following restriction on the main parameters of the Oldroyd-B model:
\begin{equation*}
	\mu_1 = 1\quad\quad \mu_2 = \mu\geq 0,\quad\quad b=0,
\end{equation*}
from which one obtain our main system \eqref{main_system}. The first condition is introduced just for the sake of a clear presentation, while the second and third conditions will play a major role in the analysis techniques we will perform in the forthcoming sections.  The case of $\mu_2=0$ 
can roughly be interpreted as the case of the infinite Weissenberg number limit. Furhtermore, following \cite{CM,FeGuOr}, we determine that the limit model $\mu=0$ occurs when $\theta$, the ratio between the so called relaxation and retardation times, is converging to $1$. Indeed,  when $\mu_1=1$, we have $\mu_2=\frac{2\nu}{\lambda_1}(1-\theta)$ where $\lambda_1$ is the retardation time and $\theta=\frac{\lambda_2}{\lambda_1}$ with $\lambda_2$ being the relaxation time. We finally note that some different Oldroyd models with infinite Weissenberg number (and so $\mu_2=0$) have been studied in \cite{Huo}. We additionally assume that $a=0$, increasing the challenging of our model, since no damping effect is now assumed on the evolution of the conformation tensor $\tau$. We claim that all our results hold also for the damped case given by $a>0$, which is somehow much easier to treat.

\smallskip
\noindent
We also investigate the case of a non-corotational setting, i.e. when $b>0$ (cf. Theorem \ref{main_thm0b}.

\smallskip
\noindent
We present an overview of some recent results concerning systems \eqref{main_system} and \eqref{Oldroyd}. In \cite{GS}, the authors dealt with the existence and uniqueness of local strong solutions for system \eqref{Oldroyd} in Sobolev space $H^s(\Omega)$, for a sufficiently-smooth  bounded domain $\Omega$ and a sufficiently large regularity $s>0$.  The same authors in \cite{GS2} showed that these solutions are global if the initial data are sufficiently small as well it is small the coupling between the main terms of the constitutive equations.

\noindent Lions and Masmoudi \cite{LM} addressed the corotational case of system \eqref{Oldroyd}, given by $b=0$ and showed the existence and uniqueness of global-in-time weak solutions in a bidimensional setting.

\noindent Lei, Liu and Zhou \cite{LLZ1} proved existence and uniqueness of classical solutions near equilibrium of system \eqref{Oldroyd} for small initial data, assuming the  domain to be periodic or to be the whole space.

\noindent Bresh and Prange \cite{BP} analyze the low Weissenberg asymptotic limit of solutions for system \eqref{Oldroyd} in a corotational setting $b=0$. They focus on the specific formulation of the Oldroyd-B system in which the main parameter of \eqref{Oldroyd} are explicitly  defined in terms of the Weissenberg number {\rm We}, which compares the viscoelastic relaxation time to a time scale relevant to the fluid flow. The authors study the weak convergence towards the Navier--Stokes system, as {\rm We$\rightarrow 0$}. Furthermore, they take into account the presence of defect measures in the initial data and show that they do not perturb the Newtonian limit of the corotational system.

\noindent Chemin and Masmoudi \cite{CM} proved the existence and uniqueness of strong solutions of the Oldroyd-B model \eqref{Oldroyd}, within the framework of homogenous  Besov spaces with critical index of regularity. The authors particularly showed local and global-in-time existence of solutions for large and small initial data, respectively, under the assumption of a smallness condition on the coupling parameters of system \eqref{Oldroyd}. In \cite{ZiFaZh}, Zi, Fang and Zhand extended the mentioned result, relaxing this smallness condition. 
The main results of this manuscript should be seen as suitable improvements of \cite{CM, ZiFaZh}, within the setting of system \eqref{main_system}. In particular, we relax the assumption of  small initial data for global-in-time solution, considering small functions in weak Lebesgue spaces (cf. Theorem \ref{main_thm2}). 

\noindent Elgindi and Rouss\'et \cite{ER} addressed the well-posedness of a system related to \eqref{main_system}. On the one hand they introduced a dissipative and regularizing mechanism on the evolution of the deformation tensor $\tau$. On the other hand they considered the case of a null viscosity $\nu = 0$, i.e. of an Euler type equation for the flow $u$. Neglecting the bilinear term \eqref{bilinear-term}, they proved the global existence of classical solution for large initial data in dimension two. Furthermore, for a general bilinear term \eqref{bilinear-term}, they showed the existence of global-in-time classical solution for small initial data. These initial data belong to a somehow-similar space as the one introduced in our Theorem \ref{main_thm0}.

\smallskip
\noindent In this article, we aim to show that three of the most widespread results about the Navier-Stokes equations can be extended to the Oldroyd-B model \eqref{main_system}, under suitable condition on the initial data. More precisely:
\begin{itemize}
	\item Existence of global-in-time classical solutions in dimension two for large initial data,
	\item	Uniqueness in dimension two of strong solutions,
	\item Existence and uniqueness of global-in-time strong solution in dimension $\dd\geq 3$, with a Fujita-Kato \cite{FK} smallness condition for the initial 
	data.
\end{itemize}
The first problem we address in this article is the existence and uniqueness of global-in-time solutions when dealing with large initial data in $L^2(\RR^2)$. In the case of the classical Navier-Stokes system, existence of weak solutions \`a la Leray was proven in \cite{JLeray} while the uniqueness was showed by Lions and Prodi in \cite{LionsProdi}. In our contest, specific difficulties arise when transposing these results to the Oldroyd-B system. This difficulties should be recognized within the intrinsic structure of the model \eqref{main_system}:
\begin{itemize}
	\item 	the equation of the conformation tensor $\tau$ is of hyperbolic type, sharing the majority of the behaviours of a standard transport equation,
	\item 	the fluid equation is driven by a forcing term which behaves as the gradient of the conformation tensor, complicating the behavior of the flow $\uu$ 
			for large value of time.
\end{itemize}
When dealing with just the Navier-Stokes equations, it is rather common to build Leray-type solutions making use of compactness methods. However, the specific structure of the hyperbolic equation for $\tau$ adds complexity when transposing these techniques to our system. For instance, when dealing with nonlinearities such as 
\begin{equation*}
	-\omega \tau\,+\,\tau\omega\quad\quad\omega = \frac{\nabla \uu -\,^t\nabla \uu}{2},
\end{equation*}
in the $\tau$-equation, one should recognize the typical difficulty related to the product of two weakly convergent sequences, when passing to the limit of suitable approximate solutions. In order to overcome this issue, we assume some extra regularity on the initial tensor $\tau_0$ that will somehow allow to achieve suitable strong convergences of solutions. One can evince how the hyperbolic behavior of the conformation equation counteracts against the propagation of this regularity. Typically, this can be overcome when dealing with sufficiently regular flow, such as velocity field $\uu$ that are Lipschitz in space. This Lipschitz condition, however, is above the properties of Leray type solutions and this leads in considering an additional regularities also for the initial velocity field $\uu_0$. 
We can hence summarize our first result in the following statement:
\begin{theorem}\label{main_thm0}
	Consider system \eqref{main_system} within the bidimensional case $d\,=\, 2$, $\mu=0$, $b = 0$ and 
	$a\geq 0$. Let the initial data $\uu_0$ be a free divergence vector field in 
	$L^2(\RR^2)\cap \BB_{p, 1}^{2/p-1}$, and $\tau_0$ an element of 
	$L^2(\RR^2)\cap \BB_{p, 1}^{2/p}$, for an index $p\in (2, \infty)$. 
	When $\mu = 0$, then the system \eqref{main_system} admits a  global-in-time classical solution 
	$(\uu,\,\tau)$ within the functional framework
	\begin{equation*}
	\begin{alignedat}{16}
		\uu		\,	&\in	\,&& L^\infty_{\rm loc}(\RR_+;  \,L^2(\RR^2))&&&&\cap 
		L^2_{\rm loc}(\RR_+,\dot H^1(\RR^2)),\quad\quad
		&&&&&&&&\tau	\,	\in	\, L^\infty(\RR_+,\,L^2(\RR^2)),\\
		\uu		\,	&\in	\,&& L^\infty_{\rm loc}(\RR_+;  \,\BB_{p,1}^{\frac{2}{p}-1})
		&&&&\cap 
		L^1_{\rm loc}(\RR_+,\BB^{\frac{2}{p}+1}_{p,1}),\quad\quad
		&&&&&&&&\tau	\,	\in	\, L^\infty_{\rm loc}(\RR_+,\,\BB_{p,1}^{\frac{2}{p}}).
	\end{alignedat}
	\end{equation*}
	The local Lebesgue conditions in time can be replaced with global ones when $a>0$. 
	This solution is unique if $p\in [1, 4]$. 	
	In addition, the Besov regularities satisfy
	\begin{equation*}
	\begin{aligned}
		\|\,\tau(t)\,\|_{\BB_{p,1}^{\frac{2}{p}}}
		\,&\leq\,
		e^{-at}
		\|\,\tau_0\,\|_{\BB_{p,1}^{\frac{2}{p}}}\exp\bigg\{\,C\nu^{-1}\,\Upsilon_{1,\nu}(t, \,\uu_0,\,\tau_0)\,\bigg\},\\
		\|\,\uu(t)\,\|_{\BB_{p,1}^{\frac{2}{p}-1}}
		\,+\,
		\nu\int_0^t
		\|\,\uu(s)\,\|_{\BB_{p,1}^{\frac{2}{p}+1}}
		\dd s
		\,
		\,&\leq\,
		\Big(
			\|\,\uu_0\,\|_{\BB_{p,1}^{\frac{2}{p}-1}}
			\,+\,
			\Theta_a(t)\|\,\tau_0\,\|_{\BB_{p,1}^{\frac{2}{p}}}
		\Big)\times \\&\hspace{2cm}\times 
		\exp\bigg\{C\nu^{-1}\Upsilon^1_\nu(t,\,\uu_0,\,\tau_0)\big(\nu^{-1}\Upsilon^2_\nu(t,\,\uu_0,\,\tau_0)+1\big)\bigg\}.
	\end{aligned}
	\end{equation*}
	where $\Theta_a(t) = t$ if $a=0$ while $\Theta_a(t) = (1-e^{-at})/a$ if $a>0$. Furthermore $\Upsilon^1_\nu(t,\,\uu_0,\,\tau_0)$ and $\Upsilon^2_\nu(t,\,\uu_0,\,\tau_0)$ are two smooth functions depending on the time $T$ and on the norms of $(\uu_0,\,\tau_0)$ in $L^2(\RR^2)\cap \BB_{\infty, 1}^{-1} \times L^2(\RR^2)\cap \BB_{\infty, 1}^0$ (cf. Definition \ref{def-Upsilon}). 
	
	\noindent The result is extended to the case of $\mu>0$, for a local in time classical solution defined on 				$[0,T_{\rm max})$, 
	whose lifespan $T_{\rm max}>0$ satisfies the lower bound
	\begin{equation*}
		T_{\rm max} \geq \sup \Big\{ 
		T>0 \quad \text{such that}\quad 
		C\frac{\mu}{\nu^{2}}T^2 \Psi_{2,\nu,\mu}(T,\uu_0,\tau_0)
		e^{
		2\frac{C}{\nu} 
		\Psi_{2,\nu,\mu}(T,\uu_0,\tau_0)
		\frac{1-e^{-aT}}{a}
		\|\tau_0 \|_{\BB_{\infty,1}^0}
		+
		2C\frac{\mu}{\nu}
		\Psi_{2,\nu,\mu}(T,\uu_0,\tau_0)
		T
		}
		< 1
		\Big\},
	\end{equation*}	 
	for a fixed positive constant $C$ and a suitable function $\Psi_{2,\nu,\mu}(T,\uu_0,\tau_0)$ 
	(cf. Definition \ref{def-Upsilon} below). 
\end{theorem}

\noindent 
One of the main novelty of the theorem is the arbitrariness of the parameter $a\geq 0$. When $a>0$ then a damping mechanism insures that any classical solution is uniformly bounded in time. However, the theorem provides also the existence and uniqueness of global-in-time classical solutions when the damping mechanism is neglected, i.e. $a=0$. In this scenario the norms of the solution grow up exponentially in time, never blowing up.

\smallskip\noindent
When $\mu=\mu_2>0$ in \eqref{Oldroyd} and $b>0$ in \eqref{bilinear-term} are sufficiently small, we are further capable to provide a unique global-in-time classical solution under additionally regularity on the initial data. 
The smallness condition on $\mu$ and $b$ depends on the size of the initial data. We further require the damping coefficient $a>0$ to be strictly positive. 
\begin{theorem}\label{main_thm0b}
	Consider system \eqref{main_system} within the bidimensional case $d\,=\, 2$, $a>0$, $\mu>0$ and 
	$b>0$. There exists a constant $c>0$ depending on the parameter $a>0$ such that if  
	\begin{equation*}
		|b|
		+ 
		\mu
		\leq 
		\frac{c}{1+ 
		\exp\Big\{
		\| \tau_0 \|_{L^2(\mathbb{R}^d)}	+
		\| \tau_0 \|_{\BB_{p,1}^{\frac{\dd}{p}}\cap \BB_{p,1}^{\frac{\dd}{p}+1}}
		+
		\| \uu_0 \|_{L^2(\mathbb{R}^d)}
		+
		\| \uu_0 \|_{\BB_{p,1}^{\frac{\dd}{p}-1}\cap \BB_{p,1}^{\frac{\dd}{p}}}\Big\}}, 
	\end{equation*}	
	i.e. if $\mu$, $|b|$ are sufficiently small in relation to the initial data and the parameter $a$, 
	then there exists a global-in-time classical solution for problem \eqref{main_system}. 
\end{theorem}

\noindent 
As pointed out, the initial condition $(\uu_0,\,\tau_0)$ that belongs to $\BB_{p,1}^{2/p-1}\times \BB^{2/p}_{p,1}$ (cf. Section \ref{sec:Besov} for some details about these functional spaces) is the precursor of the Lipschitz behavior of the fluid $\uu$. Nevertheless, the real regularity which unlock the Lipschitz condition for $\uu$ is somehow hidden in the above statement. We overview some specific about that:
when considering the case $p=2$, that is  $(\uu_0,\,\tau_0)\in \BB_{2,1}^0\times \BB_{2,1}^1$, we are dealing with a strict subcase of the framework $(\uu_0,\,\tau_0)\in L^2(\RR^2)\times \dot H^1(\RR^2)$, where $\Hh^1(\RR^2)$ stands for the homogeneous Sobolev space. It is well known that just considering the simplified case of a transport equation
\begin{equation*}
	\partial_t \tau\,+\,\uu\cdot \nabla \tau = 0,
\end{equation*}
the Sobolev regularity $\Hh^1(\RR^2)$ is propagated by a Lipschitz flow with the following exponential growth:
\begin{equation*}
	\|\,\tau(t)\,\|_{\dot H^1(\RR^2)}\,\leq\,\|\,\tau_0\,\|_{\dot H^1(\RR^2)}\exp\left\{\int_0^t\|\,\uu\,\|_{\mathcal{L}ip}\right\}.
\end{equation*}
Coupling this inequality with the structure of system \eqref{main_system} would eventually lead to a bound for the Lipschitz regularity of the flow $\uu$ of the following type:
\begin{equation*}
	\frac{\dd}{\dd t}\|\,\uu\,\|_{\mathcal{L}ip}\,\leq\, C(\|\, \uu_0\,\|_{\mathcal{L}ip}, \|\,\tau_0\,\|_{\dot H^1(\RR^2)})\exp\left\{\int_0^t\|\,\uu\,\|_{\mathcal{L}ip}\right\},
\end{equation*}
for which no global-in-time bound is automatically determined. Hence, we will first propagate the norm of the initial data within a largest functional framework than the one specifically stated in Theorem \ref{main_thm0}, namely we will propagate the following regularities:
\begin{equation*}
		\uu_0\,\in\,\BB_{\infty, 1}^{-1}\quad\quad\text{and}\quad\quad \tau_0\in \BB_{\infty, 1}^0,
\end{equation*}
in which $\BB_{p,1}^{2/p-1}$ and $\BB_{p,1}^{2/p}$ are embedded, respectively. We will show that this particular choice is essential to achieve a linear growth of the Lipschitz regularity: 
\begin{equation*}
	\|\,\tau(t)\,\|_{\BB_{\infty, 1}^{0}}\,\leq\,\|\,\tau_0\,\|_{\BB_{\infty, 1}^{0}}\left(1 +\int_0^t\|\,\uu\,\|_{\mathcal{L}ip}\right)
	\;\Rightarrow\;
	\frac{\dd}{\dd t}\|\,\uu\,\|_{\mathcal{L}ip}\,\leq\, C(\|\,\uu_0\,\|_{\mathcal{L}ip}, \|\,\tau_0\,\|_{\dot H^1(\RR^2)})\left(1 +\int_0^t\|\,\uu\,\|_{\mathcal{L}ip}\right).
\end{equation*}
Thus a standard Gronwall inequality will unlock the propagation of Lipschitz regularity for $\uu$, which is our main tool to build global-in-time classical solutions.
\begin{definition}\label{def-Upsilon}
	In Theorem \ref{main_thm0} we have avoided to explicitly present the form of 
	$\Upsilon^1_\nu(T,\, \uu_0,\,\tau_0)$ and $\Upsilon^2_\nu(T,\, \uu_0,\,\tau_0)$ for the 
	sake of a compact formulation. We report here their exact expression. We first introduce the functional 
	\begin{equation*}
	\begin{aligned}
			\Gamma_{a,\nu}(T) &:=\sqrt{\frac{1-e^{-2aT}}{2a\nu}}\quad\text{when }a>0,\qquad 
			\Gamma_{a,\nu}(T):= \nu^{-\frac{1}{2}}T^\frac{1}{2}\quad\text{when }a=0\\
			\Phi_{\nu,\mu,a}(T, \uu_0,\tau_0)=
			&\begin{cases}
		\left(
			\| \uu_0  \|_{L^2(\RR^2)}
			+
			\nu^{-1}\| \,\uu_0 \, \|_{L^2(\RR^2)}^2
			+
			\| \tau_0  \|_{L^2(\RR^2)}
			\Gamma_{a,\nu}(T)
			+
			\nu^{-1}
			\| \tau_0  \|_{L^2(\RR^2)}^2
			\Gamma_{a,\nu}(T)^2
		\right)^2\; &\text{if}\quad \mu = 0,\\
		\Big\{
			\big(1+\Gamma_{a,\nu}(T)\mu^\frac{1}{2}\big)
			\| \uu_0  \|_{L^2(\RR^2)}
			+ 
			\nu^{-1}\| \uu_0  \|_{L^2(\RR^2)}^2		
			+ \\ \hspace{2.3cm}+
			\big(\mu^{-\frac{1}{2}}
			+
			\Gamma_{a,\nu}(T)
			\big)
			\| \tau_0  \|_{L^2(\RR^2)}
			+
			\nu^{-1}
			\| \tau_0  \|_{L^2(\RR^2)}^2
			\Gamma_{a,\nu}(T)^2
		\Big\}^2	
		 &\text{if}\quad \mu > 0,
			\end{cases}\\
			\Psi_{1,\nu,\,\mu,\,a}(T,\,\uu_0,\,\tau_0)
			\,&=\,
			C
		\bigg\{\,
			\nu^{-\frac 32}
			\Phi_{\nu,\,\mu,\,a}(T, \,\uu_0, \,\tau_0\,)^2
			\,+\,
			\nu^{-\frac{5}{4}}
			\Phi_{\nu,\,\mu,\,a}(T, \,\uu_0, \,\tau_0\,)
			\|\,\uu_0\,\|_{L^2(\RR^2)}
		\bigg\},			\\
		\Psi_{2,\,\nu,\,\mu,\,a}(T,\,\uu_0,\,\tau_0)
		\,&=\,
		C
		\left\{
			\Gamma_{a,\nu}(T)
			\big(\mu\|\uu_0\|_{L^2(\RR^2)}^2+\|\,\tau_0\,\|_{L^2(\RR^2)}^2
			\big)^\frac{1}{2}
			\,+\,
			\nu^{-\frac{5}{4}}\Phi_{\nu,\,\mu,\,a}(T, \,\uu_0, \,\tau_0\,)
			\,+\,
			\nu^{-1}
			\|\,\uu_0\,\|_{L^2(\RR^2)}
		\right\},
	\end{aligned}
	\end{equation*}
	for a fixed and sufficiently large constant $C$. When $\mu=0$, the exact formulations of 
	$\Upsilon^1_{\nu,\mu}$ is given by 
	\begin{equation*}
		\begin{aligned}
			\Upsilon^1_{\nu,\mu,a}(T,\, \uu_0,\,\tau_0)\,:=\,
			\nu^{-1}
			\Big\{\,
			\|\,\uu_0\,\|_{\BB_{\infty,1}^{-1}}
			\,+\,
			\Psi_{1,\nu,\mu,a}(T,\,\uu_0,\,\tau_0)
			\,	+\,
			C
			&\Big(\,
			\Psi_{2,\nu,\mu,a}(T,\,\uu_0,\,\tau_0)
			\,+\,
			C
			\Big)
			\|\,\tau_0\,\|_{ \BB_{\infty,1}^0}
			\Theta_a(T)
			\Big\}{\scriptstyle\times}\\&{\scriptstyle\times}
			\exp
			\left\{
				C\nu^{-1}\Theta_a(T)
				\Psi_{2,\nu,\mu}(T,\,\uu_0,\,\tau_0)
			\right\},
		\end{aligned}
		\end{equation*}
		where $\Theta_a(T)=(1-e^{-at})/a$ when $a>0$, while $\Theta_a(T)=T$ when $a=0$.
		Furthermore, for $\mu>0$,
		\begin{equation*}
		\begin{aligned}
			\Upsilon^1_{\nu,\mu,a}(T,\uu_0,\tau_0):=
			\frac{\nu^{-1}
			\Big\{
			\|\uu_0\|_{\BB_{\infty,1}^{-1}}
			+
			\Psi_{1,\nu,\mu,a}(T,\uu_0,\tau_0)
				+
			C
			\Big(
			\Psi_{2,\nu,\mu,a}(T,\uu_0,\tau_0)
			+
			C
			\Big)
			\|\tau_0\|_{ \BB_{\infty,1}^0}
			\Theta_a(T)
			\Big\}}
			{
			1-
			C\frac{\mu}{\nu^{2}}T^2 \Psi_{2,\nu,\mu}(T,\uu_0,\tau_0)
		e^{
		2\frac{C}{\nu} 
		\Psi_{2,\nu,\mu,a}(T,\uu_0,\tau_0)
		\Theta_a(T)
		\|\tau_0 \|_{\BB_{\infty,1}^0}
		+
		2C\frac{\mu}{\nu}
		\Psi_{2,\nu,\mu,a}(T,\uu_0,\tau_0)
		T
		}
			}
			{\scriptstyle\times}\\{\scriptstyle\times}
			\exp\Big\{
		\frac{C}{\nu} 
		\Psi_{2,\nu,\mu,a}(T,\uu_0,\tau_0)
		\Theta_a(T)
		\|\tau_0 \|_{\BB_{\infty,1}^0}
		+
		2C\frac{\mu}{\nu}
		\Psi_{2,\nu,\mu}(T,\uu_0,\tau_0)
		T
		\Big\}.
		\end{aligned}
		\end{equation*}
		Finally
		\begin{equation*}
		\begin{aligned}
			\Upsilon^2_{\nu,\mu}&(T,\, \uu_0,\,\tau_0)\,:=\,
			\|\,\uu_0\,\|_{\BB_{\infty,1}^{-1}}
			\,+\,
			\Psi_{1,\nu,\mu}(T,\,\uu_0,\,\tau_0)
			\,	+\,
			C
			\Big(\,
			\Psi_{2,\,\nu}(T,\,\uu_0,\,\tau_0)
			\,+\,
			1
			\Big)
			\|\,\tau_0\,\|_{ \BB_{p,1}^0}
			T\,+\\
			\,&+\,
			C\nu
			\|\,\tau_0\,\|_{ \BB_{\infty,1}^0}
			\Big(\,
			\Psi_{2,\nu,\mu}(T,\,\uu_0,\,\tau_0)
			\,+\,
			1
			\Big)
			\int_0^T
			\Upsilon^1_{\nu,\mu}(t,\,\uu_0,\,\tau_0)
			\dd t+
			C\mu\nu
			\Psi_{2,\nu,\mu}(T,\,\uu_0,\,\tau_0)
			\int_0^T
			\Upsilon^1_{\nu,\mu}(t,\,\uu_0,\,\tau_0)
			\dd t
			.
		\end{aligned}
		\end{equation*}
\end{definition}
\begin{remark}
	Theorem \ref{main_thm0} should be seen as a further extension of the analysis started by Lions and Masmoudi 
	in \cite{LM}. The authors indeed achieved the following estimates on the velocity field $u$:
	\begin{itemize}
		\item $\nabla u \in L^p(0,T; L^q(\mathbb{R}^\dd))$, $2\leq q \leq 3$ and $1\leq p\leq \frac{q}{2q-3}$, when 
				$\dd\geq 3$,
		\item $\nabla u \in L^p(0,T; L^q(\mathbb{R}^\dd))$, $2\leq q < \infty$ and $p> \frac{q}{q-1}$,
		 when $\dd = 2$.
	\end{itemize}
	Both cases allowed  the existence of global-in-time weak solutions, however they did not provide 
	the Lipschitz regularity $\nabla u\in L^1(0,T;L^\infty(\mathbb{R}^\dd))$ that would 
	have unlocked the existence of global-in-time classical solutions. We achieve such a property in this manuscript.
\end{remark}

\smallskip
\noindent 
The last part of our manuscript is devoted to establish global-in-time strong solutions for small initial data in dimension $\dd\geq 3$. We intend to proceed similarly as in the result of Fujita-Kato \cite{FK} for the incompressible Navier-Stokes equations, as well as in the result of Paicu and Danchin in \cite{PD} for the so-called Boussinesq system. 

\noindent We recall that the Fujita-Kato Theorem is proven by reformulating the Navier-Stokes system as a fixed point problem about an operator which is built in terms of the Stokes semigroup. Under suitable smallness condition, the Picard fixed-point Theorem allows to achieve an unique strong solution, which is global in time. This approach is specifically effective when the considered functional space is critical under the scaling behavior of the Navier-Stokes equation. More precisely a functional space that preserves the norm of a solution $\uu(t,x)$ under the scaling 
\begin{equation*}
	\uu(t,x) \rightarrow \lambda \uu(\lambda^2 t,\,\lambda x),\quad\quad\text{with}\quad\lambda\geq 0.
\end{equation*}
Fujita and Kato showed that the Navier Stokes system is well posed for small initial data that belong to the homogeneous Sobolev Space $\dot H^{1/2}(\RR^3)$. 

\noindent A trivial computation shows that the Oldroyd-B system \eqref{main_system} is invariant under the following transformation:
\begin{equation*}
	(\uu(t,x),\,\tau(t,x))\rightarrow (\lambda \uu(\lambda^2t,\,\lambda x),\,\lambda^2 \tau(\lambda^2,\,\lambda x)).
\end{equation*}
Hence, the critical functional framework for the velocity field $\uu$ is the same as in the case of the Navier-Stokes equations, while we need to impose an additional derivative for the conformation tensor $\tau$.

\noindent We hence prove the following local result of solutions for system \eqref{main_system} within critical regularities:
\begin{theorem}\label{main_thm1}
	Let $\mu=0$. Consider a dimension $d\,\geq\, 3$, let the initial data $\uu_0$ be a free divergence vector field in $\BB_{p, 1}^{\frac{\dd}{p}-1}$, while $\tau_0$ belongs to $\BB_{p, 1}^{\frac{\dd}{p}}$, for a parameter $p\in [1, 2\dd)$. 
	Then there exists a time $T^*>0$ for which system \eqref{main_system} admits a unique local solution $(\uu,\,\tau)$ within
	\begin{equation*}
	\begin{alignedat}{4}
		\uu		\,	&\in	\,&& \mathcal{C}([0, T], \,\BB_{p, 1}^{\frac{\dd}{p}-1})\cap L^1(0, T,\BB_{p, 1}^{\frac{\dd}{p}+1}),\quad\quad
		\tau	\,	\in	\, \mathcal{C}([0, T], \,\BB_{p, 1}^{\frac{\dd}{p}}),
	\end{alignedat}
	\end{equation*}
	for any $T<T^*$. Furthermore if $\uu$ belongs to $ L^\infty(0,T^*,\BB_{p, 1}^{\frac{\dd}{p}-1},)\cap L^1(0, T^*,\BB_{p, 1}^{\frac{\dd}{p}+1})$ and $\tau$ belongs to
	$ L^\infty(0,T^*,\BB_{p, 1}^{\frac{\dd}{p}},)$, the solution can be extended in time with a life span larger than $T^*$.
\end{theorem}
\noindent The construction of global-in-time classical solution is more delicate than the above local result. Indeed, the lacking of damping term for the conformation tensor $\tau$ does not allow to use the classic fixed point approach, which couples the Picard scheme with standard estimates for the Stokes operator. We hence proceed as follows: 
\begin{itemize}
	\item  we determine a suitable functional setting for the initial data for which system \eqref{main_system} preserves the smallness condition of the initial data,
	\item  we hence use this specific small condition, coupled with the Picard fixed point, to show that certain critical regularities 
			are still propagated, globally in time.
\end{itemize}
We will see that the mentioned functional framework corresponds to the Lorentz spaces $\uu_0\in L^{\dd, \infty}(\RR^\dd)$ and  $\tau_0\in L^{\dd, \infty}(\RR^\dd)$. We mention that these spaces are critical under the considered scaling behavior.
We will thus prove the following global result:
\begin{theorem}\label{main_thm2}
Let $\mu=0$ and assume that the dimension $\dd \geq 3$, the initial data $\uu_0$ belongs to $\BB_{p,1}^{\dd/p-1}\cap L^{\dd, \infty}(\RR^\dd)$ while $\tau_0$ belongs to $\BB_{p,1}^{\dd/p}\cap L^{\frac{\dd}{2}, \infty}(\RR^\dd)$, with $p\in [1, +\infty)$. Then there exists a small positive constant $\ee$ depending on the dimension $\dd$ such that, whenever the following smallness condition holds true
	\begin{equation*}
	\begin{alignedat}{8}
		&\left\|\,\uu_0	\,\right\|_{L^{\dd, \infty}}\,+\,
		\frac{1}{\nu}
		\left\|\,\tau_0	\,\right\|_{L^{\frac{\dd}{2}, \infty}}
		\,&&\leq\,
		\frac{\ee}{\nu}
	\end{alignedat}
	\end{equation*}
	then the corotational Johnson-Segalman model \eqref{main_system} admits a global-in-time classical solution $(\uu,\, \tau)$, satisfying
	\begin{equation*}
	\begin{alignedat}{4}
		\uu		\,	&\in	\,&& \mathcal{C}\big(\,\RR_+, \,\BB_{p, 1}^{\frac{\dd}{p}-1}\cap  L^{\dd, \infty}(\RR^\dd)\,\big)\cap L^1_{\rm loc}(\RR_+,\BB_{p, 1}^{\frac{\dd}{p}+1}),\quad\quad
		\tau	\,	\in	\, \mathcal{C}(\,\RR_+, \,\BB_{p, 1}^{\frac{\dd}{p}}\cap L^{\frac{\dd}{2}, \infty}(\RR^\dd)\,).
	\end{alignedat}
	\end{equation*}
	If $p$ belongs to $[1, 2\dd]$ then the solution is unique. Furthermore, there exists a constant $C$ depending just on the dimension $\dd$ such that for any time $T\geq 0$ we have
	\begin{equation*}
	\begin{aligned}
			\| \tau(T)\|_{L^{\frac{\dd}{2}, \infty}} &\leq \|\,\tau_0\,\|_{L^{\frac{\dd}{2}, \infty}},
			\quad\quad 
			\|\,\uu(T)\,\|_{L^{\dd, \infty}}
			\,\leq\,	
			C\left(\, \|\,\uu_0\,\|_{L^{\dd, \infty}} \,+\, \nu^{-1}\|\,\tau_0\,\|_{L^{\frac{\dd}{2}, \infty}}\right),\\
			\|\,\uu\,\|_{L^\infty(0,T;\BB_{p,1}^{\frac{\dd}{p}-1})}
			\,+\,
			\nu
			\|\,\uu\,\|_{L^1(0,T;\BB_{p,1}^{\frac{\dd}{p}+1})}
			\,&\leq\,
			\Big\{
	\|\,\uu_0\,\|_{\BB_{p,1}^{\frac{\dd}{p}-1}}
	\,+\,
		CT
	\|\,\tau_0\,\|_{\BB_{p,1}^{\frac{\dd}{p}}}
	{\rm e}^{
	CT^*
	\Theta_\nu(\uu_0,\,\tau_0,\,T)
	}
	\Big\}
	{\rm e}^{C\nu^{-1}\Theta_\nu(\uu_0,\,\tau_0,\,T)},\\
	\|\,\tau\,\|_{L^\infty(0,T;\BB_{p,1}^\frac{\dd}{p})}
	\,&\leq\,
	\|\,\tau_0\,\|_{\BB_{p,1}^{\frac{\dd}{p}}}
	\exp
	\left\{
	CT
	\Theta_\nu(\uu_0,\,\tau_0,\,T)
	\right\},
	\end{aligned}
	\end{equation*}
	where the function $\Theta_\nu(\uu_0,\,\tau_0,\,T)$ is a smooth function depending on the time $T>0$ and on the norms of the initial data $(\uu_0,\,\tau_0)$ 
	in the functional framework $\BB_{\infty, 1}^{-1}\times \BB_{\infty, 1}^0$:
	\begin{equation*}
	\Theta_\nu(\uu_0,\,\tau_0,\,T):=
	C\|\,\uu_0\,\|_{\BB_{\infty, 1}^{-1}}
	\exp\left\{	CT\nu^{-1}\|\,\tau_0\,\|_{\BB_{\infty, 1}^0}	\right\}
	\,+\,
	\nu 
	\left(
		\exp\left\{	CT\nu^{-1}\|\,\tau_0\,\|_{\BB_{\infty, 1}^0}	\right\}
		\,-\,1
	\right).
\end{equation*}
\end{theorem}

\noindent The paper is structured as follows. In Section \ref{sec:Besov} we present the main functional settings in which we develop our main results. Section \ref{sec:tensor} is devoted to suitable inequalities related to the tensor equation $\tau$, that will play a major role in the mains proofs. Section \ref{sec:bidim} is devoted to the proof of Theorem \ref{main_thm0}, about the existence and uniqueness of global-in-time strong solutions for large initial data in dimension two. Section \ref{sec:non-cor} addresses the case of a positive $\mu>0$. Section \ref{sec:loc-sol}, Section \ref{sec:globl-sol} and Section \ref{sec:globl-sol2}  are devoted to Theorem \ref{main_thm1} and Theorem \ref{main_thm2}, respectively, namely to the existence and uniqueness of strong solutions in dimension $\dd\geq 3$.


\subsection{\bf Physical aspects of the corotational approximation}$\,$

\medskip
\noindent 
Before developing the proof for Theorems \ref{main_thm0}, \ref{main_thm1} and \ref{main_thm2}, we briefly overview certain physical aspects of our model.

\smallskip
\noindent
The considered equations \eqref{main_system} can be derived as a moment-closure of a suitable multiscale model for the evolution of viscoelastic fluids (cf. \cite{DuLiuYu,Masmoudi,LionsMasmoudi}). This multiscale approach couples the Navier-Stokes equation for the flow $u = u(t,x)$ on the macroscopic level, together with a Fokker-Plank equation describing the evolution of the probability density function $f = f(t,x,Q)$ for the molecular orientation $Q\in \mathbb{R}^d$ on the microscopic level:
\begin{equation*}
\left\{
	\begin{alignedat}{4}
		\partial_t \uu + \uu\cdot \nabla \uu - \nu \Delta \uu +\nabla \pre &=  \Div\, \tau 
		\qquad &&\mathbb{R}_+\times\mathbb{R}^d,\\
		\Div\,\uu &= 0\qquad &&\mathbb{R}_+\times\mathbb{R}^d,\\
		\partial_t f + \uu\cdot\nabla_x f + \Div_Q(\omega Q) &= \frac{2}{\zeta}\Delta_Q f +
		\frac{2}{\zeta}\Div_Q(f\nabla_Q\Psi(Q))\qquad &&\mathbb{R}_+\times\mathbb{R}^d\times\mathbb{R}^d.
	\end{alignedat}
\right.
\end{equation*}
In this framework, $\zeta$ is a friction coefficient, $\tau= \tau(t,x)$ represents the polymeric contribution to the stress
\begin{equation*}
	\tau(t,x): = \lambda \int_{\RR^\dd} \nabla_Q\Psi(Q) \otimes Q f(t,x,Q)\dd Q,
\end{equation*}
while $\Psi(Q)$ stands for a suitable potential depending on the molecular orientation. The considered Fokker-Plank equation retains a corotational simplification at the microscopic level. 

\smallskip
\noindent One of the simplest molecular potential is given by the Hookean law $\Psi(Q)=H |Q|^2/2$, where $H$ is the elasticity constant. This particular formulation allows to determine a constitutive equation for the tensor $\tau$. Indeed, using a moment-closure approximation, we can multiply the Fokker-Plank equation by $Q\otimes Q$ and perform integration by parts along $Q\in \mathbb{R}^\dd$, to gather an equation for $\tau$:
$$\partial_t \tau + \uu\cdot \nabla \tau - \omega \tau +\tau \omega = -\frac{2}{\zeta}\tau.$$
This is exactly our model that we tudy in \eqref{main_system} with $\mu=0$ and $a=\frac{2}{\zeta}\geq 0$.

\smallskip
\noindent 
In addition, the more physical scenario given by $\mu>0$ can be derived by taking into account the complete Fokker-Plank equation
$$\partial_t f + \uu\cdot\nabla_x f + \Div_Q(\nabla u Q) = \frac{2}{\zeta}\Delta_Q f +
		\frac{2}{\zeta}\Div_Q(f\nabla_Q\Psi(Q)),$$
without corotational approximation. Indeed performing the moment-closure for the equation of $\tau$, one has
$$\partial_t \tau + \uu\cdot \nabla \tau - \nabla \uu \tau +\tau\,\nabla \uu^T = -\frac{2}{\zeta}\tau.$$
The model \eqref{main_system} with $\mu>0$ is then recovered, recasting the tensor $\tau$ as $\tau - \mu \Id$ and finally re-introducing the corotational approximation for the stretching terms.

\smallskip
\noindent
Therefore, we remark that the different levels in which we introduce the corotational approximation, i.e. the microscopic or the macroscopic levels, lead to the different cases of \eqref{main_system} given by $\mu=0$ or $\mu>0$. We are aware that considering these approximations decreases the physical applications of our model. Nonetheless, up to our knowledge, the existence of global-in-time classical solutions in the general case ($\mu\geq 0$) is still an open question. In this manuscript, we provide a positive answer in the case of the two dimensional corotational Oldroyd-B model with $\mu=0$. On the other-hand, we further address the case of $\mu>0$, refining the lifespan of classical solutions within the functional framework of Theorem \ref{main_thm0}.

\section{Functional spaces and toolbox of harmonic analysis}\label{sec:Besov}

\noindent We begin with recalling the definition of weak Lebesgue spaces $L^{p,\infty}(\RR^\dd)$.
\begin{definition}
	For any $p\in [1,\infty)$,
 the functional space $L^{p,\infty}(\RR^\dd)$ is composed by Lebesgue measurable function, for which the following norm is bounded:
 \begin{equation*}
 	\|\,f\,\|_{L^{p,\infty}}
 	\,=\,
 	\sup_{\lambda>0}\lambda\, \m\left(\left\{ x\in\RR^\dd	\quad\text{such that}\quad |f(x)|>\lambda	\right\}\right)^{\frac{1}{p}}
 	\,<\,\infty
 \end{equation*}
 where $\m$ stands for the Lebesgue measure on $\RR^\dd$.
\end{definition}

\begin{remark}\label{rmk:Lorentz_space}
	Whenever $1<p<\infty$, the functional space $L^{p, \infty} $ coincides with the Lorentz space defined by real interpolation by means of
	\begin{equation*}
		L^{p, \infty}\,=\,(\,L^1,\,L^\infty\,)_{\frac{1}{p},\,\infty}.
	\end{equation*}
	More precisely, any function $f$ of $L^{p, \infty} $ can be decomposed as $f= f_A\,+\,f^A$, for any positive real $A$, where
	\begin{equation*}
		\|\,f_A\,\|_{L^1(\RR^\dd)}\leq C\,A^{1-\frac{1}{p}}\quad\quad \|\,f^A\,\|_{L^\infty(\RR^\dd)}\,\leq\,C\,A^{-\frac{1}{p}}.
	\end{equation*}
	The optimal constant $C$ of the above inequality is an equivalent norm for the quantity defined above.
	
	\noindent In general, for any $1\leq p,\,q\leq \infty$ the Lorentz space $L^{p,q}$ can be defined by real interpolation as follows:
	\begin{equation*}
		L^{p,\,q}\,=\,(\,L^1,\,L^\infty\,)_{\frac{1}{p},\,q}.
	\end{equation*}

\end{remark}

\noindent We now briefly recall the definition of the Littlewood -- Paley decomposition as well as of the Besov spaces. 

\smallskip\noindent
The Littlewood -- Paley theory is defined making use of the so called dyadic partition of unity:  let $\chi\,=\,\chi(\xi)$ be a radial function depending on the frequencies $\xi\in\RR^\dd$ of class $\mathcal{C}^\infty(\RR^\dd)$ whose support is included in the ball $\{\xi\in \RR^\dd_\xi\,:\,|\xi|\leq 4/3\,\}$. We assume that $\chi$ is identically $1$ in the ball $\{\xi\in \RR^\dd\,:\,|\xi|\leq 3/4\,\}$ while the function $ r\rightarrow \chi(r\xi)$ is decreasing. We then denote
$\varphi(\xi)=\chi(\xi/2)-\chi(\xi)$, so that
\begin{equation}\label{LP-identiy1}
	\forall \xi\in\RR^\dd\setminus\{0\},\,\sum_{q\in\ZZ} \varphi\left(2^{-q}\xi\right)\,=\,1\quad\quad\text{and}\quad\quad
	\chi(\xi)\,=\,1\,-\,\sum_{q\in\NN}\varphi(2^{-q}\xi).
\end{equation}
We then define the homogeneous dyadic block $\Dd_q$ and the operator $\Sd_q$ localizing the frequencies $\xi$ as follows:
\begin{equation*}
	\Dd_q f\,=\,\mathfrak{F}^{-1}(\,\varphi(2^{-q}\xi)\hat f(\xi)\,),
	\quad
	\quad
	\quad
	\Sd_q f\,=\,\mathfrak{F}^{-1}(\chi(2^{-q}\xi)\hat f(\xi))
\end{equation*}
where $\mathfrak{F}$ stands for the standard Fourier transform. We remark that for any tempered distribution $\uu\in\mathcal{S}'(\RR^\dd)$, the functions $\Dd_q \uu$ and $\Sd_q \uu$ are analytic. Furthermore, if there exists a real $s$ for which $\uu\in H^s(\RR^\dd)$, then both  $\Dd_q \uu$ and $\Sd_q \uu$ belong to the space $H^\infty(\RR^\dd)\,=\,\cap_{\sigma\in \RR} H^\sigma(\RR^\dd)$.
 
\smallskip
\noindent We state that thanks to \eqref{LP-identiy1} the identity $\uu\,=\,\Sd_0 \uu\,+\,\sum_{q\in\NN}\Dd_q\uu$ holds in $\mathcal{S}'(\RR^\dd)$ while 
$\uu \,=\,\sum_{q\in \ZZ}\Dd_q \uu$ for any homogeneous temperate distribution $\uu\in \mathcal{S}_h'(\RR^\dd)$.

\smallskip
\noindent
We will frequently use the following orthogonal condition on the dyadic blocks $\Dd_q$:
\begin{equation*}
	\Dd_q \Dd_j \,\equiv\,0\quad\text{ if }\quad |\,q-\,j\,|\,\geq\,2
	\quad\text{ and }\quad \Dd_k\left(\Sd_{q-1}\uu \Dd_q v\,\right)\equiv 0
	\quad\text{ if }\quad |\,q-j\,|\geq 5.
\end{equation*}

\noindent We can now define the functional set of the homogeneous Besov space as follows:
\begin{definition}\label{Besov-def}
	Let $s\in\RR$, $(p,\,r)\in [1,\,\infty]^2$ and $\uu\in \mathcal{S}'(\RR^\dd)$. We denote by
	\begin{equation*}
	\|\,\uu\,\|_{\BB_{p,r}^s}\,:=\,
	\begin{cases}
		\left(
			\,
			\sum_{q\in\ZZ}2^{rqs} \|\,\Dd_q\uu\,\|_{L^p}^r
			\,
		\right)^\frac{1}{r}
		\hspace{1cm}	&
		\text{ if }
		r<+\infty,\\
		\;\sup_{q\in\ZZ} 2^{qs}\;\|\,\Dd_q\uu\,\|_{L^p}
		\hspace{1cm}	&
		\text{ if }
		r\,=\,+\infty.
	\end{cases}
	\end{equation*}
	Thus, we define the homogenous Besov space $\BB_{p,r}^s\,=\,\BB_{p,r}^s(\RR^\dd)$ by
	\begin{equation*}
		\BB_{p,r}^s\,:=\,
		\left\{
			\uu\,\in\,\mathcal{S}'(\RR^\dd)\,|\,\|\,\uu\,\|_{\BB_{p,r}^s}<+\infty
		\right\}
	\end{equation*}
	if $s<\dd/p$ or $s=\dd/p$ with $r=1$, and by
	\begin{equation*}
		\BB_{p,r}^s\,:=\,
		\left\{
			\uu\,\in\,\mathcal{S}'(\RR^\dd)\,|\quad\forall |\alpha|=k+1\quad\|\,\partial^\alpha \uu\,\|_{\BB_{p,r}^{s-k-1}}<+\infty
		\right\}
	\end{equation*}
	if $\dd/p+k\leq s <\dd/p+k+1$ or $s=\dd/p+k+1$ and $r=1$, for some $k\in\NN$.
\end{definition}

\begin{remark}
	The functional space $\BB_{p,r}^s$ is a Banach space if and only if $s<\dd/p$ or $s=\dd/p$ and $r=1$.
\end{remark}

\noindent
For the sake of completeness we recall also the definition of the non-homogeneous Besov spaces:
\begin{definition}
	Let $s\in\RR$, $(p,\,r)\in [1,\,\infty]^2$ and $\uu\in \mathcal{S}'(\RR^\dd)$. We denote by
	\begin{equation*}
	\|\,\uu\,\|_{\BB_{p,r}^s}\,:=\,
	\begin{cases}
		\left(
			\|\,\Sd_0\uu\,\|_{L^p}^r
			\,+\,
			\sum_{q\in\NN}2^{rqs} \|\,\Dd_q\uu\,\|_{L^p}^r
			\,
		\right)^\frac{1}{r}
		\hspace{1cm}	&
		\text{ if }
		r<+\infty,\\
		\max\left\{
		\|\,\Sd_0\uu\,\|_{L^p},
		\;\sup_{q\in\NN} 2^{qs}\;\|\,\Dd_q\uu\,\|_{L^p}\,
		\right\}
		\hspace{1cm}	&
		\text{ if }
		r\,=\,+\infty.
	\end{cases}
	\end{equation*}
	The non-homogeneous Besov space $B_{p,r}^s\,=\,B_{p,r}^s(\RR^\dd)$ is the set of temperate distributions for which $\|\,\uu\,\|_{B_{p,r}^s}$ is finite.
\end{definition}

\begin{remark}
	The Besov spaces $\BB_{2,2}^s$ and $B_{2,2}^s$ coincide with the Sobolev spaces $\Hh^s(\RR^\dd)$ and $H^s(\RR^\dd)$ respectively. Furthermore, if $s\in \RR_+\setminus \NN$, the Besov spaces $\BB_{\infty, \infty}^s$ and $B_{\infty, \infty}^s$ coincide with the H{\"o}lder spaces $\dot C^s$ and $C^s$.
\end{remark}

\noindent
The following estimates are known as Bernstein type inequalities, and they will be frequently used in our proofs.

\begin{lemma}
	Let $1\leq p\leq l\leq \infty$ and $\psi\in \mathcal{C}^\infty_c(\RR^\dd)$. We hence have
	\begin{equation*}
		c 2^{-q\left(\frac \dd l\,-\,\frac \dd p \right)}\|\,\Dd_q \uu \,\|_{L^{l}}
		\,\leq\,\|\,\Dd_q \uu \,\|_{L^{p}}
		\,\leq\,
		\,C 2^{q\left(\frac \dd l\,-\,\frac \dd p \right)}\|\,\Dd_q \uu \,\|_{L^{l}}
	\end{equation*}
	and
	\begin{equation*}
		\|\,\Sd_q \uu \,\|_{L^{p}}
		\,\leq\,
		\,C 2^{q\left(\frac \dd l\,-\,\frac \dd p \right)}\|\,\Sd_q \uu \,\|_{L^{l}}.
	\end{equation*}
\end{lemma}

\noindent
As a consequence of the Bernstein type inequality and the definition of Besov Spaces $\BB_{p,r}^s$, we have the following proposition:
\begin{prop}
	$(i)$ There exists a constant $c>0$ such that
	\begin{equation*}
		\frac 1 c \|\,\uu\,\|_{\BB_{p,r}^s}
		\,\leq\,
		\|\,\nabla \uu\,\|_{\BB_{p,r}^{s-1}}
		\,\leq\,
		c  \|\,\uu\,\|_{\BB_{p,r}^s}.
	\end{equation*}
	$(ii)$ 		For $1\,\leq\,p_1\,\leq p_2\leq \infty$ and $1\leq r_1\,\leq r_2 \leq \infty$, 
				we gather $\BB_{p_1, r_1}^s\hookrightarrow \BB_{p_2, r_2}^{s-\dd\left({1}/{p_1}\,-\,{1}/{p_2}\right)}$.
	
	\noindent
	$(iii)$ 	If $p\in [1,\infty]$ then $\BB_{p,1}^{\dd/p}\hookrightarrow \BB_{p,\infty}^{\dd/p}\cap L^\infty$. Furthermore, for any $p\in [1, \infty]$, 
				$\BB_{p,1}^{\dd/p}$ is an algebra embedded in $L^\infty(\RR^\dd)$.
				
	\noindent
	$(iv)$ 		The real interpolation $(\BB_{p,r}^{s_1},\,\BB_{p,r}^{s_2})_{\theta, \tilde r}$, for a parameter $\vartheta\in (0,1)$, is isomorphic to 
				$\BB_{p,\tilde r}^{\vartheta s_1+(1-\vartheta)s_2}$.
\end{prop}
\noindent
We recall further the following results about inclusions between Lorentz and Besov spaces.
\begin{lemma}
	For any $1< p<q\leq \infty$, we have
	\begin{equation*}
		L^{p,\,\infty}	\hookrightarrow \BB_{q, \,\infty}^{\frac{\dd}{q}\,-\,\frac{\dd}{p}}.
	\end{equation*}
\end{lemma}
\begin{proof}
	Denoting by $h_j\,=\,2^{jN}h(2^j\cdot)$ with $h\,=\,\mathfrak{F}^{-1}\varphi$, we recast that the dyadic block $\Dd_j$ as a convolution operator
	\begin{equation*}
		\Dd_j \uu\,=\,h_j* \uu.
	\end{equation*}
	Hence, making use of the following convolution inequalities between Lorentz spaces
	\begin{equation*}
		\|\,\Dd_j\uu\,\|_{L^q}\,\leq \|\,h_j\,\|_{L^{r,1}}\|\,\uu\,\|_{L^{p,\infty}}
		\quad\text{with}\quad
		\frac{1}{p}\,+\,\frac{1}{r}\,=\,1+\frac{1}{q}
	\end{equation*}
	and observing that by change of variables
	\begin{equation*}
		\|\,h_j\,\|_{L^{r,1}}\,=\,2^{j\dd\left(1-\frac{1}{r}\right)}\|\,h\,\|_{L^r},
	\end{equation*}
	we eventually gather that
	\begin{equation*}
		\sup_{j\in \ZZ}2^{j\left(\frac{\dd}{q}-\frac{\dd}{p}\right)}\|\,\Dd_j\uu\,\|_{L^q}
		\,\leq\,\|\,h\,\|_{L^{1,r}}\|\,\uu\,\|_{L^{p,\infty}}
	\end{equation*}
	which concludes the proof of the lemma.
\end{proof}

\noindent We now consider several a-priori estimates in the functional framework of Besov spaces for the heat semigroup (cf. \cite{BCD}, Lemma $2.4$).
\begin{lemma}
	There exists two constant $c$ and $C$ such that for any $\tau\geq 0$, $q\in\ZZ$ and $p\in [1,\infty]$, we get
	\begin{equation*}
		\left\|\,e^{\tau\Delta}\Dd_q \uu\,\right\|_{L^p}\,\leq \,Ce^{-c\tau 2^{2q}} \left\|\,\Dd_q\uu\,\right\|_{L^p}.
	\end{equation*}	 
\end{lemma}
\noindent
From the above Lemma we can then deduce the following result (cf. \cite{PD}, Proposition $3.11$)
\begin{prop}	\label{prop-heat-eq}
Let $s\in \RR$, $1\leq p,\,r,\,\rho_1\,\leq\, \infty$. Let $\uu_0$ be in $\BB_{p,r}^s$ and $f$ be in $\tilde L^{\rho_1}(0,T;\BB_{p,r}^{s-2+2/\rho_1})$ for some positive time $T$ (possibly $T=\infty$). Then the heat equation
	\begin{equation*}
	\begin{cases}
		\partial_t \uu\, -\,\nu\Delta \uu\,=\, f &\hspace{2cm} [0,T)\times \RR^\dd, \\
		\uu_{|t=0}=\uu_0					 &\hspace{3.365cm} \RR^\dd,			
	\end{cases}
	\end{equation*}
	admits an unique strong solution in $\tilde L^\infty(0,T;\BB_{p,r}^s)\cap \tilde L^{\rho_1}(0,T;\BB_{p,r}^{s-2+2/\rho_1})$. Moreover there exist a constant 
	$C$ depending just on the dimension $\dd$ such that the following estimate holds true for any time $t\in [0,T]$ and $\rho\geq \rho_1$:
	\begin{equation}\label{heat-bound-sol}
		\nu^{\frac{1}{\rho}}\|\,\uu\,\|_{\tilde L^\rho(0,\,t, \BB_{p,r}^{s+\frac{2}{\rho}})}
		\,\leq\,
		C
		\left(
			\|\,\uu_0\,\|_{\BB_{p,r}^s}\,+\,\nu^{\frac{1}{\rho_1}-1}\|\,f\,\|_{\tilde L^{\rho_1}(\RR_+,\,\BB_{p,r}^{s-2+\frac{2}{\rho_1}})}
		\right).
	\end{equation}
\end{prop}

\noindent In the previous Proposition we introduce the functional space $\tilde L^\rho(0,T;\BB_{p,r}^s)$, which is known as a Chemin-Lerner space. This is defined similarly as in Definition \eqref{Besov-def}, imposing 
\begin{equation*}
	\|\,\uu\,\|_{\tilde L^\rho(0,T;\BB_{p,r}^s)} \,:=
	\left\|\,\left(2^{qs} \|\,\Dd_q \uu\,\|_{L^\rho(0,T;L^p_x)}\right)_{q\in\ZZ}\,\right\|_{\ell^r(\ZZ)}.
\end{equation*}
We remark that thanks to the Minkowski inequality
\begin{equation*}
	\|\,\uu\,\|_{\tilde  L^\rho(0,T;\BB_{p,r}^s)}
	\,\leq\,
	C\,\|\,\uu\,\|_{		 L^\rho(0,T;\BB_{p,r}^s)}\quad\quad \text{for}\quad \quad \rho \geq r,
\end{equation*}
while the opposite inequality holds when $\rho \leq r$. We hence denote
\begin{equation*}
	\tilde{\mathcal{C}}([0,T], \BB_{p,r}^s) \,:=\,\tilde L^\infty(0,T;\BB_{p,r}^s)\cap \mathcal{C}([0,T], \BB_{p,r}^s)
	\quad\quad\text{and by}\quad\quad
	\tilde L^\rho_{\rm loc}(\RR_+, \BB_{p,r}^s)\,:=\,\cap_{T>0} \tilde L^{\rho}(0,T;\BB_{p,r}^s).
\end{equation*}
Similarly, one can define the non-homogeneous Chemin-Lerner spaces $\tilde L^\rho(0,T;B_{p,r}^s)$. In the particular case of $p=r=2$ we will use the notation $\tilde L^\rho(0,T;\dot H^s)$ or $\tilde L^\rho(0,T; H^s)$.

\begin{remark}\label{rmk:Stokes-bound}
	Thanks to Proposition \ref{prop-heat-eq}, and using the fact that the projector $\Pp$ on the free divergence vector fields is an homogeneous Fourier 
	mutiplier of degree $0$, namely it is continuous from $\BB_{p,r}^s$ to itself, we can easily solve the non-stationary Stokes problem
	\begin{equation*}
		\begin{cases}
			\,\partial_t \uu\, -\,\nu\Delta \uu\,+\,\nabla \pre\,=\, f 	&\hspace{2cm} [0,T)\times \RR^\dd, \\
			\,\Div\,\uu\,=\,0												&\hspace{2cm} [0,T)\times \RR^\dd, \\
			\,\uu_{|t=0}=\uu_0					 						&\hspace{3.365cm} \RR^\dd,			
		\end{cases}
	\end{equation*}
	with initial data $\uu_0\in\BB_{p,r}^s$ with null divergence and a source term $f$ in $\tilde L^1(0,T; \BB_{p,r}^s)$. We hence achieve a unique solution 
	$\uu$ in the class affinity
	\begin{equation*}
		\uu \in  \tilde L^\infty(0,T;\BB_{p,r}^s)\cap \tilde L^1(0,T;\BB_{p,r}^{s+2}),
		\quad\quad
		\nabla \pre \in \tilde L^1( 0,T;\BB_{p,r}^{s}),
	\end{equation*}
	with $\uu$ satisfying
	\begin{equation*}
		\nu^{\frac{1}{\rho}}\|\,\uu\,\|_{\tilde L^\rho(0,T;\BB_{p,r}^{s+\frac{2}{\rho}})}
		\,\leq\,
		C
		\left(
			\|\,\uu_0\,\|_{\BB_{p,r}^s}
			\,+\,
			\|\,\Pp f\,\|_{L^1(0,T;\BB_{p,r}^s)}
		\right).
	\end{equation*}
	Furthermore, if $r<\infty$ the solution $\uu$ belongs to $\mathcal{C}([0,T];\BB_{p,r}^s)$.
\end{remark}

\begin{remark}
	We can introduce also the non-homogeneous version of the Proposition \ref{prop-heat-eq}, for some initial data $\uu_0$ in $B_{p,r}^s$ and 
	$f$ in $\tilde L^{\rho_1}(0,T;B_{p,r}^{s-2+2/\rho_1})$.  The existence and uniqueness of a solution still holds, nevertheless for a constant $C$ in 
	\eqref{heat-bound-sol} that this time depends linearly on the time $T$.
\end{remark}

\noindent The proof of a-priori estimates for certain nonlinear terms is mainly handled through the use of the paradifferential calculus, in particular of the so called Bony type decomposition:
\begin{equation}\label{Bony-dec}
	f\,g\,= \dot T_f g\,+\,\dot T_g f\,+\,\dot R(f,\,g),
\end{equation}
where the paraproduct $\dot T$ and the homogenous reminder $\dot R$ are defined by
\begin{equation*}
	\dot T_f g\,:=\,\sum_{q\in \ZZ}\Sd_{q-1}g \Dd_q f\quad\quad\text{and}\quad\quad
	\dot R(f,\,g)\,=\,\sum_{q\in\ZZ}\Dd_q f \Big(\sum_{|j-q|\leq 1}\Dd_j g\Big).
\end{equation*}
We then state some results of continuity of these operators that we will often use in our proof.
\begin{prop}\label{prop:reminder-und-paraproduct}
	Let $1\leq p,\,p_1,\,p_2,\,r,\,r_1,\,r_2\,\leq \infty$ satisfying $1/p\,=\,1/p_1+1/p_2$ and $1/r = 1/r_1\,+\,1/r_2$. The the homogeneous paraproduct $\dot T$ 
	is continuous 
	\begin{itemize}
		\item from $L^\infty \times \BB_{p,r}^t$ into $\BB_{p,r}^t$ for any real $t\in \RR$,
		\item from $\BB_{p_1,r_1}^{-s}\times \BB_{p_2,r_2}^s$ into $\BB_{p,r}^{t-s}$, for any $t\in \RR$ and $s>0$.
	\end{itemize}
	The homogeneous reminder $R$ is continuous
	\begin{itemize}
		\item from $\BB_{p_1,r_1}^s\times \BB_{p_2,r_2}^t$ into $\BB_{p,r}^{s+t}$ for any $(s,t)\in \RR^2$, such that $s+t>0$,
		\item from $\BB_{p_1,r_1}^s\times \BB_{p_2,r_2}^{-s}$ into $\BB_{p,\infty}^0$ if $s\in \RR$ and $1/r_1\,+\,1/r_2 \geq 1$.
	\end{itemize}
\end{prop} 
\noindent
The above proposition allows to determine almost any continuity results for the product of two distributions that belong to two Besov spaces.  
Further extensions of the above result can be achieved assuming some additional regularity of the distributions:
\begin{lemma}\label{lemma-product-f^2}
	Let $f$ be a function in $L^2(\RR^2)\cap  \BB_{\infty, 1}^1$, then $f^2$ belongs to $  \BB_{\infty, 1}^0$ and satisfies 
	\begin{equation*}
		\|\,f^2 \,\|_{ \BB_{\infty, 1}^0}
		\,\leq\,
		C
		\|\,f \,\|_{L^2(\RR^2)}
		\|\,f \,\|_{\BB_{\infty, 1}^1}
	\end{equation*}
\end{lemma}
\begin{proof}
	We begin with localizing the frequencies  of $f^2$ through the standard Bony decomposition \eqref{Bony-dec}
	\begin{equation*}
		f^2 \,=\, 2\dot T_f f\, +\,\dot R(f, f).
	\end{equation*}
	Thus, the triangular inequality implies that
	\begin{equation*}
	\begin{aligned}
		\|\, f^2\,\|_{\BB_{\infty, 1}^0}\, 
		&\leq \,
		2\|\,\dot T_f f\,\|_{\BB_{\infty, 1}^0}\,+\,
		\|\,\dot R(f, f)\,\|_{\BB_{\infty, 1}^0},\\
		&\leq
		2\sum_{q\in \ZZ}
		\|\, \Dd_q \dot T_f f \,\|_{L^\infty_x}
		\,+\,
		\|\, \Dd_q \dot R(f, f) \,\|_{L^\infty_x}.
	\end{aligned}
	\end{equation*}
	We first remark that for any integer $q\in \ZZ$ 
	\begin{equation*}
	\begin{aligned}
		\|\,\Dd_q \dot T_f f\,\|_{L^\infty} \,&\leq \,
		\sum_{| q-j|\leq 5}
		\| \,\dot S_{j-1} f\,\|_{L^\infty_x}
		\| \,\Dd_{j} f\,\|_{L^\infty_x}\\
		&\leq\,
		\sum_{| q-j|\leq 5}
		\| \,\dot S_{j-1} f\,\|_{L^2_x}
		2^{j}
		\| \,\Dd_j f\,\|_{L^\infty_x}
		\,\leq\,
		\| \, f\,\|_{L^2_x}
		\sum_{| q-j|\leq 5}
		2^{j}
		\| \,\Dd_j f\,\|_{L^\infty_x},
	\end{aligned}
	\end{equation*}
	hence, the homogeneous paraproduct is bounded by
	\begin{equation*}
		\|\,\dot T_f f\,\|_{\BB_{\infty, 1}^0}
		\,\leq \,
		\|\,f \,\|_{ L^2_x}
		\|\,f \,\|_{ \BB_{\infty, 1}^1}.
	\end{equation*}
	We now take into account the homogeneous reminder. By definition, we gather that
	\begin{equation*}
	\begin{aligned}
			\|\,\Dd_q \dot R(f, f) \,\|_{L^\infty_x}
			\,&	\leq \,
			\sum_{\substack{j-q \geq-5\\ |\nu| \leq 1}}  
			\|\, \Dd_q (	 \Dd_{j+\nu } f \Dd_j f)\,\|_{L^\infty_x}
			\,\leq \,
			\sum_{\substack{j \geq q-5\\ |\nu| \leq 1}}  
			2^q\|\, \Dd_q (	 \Dd_{j+\nu } f \Dd_j f)\,\|_{L^2_x}\\
			&\leq\,
			\sum_{\substack{j \geq q-5\\ |\nu| \leq 1}}  
			2^{q-j}\|\, \Dd_{ j + \nu } f \|_{L^2_x} 2^j \|\, \Dd_j f \,\|_{L^\infty_x}
			\,\leq\,
			\|\,f\, \|_{L^2_x}
			\sum_{\substack{j \in\ZZ}}  
			2^{q-j}{\bf 1}_{(-\infty, 5]}(q-j) 2^j \|\, \Dd_j f \,\|_{L^\infty_x}.
	\end{aligned}
	\end{equation*}
	Defining $a_j =2^{j}{\bf 1}_{(-\infty, 5]}(j)$ for any $j\in \ZZ$, we can recast the last term in convolution form, namely
	\begin{equation*}
			\sum_{\substack{j \in\ZZ}}  
			2^{q-j}{\bf 1}_{(-\infty, 5]}(q-j) 2^j \|\, \Dd_j f \,\|_{L^\infty_x}
			\,=\, 
			\left(
				 (a_j)_{j\in\ZZ}* (2^j \|\, \Dd_j f \,\|_{L^\infty_x}
			\right)_q,
	\end{equation*}
	for any $q\in\ZZ$. Hence, applying the Young inequality we deduce that
	\begin{equation*}
			\left\| \, (a_j)_{j\in\ZZ}* (2^j \|\, \Dd_j f \,\|_{L^\infty_x})_{j\in\ZZ}
			\right\|_{\ell^1}
			\,\leq\,
			\left\| 
			 (2^j\|\, \Dd_j f \,\|_{L^\infty_x})_{j\in\ZZ}
			\right\|_{\ell^1}
			\,=\, \|\,f \,\|_{\BB_{\infty, 1}^1},
	\end{equation*}
	from which
	\begin{equation*}
			\sum_{\substack{q \in\ZZ}}  
			\|\,\Dd_q \dot R(f, f) \,\|_{L^\infty_x}
			\,\leq\,
			\|\,f \,\|_{L^2_x}
			\|\,f \,\|_{\BB_{\infty, 1}^1},
	\end{equation*}
	which concludes the proof of the Lemma.
\end{proof}

\subsection{\bf Estimates for the conformation tensor}$\,$

\noindent
In this section we perform several a priori estimates for the following equation that governs the evolution of the conformation tensor $\tau(t,x)$:
\begin{equation}\label{a-priori-tau-eq}
	\left\{\hspace{0.2cm}
	\begin{alignedat}{2}
		&\,\partial_t \tau\,+\,\uu\cdot \nabla \tau+a\,\tau 	- \,\omega \tau \,+\,\tau\omega \,=\,f
		\hspace{3cm}
		&&\RR_+\times \RR^\dd, \vspace{0.1cm}\\
		&\,
		\tau_{|t=0}	\,=\,\tau_0			
		&&\hspace{1.02cm} \RR^\dd.\vspace{0.1cm}																							
	\end{alignedat}
	\right.
\end{equation} 
We begin with a standard bound for Lebesgue and Lorentz norms. 
\begin{lemma}\label{lemma:L^pbound_of_tau}
	For any $p\in[1, \infty]$, the following estimate in Lebesgue spaces holds true:
	\begin{equation*}
		\|\,\tau(t)\,\|_{L^p_x}\,\leq\,\|\,\tau_0\,\|_{L^p_x}e^{-at}\,+\,\int_0^te^{a(s-t)} \|\, f(s)\, \|_{L^p_x}\dd s.
	\end{equation*}
	More in general, one has
	\begin{equation*}
		\|\,\tau(t)\,\|_{L^{p, \infty}}\,\leq\,\|\,\tau_0\,\|_{L^{p, \infty}}e^{-at}\,+\,\int_0^te^{a(s-t)}\|\, f(s)\, \|_{L^{p, \infty}}\dd s
	\end{equation*}
	and the inequality reduces to an equality whenever $f$ is identically null.
\end{lemma} 
\begin{proof}
	Considering $p\in [1, \infty)$, we take the matrix inner product between the $\tau$-equation and $\tau |\tau\,|^{p-2}$. Hence, integrating in spatial domain, we first observe that
	\begin{equation*}
		-\int_{\RR^\dd} \omega \tau : \tau |\tau|^{p-2}
		\,+\,
		\int_{\RR^\dd} \tau \omega : \tau |\tau|^{p-2}
		\,=\,
		0
	\end{equation*}
	from which we deduce the following $L^p$-bound of $\tau$
	\begin{equation*}
		\frac{1}{p}\frac{\dd}{\dd t} \|\, \tau(t)\,\|_{L^p_x}^p+a  \|\, \tau(t)\,\|_{L^p_x}^p
		\,\leq \,
		\| \,f \,\|_{L^p}
		\| \,\tau \,\|_{L^p}^{p-1}
		\quad\quad\Rightarrow\quad\quad
		\frac{\dd}{\dd t} \|\, \tau(t)\,\|_{L^p_x}+a \|\, \tau(t)\,\|_{L^p_x}
		\,\leq \,
		\| \,f \,\|_{L^p}.
	\end{equation*}
	The case of $p=+\infty$ can be achieved as the limit case of the previous inequalities.
\end{proof}

\section{Some a-priori estimates for the conformation tensor}\label{sec:tensor}
\noindent In this section we present some a-priori estimates for the $\tau$ equation. We begin with the following lemma about the propagation of Besov regularities.
\begin{lemma}\label{lemma:bound_of_tau-exp_on_nablau}
	Let $(p,\,r)\in [1, \infty]^2$ and $(s,\sigma)\in (-1,1)\in (-1,\infty)$. Assume that $\uu$ is a free divergence vector field whose coefficients belong to $L^1(0, T; \BB_{\infty,1}^{1})$, that the source term $f$ belongs to 
	$\tilde L^1(0, T; \BB_{p,r}^{s}\cap \BB_{p,r}^{\sigma})$ and that the initial data $\tau_0$ is in $\BB_{p,r}^s\cap \BB_{p,r}^\sigma$. Then system \eqref{a-priori-tau-eq} admits a unique solution $\tau$ in the class affinity
	\begin{equation*}
		\tau \,\in\, L^\infty(0,T;\BB_{p,r}^s\cap \BB_{p,r}^\sigma)
	\end{equation*}
	which fulfils the following estimate for any time $t\in [0, T]$
	\begin{equation*}
		\| \, \tau \,\|_{\tilde L^\infty(0,t;\,\BB_{p,r}^s\cap \BB_{p,r}^\sigma)}
		\,\leq\,
		\left(
		\,
			\|\,  \tau_0 \,\|_{\BB_{p,r}^s\cap \BB_{p,r}^\sigma}e^{-at}
			\,+\,
			\| e^{a(\cdot - t)} f(\cdot)\, \|_{\tilde L^1(0,t; \,\BB_{p,r}^s\cap \BB_{p,r}^\sigma)}
		\,\right)
		\exp
		\left\{
		\,
		C
		\int_0^t
		\| \,\nabla \uu \,\|_{\BB_{\infty,1}^{0}}
		\,
		\right\},
	\end{equation*}
	for a suitable positive constant $C>0$.	
\end{lemma}
\noindent
The proof of Lemma \ref{lemma:bound_of_tau-exp_on_nablau} is equivalent to the one of Proposition $4.7$ in \cite{PD}. The presence of an exponential term in the a-priori estimate of Lemma \ref{lemma:bound_of_tau-linear_on_nablau} produces intrinsic difficulties when dealing with the existence of global-in-time solutions. Nevertheless, we can refine such an inequality taking into account Besov spaces with null index of regularity:
\begin{lemma}\label{lemma:bound_of_tau-linear_on_nablau}
	Let $\tau$ be a solution of \eqref{a-priori-tau-eq} in $L^\infty(0, T; \BB_{p,1}^0)$ with $f$ in $ L^1(0, T;\BB_{p,1}^0)$ and also 
	$\nabla \uu$ in $L^1(0, T; \BB_{\infty, 1}^0)$, for some $p\in [1,\infty]$. Then, the following bound holds true for any time $t\in [0, T]$:
	\begin{equation*}
		\|\,\tau\,\|_{L^\infty(0,\, t\,;\, \BB_{p,1}^0)}
		\,\leq\,
		C
		\left(
		\,
			\|\,\tau_0\,\|_{ \BB_{p,1}^0}
			e^{-at}
			\,+\,
			\int_0^t 
			e^{a(s-t)}
			\|\,f(s)\,\|_{\BB_{p, 1}^0}
			\dd s
		\,
		\right)
		\left(
		\,
			1\,+\,\int_0^t \|\,\nabla \uu\,\|_{\BB_{\infty, 1}^0}\dd s
		\,
		\right)
	\end{equation*}
\end{lemma}	
\begin{proof}
	We prove the Lemma for $a=0$. Indeed, the general case $a>0$ can be dealt with considering the weighted tensor $e^{at}\tau$ instead of 
	$\tau$. 
	We decompose the solution $\tau\,=\,\sum_{q\in\ZZ} \tau_q$, where $\tau_q$ is solution of the following system of PDE's:
	\begin{equation*}
	\left\{\hspace{0.2cm}
	\begin{alignedat}{2}
		&\,\partial_t \tau_q\,+\,\uu\cdot \nabla \tau_q +	- \,\omega \tau_q \,+\,\tau_q\omega \,=\,f_q
		\hspace{3cm}
		&&\RR_+\times \RR^\dd, \vspace{0.1cm}\\
		&\,
		\tau_{|t=0}	\,=\,\Dd_q \tau_0			
		&&\hspace{1.02cm} \RR^\dd.\vspace{0.1cm}																							
	\end{alignedat}
	\right.
	\end{equation*}
	We first remark that, for any fixed positive $N>0$
	\begin{equation*}
		\|\,\tau(t)\,\|_{\BB_{p,1}^0}
		\leq 
		\sum_{j,\,q\in \ZZ}\|\,\Dd_j\tau_q(t)\,\|_{L^p_x}
		\leq 
		\sum_{\substack{ j,\,q\in \ZZ\\ |j-q|\leq N}}\|\,\Dd_j \tau_q(t)\,\|_{L^p_x}
		\,+\,
		\sum_{\substack{ j,\,q\in \ZZ\\ |j-q|> N}}\|\,\Dd_j \tau_q(t)\,\|_{L^p_x}
		\,=:\,\mathcal{I}_N\,+\,\I\I_N.
	\end{equation*}
	Hence, thanks to Lemma \ref{lemma:L^pbound_of_tau}, we gather
	\begin{equation*}
	\begin{aligned}
			\I_N\,=\,
			\sum_{\substack{ j,\,q\in \ZZ\\ |j-q|\leq N}}\|\,\Dd_j \tau_q(t)\,\|_{L^p_x}
			\, &\lesssim \,
			\sum_{\substack{ j,\,q\in \ZZ\\ |j-q|\leq N}}\|\,\tau_q(t)\,\|_{L^p_x}
			\, \lesssim \,
			N \sum_{\substack{ q\in \ZZ}}\|\,\tau_q(t)\,\|_{L^p_x}\\
			\,&\lesssim
			N \sum_{\substack{ q\in \ZZ}}
			\left(
			\,
				\|\,\Dd_q \tau_0\,\|_{L^p_x}
				\,+\,
				\int_0^t
				\,\|\, f_q(s)\,\|_{L^p_x}
				\dd\, s
			\,
			\right)\\
			\,&\lesssim
			N
			\left(
			\,
				\|\, \tau_0\,\|_{\BB_{p,1}^0}
				\,+\,
				\int_0^t
				\,\|\, f(s)\,\|_{\BB_{p,1}^0}
				\dd\, s
			\,
			\right).
	\end{aligned}
	\end{equation*}
	In order to handle $\I\I_N$, we make use of Lemma \ref{lemma:bound_of_tau-exp_on_nablau} where the initial data $\Dd_q\tau_0$ is assumed in $\BB_{p,1}^{\pm \ee}$ for a small parameter $\ee\in (0,1)$ and 
	a source term $f$ in $L^1(0,T;\BB_{p,1}^s)$:
	\begin{equation*}
		\| \,\tau_q(t) \,\|_{\BB_{p,1}^{\pm \ee}}
		\,\leq\,
		\left(
		\,
			\| \,\Dd_q \tau_0 \,\|_{\BB_{p,1}^{\pm \ee}}\,+\,
		   \int_0^t 
		   \| \,\Dd_q f\,\|_{\BB_{p,1}^{\pm \ee}}
		\,
		\right)
		\exp
		\left\{
		\,
		C
		\int_0^t
		\| \,\nabla \uu \,\|_{\BB_{\infty,1}^0}
		\,
		\right\}.
	\end{equation*}
	We thus recast the above inequality in the following form
	\begin{equation*}
		\| \,\Dd_j \tau_q (t) \,\|_{L^p_x}
		\,\lesssim\,
		2^{-|j-q|\ee}a_j(t)
		\left(
		\,
			\|\,\Dd_q \tau_0 \|_{L^p_x}
			\,+\,
			\int_0^t
			\|\,\Dd_q f(s)\,\|_{L^p_x}
			\dd s
		\,
		\right)
		\exp
		\left\{
		\,
		C
		\int_0^t
		\| \,\nabla \uu \,\|_{\BB_{\infty,1}^0}
		\,
		\right\}
	\end{equation*}
	where $(a_j(t))_{j\in\ZZ}$ belongs to $\ell^1(\ZZ)$ with norm $\|\,(a_j(t))\,\|_{\ell^1(\ZZ)}=1$. We deduce that $\I\I_N$ is bounded by
	\begin{equation*}
		\begin{aligned}
			\I\I_N
			\,=\,
				\sum_{\substack{ j,\,q\in \ZZ\\ |j-q|> N}}\|\,\Dd_j \tau_q(t)\,\|_{L^p_x}
			\,\leq \,
			2^{-N\ee}	
			\left(
		\,
			\|\,\tau_0 \|_{\BB_{p,1}^0}
			\,+\,
			\int_0^t
			\|\,f(s)\,\|_{\BB_{p,1}^0}
			\dd s
		\,
		\right)
		\exp
		\left\{
		\,
		C
		\int_0^t
		\| \,\nabla \uu(s) \,\|_{\BB_{\infty,1}^0} \dd s
		\,
		\right\}.
		\end{aligned}
	\end{equation*}
	Imposing the following relation between $N$, $\ee$ and $\uu$
	\begin{equation*}
		 N\,=\,\frac{C}{\ee \ln 2}  \int_0^t 	
		\| \,\nabla \uu(s) \,\|_{\BB_{\infty,1}^0}\dd s,
	\end{equation*}
	we finally achieve the statement of the Lemma.
\end{proof}


\smallskip
\noindent
We now prove some a-priori estimates within the functional framework of Lorentz norms, for the non-stationary Stokes system. The following Lemma allows us to control the term arising from the combination of the conformation equation within the Navier-Stokes system.
\begin{lemma}\label{lemma_Lorentz_bound}
	For any time $t\geq 0$ and viscosity $\nu>0$, the following estimate holds true
	\begin{equation*}
		\left\|\,\int_0^t \mathcal{P} e^{\nu(t-s)\Delta} \nabla f(s) \dd s \right\|_{L^{\dd, \infty}_x} \,\leq\,
		\begin{cases}
		 C\| \,f\,\|_{L^\infty(0, t; L^1_x)}\quad & \text{if }\dd = 2,\\
		 C \| \,f\,\|_{L^\infty(0, t; L^{\frac{\dd}{2}, \infty}_x)}\quad & \text{if }\dd > 2.\\
		\end{cases}
	\end{equation*}
\end{lemma} 
\begin{proof}
	Without loss of generality, we recast ourselves to the case of $\nu=1$. We begin remarking that the projector $\mathcal{P}$ is a bounded operator from 
	$L^p(\RR^\dd)$ to itself, for any $1>p<\infty$. Hence, by real interpolation we get that $\Pp$ is a bounded linear operator within Lorentz spaces
	\begin{equation*}
		\Pp \in \mathcal{L}(L^{\frac{\dd}{2},\,\infty}(\RR^\dd),\,L^{\frac{\dd}{2},\,\infty}(\RR^\dd)).
	\end{equation*}
	This allows us to cancel the presence of the project $\Pp$ in any inequality we aim to prove and limit the proof to the case of the heat kernel operator.
	We hence decouple the integral we aim to bound 
	\begin{equation*}
		\int_0^t e^{(t-s)\Delta}\nabla f(s)\dd s
	\end{equation*}
	by $J_\ee$ and $J^\ee$ defined as follows: fixing a small positive parameter $\ee$ 
	\begin{equation*}
		J_\ee\,:=\,\int_0^{t-\ee} e^{(t-s)\Delta}\nabla f(s)\dd s
			\quad\quad\text{and}\quad\quad
		J^\ee\,:=\,\int_{t-\ee}^t e^{(t-s)\Delta}\nabla f(s)\dd s
	\end{equation*}
	For any $\tau\geq 0$ and $1\leq p\leq p\leq \infty$ we have
	\begin{equation*}
		\left\|\,e^{\tau \Delta}\nabla \,\right\|_{\mathcal{L}(L^p_x,\,L^q_x)} 
		\,=\,
		\left\|\,\mathfrak{F}^{-1}(e^{-\tau|\xi|^2}i\xi) \,\right\|_{L^r_x} 
		\,=\,
		\frac{1}{\sqrt{\tau}}
		\left\|\,\mathfrak{F}^{-1}(e^{-|\sqrt{\tau}\xi|^2}i\sqrt{\tau} \xi) \,\right\|_{L^r_x}
		\frac{1}{\tau^{\frac{\dd+1}{2}}}
		\,=\,
		\left\|\,\mathfrak{F}^{-1}(e^{-|\xi|^2}i\xi)\left(\frac{x}{\sqrt{\tau}}\right) \,\right\|_{L^r_x}
	\end{equation*}
	hence, by a change of variable, we finally get that there exists a constant $C$ for which
	\begin{equation*}
		\left\|\,e^{\tau \Delta}\nabla \,\right\|_{\mathcal{L}(L^p_x,\,L^q_x)} 
		\,\leq\,
		\frac{C}{\tau^{\frac{\dd}{2}\frac{1}{r'}+\frac{1}{2}}}
		\quad\quad\text{where}\quad\quad
		\frac{1}{r}\,+\,\frac{1}{p}\,=\,\frac{1}{q}\,+\,1,
		\quad\quad\text{namely}\quad\quad
		\frac{1}{r'}\,=\,\frac{1}{p}\,-\,\frac{1}{q}.
	\end{equation*}
	We first assume that $\dd \,=\,2$. We hence get
	\begin{equation*}
	\begin{aligned}
		\left\|\, J^\ee\,\right\|_{L^1_x}
		\,&\leq\,
		\int_{t-\ee}^t \|\,e^{(t-s)\Delta}\nabla f(s)\,\|_{L^1_x}\dd s\\
		\,&\leq\,
		\int_{t-\ee}^t \|\,e^{(t-s)\Delta}\nabla\, \|_{\mathcal{L}(L^1_x,\,L^1_x)}\|\,f(s)\,\|_{L^1_x}\dd s\\
		\,&\leq\,
		C\int_{t-\ee}^t\frac{\|\,f(s)\,\|_{L^1_x}}{|\,t\,-\,s\,|^{\frac{1}{2}}}\dd s\\
		\,&\leq\,
		C\|\,f(s)\,\|_{L^\infty(0,t;L^1_x)}\int_0^\ee\frac{1}{s^{\frac{1}{2}}}\dd s\,\leq\, C\sqrt{\ee}\|\,f(s)\,\|_{L^\infty(0,t;L^1_x)},
	\end{aligned}
	\end{equation*}
	while similarly
	\begin{equation*}
	\begin{aligned}
		\left\|\, J_\ee\,\right\|_{L^\infty_x}
		\,&\leq\,
		\int_{0}^{t-\ee} \|\,e^{(t-s)\Delta}\nabla f(s)\,\|_{L^\infty_x}\dd s\\
		\,&\leq\,
		\int_{0}^{t-\ee} \|\,e^{(t-s)\Delta}\nabla\, \|_{\mathcal{L}(L^1_x,\,L^\infty_x)}\|\,f(s)\,\|_{L^1_x}\dd s\\
		\,&\leq\,
		C\int^{t-\ee}_0\frac{\,\,\|\,f(s)\,\|_{L^1_x}}{|\,t\,-\,s\,|^{\frac{3}{2}}}\dd s\\
		\,&\leq\,
		C\|\,f(s)\,\|_{L^\infty(0,t;L^1_x)}\int_\ee^\infty\frac{1}{s^{\frac{3}{2}}}\dd s\,\leq\, \frac{C}{\sqrt{\ee}}\|\,f(s)\,\|_{L^\infty(0,t;L^1_x)}.
	\end{aligned}
	\end{equation*}
	In virtue of Remark \ref{rmk:Lorentz_space} about the real interpolation $L^{2,\infty}(\RR^2)\,=\,(L^1,\,L^\infty)_{1/2 , \infty}$, we hence conclude that 
	in dimension $\dd\,=\,2$
	\begin{equation*}
		\left\|\,\int_0^t \mathcal{P} e^{\nu(t-s)\Delta} \nabla f(s) \dd s \right\|_{L^{2, \infty}_x}
		\,\leq\,
		C\|\,f(s)\,\|_{L^\infty(0,t;L^1_x)}.
	\end{equation*}
	We now address the case of a dimension $\dd>2$. With a similar technique as the one used above, we remark that
	\begin{equation*}
	\begin{aligned}
		\left\|\, J^\ee\,\right\|_{L^{\frac{\dd}{2}, \infty}_x}
		\,&\leq\,
		\int_{t-\ee}^t \|\,e^{(t-s)\Delta}\nabla f(s)\,\|_{L^{\frac{\dd}{2}, \infty}_x}\dd s\\
		\,&\leq\,
		\int_{t-\ee}^t \|\,e^{(t-s)\Delta}\nabla\, \|_{\mathcal{L}(L^{\frac{\dd}{2}, \infty}_x,\,L^{\frac{\dd}{2},\infty}_x)}\|\,f(s)\,\|_{L^{\frac{\dd}{2},\infty}_x}\dd s\\
		\,&\leq\,
		C\int_{t-\ee}^t\frac{\|\,f(s)\,\|_{L^{\frac{\dd}{2},\infty}_x}}{|\,t\,-\,s\,|^{\frac{1}{2}}}\dd s\\
		\,&\leq\, C\sqrt{\ee}\|\,f(s)\,\|_{L^\infty(0,t;L^{\frac{\dd}{2},\infty}_x)},
	\end{aligned}
	\end{equation*}
	while
	\begin{equation*}
	\begin{aligned}
		\left\|\, J_\ee\,\right\|_{L^\infty_x}
		\,&\leq\,
		\int_{0}^{t-\ee} \|\,e^{(t-s)\Delta}\nabla f(s)\,\|_{L^\infty_x}\dd s\\
		\,&\leq\,
		\int_{0}^{t-\ee} \|\,e^{(t-s)\Delta}\nabla\, \|_{\mathcal{L}(L^{\frac{\dd}{2},\infty }_x,\,L^\infty_x)}\|\,f(s)\,\|_{L^{\frac{\dd}{2},\infty}_x}\dd s\\
		\,&\leq\,
		C\int^{t-\ee}_0\frac{\,\,\|\,f(s)\,\|_{L^{\frac{\dd}{2},\infty}_x}}{|\,t\,-\,s\,|^{\frac{3}{2}}}\dd s\\
		\,&\leq\, \frac{C}{\sqrt{\ee}}\|\,f(s)\,\|_{L^\infty(0,t;L^{\frac{\dd}{2},\infty}_x)}.
	\end{aligned}
	\end{equation*}
	We recall again Remark \ref{rmk:Lorentz_space} concerning this time the real interpolation $L^{\dd,\infty}(\RR^2)\,=\,
	(L^{{\dd}/{2}, \infty},\,L^\infty)_{1/2 , \infty}$, we hence conclude that 
	in dimension $\dd\,\geq\,3$
	\begin{equation*}
		\left\|\,\int_0^t \mathcal{P} e^{\nu(t-s)\Delta} \nabla f(s) \dd s \right\|_{L^{\dd,\infty}_x}
		\,\leq\,
		C\|\,f(s)\,\|_{L^\infty(0,t;L^{\frac{\dd}{2},\infty}_x)}.
	\end{equation*}
\end{proof}

\noindent We conclude this section stating a lemma which will allow to estimate the lifespan of our classical solution for problem \eqref{main_system}. Indeed, we will prove in section 4 that the flow $\uu$ of our solution satisfies an inequality of the following type:
\begin{equation}\label{ineq-lemma:lifespan}
	f(t)  \leq g_1(t) + \int_0^t g_2(s) f(s) \dd s + g_3(t)\int_0^t  f(s)^2 \dd s,
\end{equation}
where $f(t) = \| \uu \|_{L^1(0,t; \BB_{\infty,1}^1)}$, and $g_1,\,g_2$ and $g_3$ are polynomials in $t$ with positive coefficients. In case of $\mu=0$, $g_2$ disappears in the above inequality, reducing to a classical Gronwall form. In the general case of $\mu\geq 0$ we still have however the following statement.
\begin{lemma}
	Let $f$ be a continuous positive scalar function in $\mathcal{C}([0,T_{max}))$ with $f(0) = 0$. Assume further that inequality \eqref{ineq-lemma:lifespan} is satisfied for some polynomials $g_1,\,g_2$ and $g_3$ with positive coefficients. Then $T_{max}$ is such that 
\begin{equation*}
	T_{\max} = \sup\left\{ T\in [0,\infty) \quad\text{such that}\quad 
	\int_0^{T} tg_3(t)\exp\Big\{2\int_0^t g_2(s)\dd s\Big\}\dd t < 1 \right\}
\end{equation*}
and the following inequality is satisfied
\begin{equation*}
	f(t) \leq \frac{g_1(t)}{1-\int_0^t  tg_3(t)\exp\Big\{2\int_0^t g_2(s)\dd s\Big\}\dd s}\exp\Big\{\int_0^t g_2(s)\dd s\Big\},
	\qquad \text{for any}\quad 
	t\in [0,T_{\rm max})
\end{equation*}
\end{lemma}

\section{Global-in-time solutions in dimension two}\label{sec:bidim}
\noindent
This section is devoted to the proof of Theorem \ref{main_thm0}. We begin with introducing the following Friedrichs-type approximation of the System \eqref{main_system}: 
\begin{equation}\label{main_system_appx1}
\left\{\hspace{0.2cm}
	\begin{alignedat}{2}
		&\,\partial_t \tau^{n}\,+\,\uu^n\cdot \nabla \tau^{n}	- \,J_n\omega^n \tau^{n} \,+\,\tau^{n}J_n \omega^n +a \tau^n\,=\,\mu \mathbb{D}^n
		\hspace{3cm}
		&&\RR_+\times \RR^2, \vspace{0.1cm}	\\		
		&\,\partial_t \uu^{n}   - \, \nu \Delta \uu^{n}  
		\,+\, \nabla  \pre^{n}  = -J_n (\uu^{n}\cdot \nabla \uu^{n}\,)
		\,+\,
		\Div\,J_n\tau^{n}
									\hspace{3cm}									&& \RR_+\times \RR^2,\vspace{0.1cm}\\
		&\,\Div\, \uu^{n}\,=\,0			
		&&\RR_+\times \RR^\dd, \vspace{0.1cm}	\\			
		& ( \uu^{n},\,\tau^{n})_{|t=0}	\,=\,(J_n \uu_0,\,J_n \tau_0)			
		&&\hspace{1.02cm} \RR^2.\vspace{0.1cm}																							
	\end{alignedat}
	\right.
\end{equation}
Denoting by $\mathbf{1}_{A}$ the characteristic function of a set $A$, for any $n\in \NN$ we introduce the regularizing operator $J^n$ by the formula
\begin{equation*}
	\FF(\,J^n g\,)(\xi) := \mathbf{1}_{\mycal{C}_n}(\xi) \hat{g}(\xi)
\end{equation*}
which localizes the Fourier transform of a suitable function $g$ into the annulus  $\mycal{C}_n = \{ \xi\in\RR^2,\, |\xi|\in [1/n,\,n]\,\}$. Hence, we claim that an approach coupling the Friedrich's scheme together with the Schaefer fixed point theorem, allows us to construct a sequence of approximate solutions $(\uu^n,\,\tau^n)_{n\in \NN}$ satisfying the following class affinity:
\begin{equation*}
\begin{aligned}
	\uu^n &\in L^\infty_{\rm loc}(\RR_+,\,L^2(\RR^2)\cap \BB_{p,1}^{\frac{2}{p}-1})
						\cap 
						L^2_{\rm loc}(\RR_+,\, \dot H^1(\RR^2))
						\cap 
						L^1_{\rm loc}(\RR_+,\,\BB_{p,1}^{\frac{2}{p}+1})),\\
	\tau^n &\in  L^\infty_{\rm loc}(\RR_+,\,L^2(\RR^2)\cap \BB_{p,1}^{\frac{2}{p}}).
\end{aligned}
\end{equation*}
We refer the reader to \cite{Dea} for some details about this procedure, where the first author showed a similar result for a different system of PDE's. The purpose of the next sections is to reveal the above regularities of the approximate solutions  $(\uu^n,\,\tau^n)_{n\in \NN}$. This result is achieved into two main steps:
\begin{itemize}
	\item[(i)] {\it Propagating the Lipschitz regularity of the velocity field $\uu^n$}: the initial data $(\uu_0,\,\tau_0)$ belongs to $\BB_{p,1}^{2/p-1}\times \BB_{p,1}^{2/p}$ which is embedded 
						into $\BB_{\infty,1}^{-1}\times \BB_{\infty,1}^{0}$. This last regularity will be hence propagated in time, allowing to control $\uu^n$ into the functional framework given by
						\begin{equation*}
								\nabla \uu^n \in L^1_{\rm loc}(\RR_+,\,\BB_{\infty, 1}^0)\,\hookrightarrow \, L^1_{\rm loc}(\RR_+,\,L^\infty(\RR^2)),
						\end{equation*}
						from which we will deduce that $\uu^n$ is Lipschitz, globally in time.
	\item[(ii)] {\it Propagating higher regularities of solutions}: we will propagate the specific regularity of the initial data $(\uu_0,\,\tau_0)$ in $\BB_{p,1}^{2/p-1}\times \BB_{p,1}^{2/p}$, making use of the 
						Lipschitz condition achieved in point $(i)$.
\end{itemize}
Last, we will estimate the mentioned norms with a bound independent on the index $n\in\NN$. This will allow us to pass to the limit and construct a classical solution of system \eqref{main_system} within the functional framework of Theorem \ref{main_thm0}.

\subsection{\bf Lipschitz regularity of the velocity field}$\,$

\noindent 
In this section we show some mathematical properties of solutions for the approximate system \eqref{main_system_appx1}. The main goal is to establish the propagation of Lipschitz regularity for the velocity field $\uu^n$, namely to show that $\nabla \uu^n$ belongs to the functional space
\begin{equation*}
	L^1_{\rm loc}([0,T_{\rm max}),\,\BB_{\infty, 1}^0)
	\,\hookrightarrow\,
	L^1_{\rm loc}([0,T_{\rm max}),\,L^\infty(\RR^2)),
\end{equation*} 
for some suitable positive time $T_{\rm max}>0$.  We also aim in controlling this regularity with a bound which is independent on the index $n\in\NN$, in order to keep this property also when passing to the limit. 

\noindent
We collect in the following statement the result we aim to prove.
\begin{theorem}\label{thm:Lipschitz-prop-u^n}
	Assume that the initial data $\uu_0$ and $\tau_0$ belongs to $L^2(\RR^2)\cap \BB_{\infty, 1}^{-1}$ and $L^2(\RR^2)\cap \BB_{\infty, 1}^0$, respectively. Then, 
	the solutions $(\uu^n,\,\tau^n)$ of the system \eqref{main_system_appx}, belongs to the following functional framework
	\begin{equation*}	
	\begin{aligned}
			\uu^n\,\in\,L^\infty_{\rm loc}([0,T_{\rm max}),\,\BB_{\infty,1}^{-1})
						\cap
						L^1_{\rm loc}([0,T),\,\BB_{\infty, 1}^1)
			\quad\quad\text{and}\quad\quad
			\tau^n\,\in\,L^\infty_{\rm loc}([0,T_{\rm max}),\,\BB_{\infty, 1}^{0}),
	\end{aligned}
	\end{equation*}
	with $T_{\rm max}=+\infty$ when $\mu=0$. Furthermore, there exists two smooth functions $\Upsilon_{1,\nu}(T,\,\uu_0,\,\tau_0)$ and $\Upsilon_{2,\nu}(T,\,\uu_0,\,\tau_0)$, $0\leq T<T_{\rm max}$ for which the following 
	inequalities hold true:
	\begin{equation*}
		\begin{aligned}
			\nu
			\|\,\uu^n\,\|_{L^1(0, T;\BB_{\infty, 1}^{1} )}
			\,&\leq\,
			\Upsilon^1_\nu(T,\, \uu_0,\,\tau_0)
			\quad\quad
			\text{and}\quad\quad
			\|\,\uu^n\,\|_{L^\infty(0, T;\BB_{\infty, 1}^{-1} )}
			\,&\leq\,
			\Upsilon^2_\nu(T,\, \uu_0,\,\tau_0),
		\end{aligned}
		\end{equation*}
		with also
		\begin{equation*}
			\|\,\tau^n\,\|_{L^\infty(0,T;\BB_{\infty, 1}^0)}
			\,\leq\,
			C\|\,\tau_0\,\|_{\BB_{\infty, 1}^0}
			\left(
				\,1\,+\,\nu^{-1}\Upsilon^1_\nu(T,\, \uu_0,\,\tau_0)
			\right).
		\end{equation*}
	Both functions $\Upsilon_{1,\nu}$ and $\Upsilon_{2,\nu}$ vanish when $T = 0$, they are increasing in time $T>0$ and they depend uniquely on the 
	norms $\|\,\uu_0\,\|_{L^2(\RR^2)\cap \BB_{\infty,1}^{-1}}$ and  $\|\,\tau_0\,\|_{L^2(\RR^2)\cap \BB_{\infty,1}^{0}}$.	The exact formulation of $\Upsilon_{1,\nu}$ and $\Upsilon_{2,\nu}$ is stated in Remark \eqref{def-Upsilon}.
\end{theorem}

\noindent 
The proof of Theorem \ref{thm:Lipschitz-prop-u^n} requires to proceed into several fundamental steps. We will first begin with determining the standard energy inequalities for system \eqref{main_system_appx} (cf. Proposition \ref{prop:bid1}). These inequalities will allow us then to unlock some delicate semi-group estimates related in-primis to the mild formulation of the velocity field $\uu^{n}(t)$:
\begin{equation}\label{mild-form-bid}
	\uu^n(t)\,:=\,
	\underbrace{{\rm e}^{\nu t \Delta}J_n\uu_0}_{\uu^L(t)}
	\,+\,
	\underbrace{\int_0^t \mathcal{P}{\rm e}^{(t-s)\nu\Delta} \Div\,J_n(\uu^n\otimes\uu^n)(s)\dd s}_{\uu^n_1(t)}
	\,+\,
	\underbrace{\int_0^t \mathcal{P}{\rm e}^{(t-s)\nu\Delta} \Div\,J_n\tau^n(s)\dd s}_{\uu^n_2(t)}.
\end{equation}
Here $\Pp$ is the Leray projector into the space of free-divergence vector fields, while the operator ${\rm e}^{\nu t\Delta}$ stands for the heat semigroup in the whole space.
We recognize in the above identity three distinct terms $\uu^L,\,\uu^n_1,\,\uu^n_2$, the first one related to the linear contribution of system \ref{main_system_appx}, the second one tackling the non-linearity due to the Navier-Stokes contribution to the system and the last one specifically correlated to the evolution of the conformation tensor $\tau^n$. Due to the different structures of these terms, each of them will be separately handled, with appropriate estimates in Chemin-Lerner Besov spaces. We will then proceed as follows:
\begin{itemize}
	\item[(i)]	we will first establish some standard energy estimate of our system 
					(cf. Proposition \ref{prop:bid1} and Proposition \ref{prop:bid2}),
	\item[(ii)]  we will then propagate suitable regularities in order to control the second term $\uu^n_1(t)$ (cf. Lemma \ref{lemma:un-un}, Remark \ref{bid-rmk} 					and Proposition \ref{prop-almost-there}),
	\item[(iii)] we will hence analyze the remaining term $\uu^n_2(t)$ (cf. Lemma \ref{control-of-u2}),
	\item[(iv)]  we will summarize our previous estimates at the end of the section, proving Theorem \ref{thm:Lipschitz-prop-u^n} and propagating the 
					Lipschitz-in-space regularity of $\uu^n$.
\end{itemize}
Thanks to a classical energy approach we begin with stating the following proposition:
\begin{prop}\label{prop:bid1}
	For any $n\in \NN$, $(\uu^n,\,\tau^n)$ belongs to 
	$\mathcal{C}_{\rm loc}(\RR_+, L^2(\RR^2))$ and $\nabla u^n$ belongs to $L^2(\RR_+, L^2(\RR^2))$. When $\mu=0$,  
	 $(\uu^n,\,\tau^n)$ satisfies
	\begin{equation}\label{energy-estimate}
		\|\tau^n(t)\|_{L^2(\RR^2)}=\|\tau_0\|_{L^2(\RR^2)}e^{-at},
		\quad
		\|\uu^n \|_{L^\infty(0, t;\, L^2(\RR^2))}^2+\nu \|\nabla \uu^n \|_{L^2(0, t;\, L^2(\RR^2))}^2
		\,\leq\,
		\| \uu_0\|_{L^2(\RR^2)}^2+ \| \tau_0\|_{L^2(\RR^2)}^2
		\frac{1-e^{-2at}}{2a\nu},
	\end{equation}
	for any time $t>0$. Otherwise, when $\mu >0$ then
	\begin{equation}\label{energy-estimate_mu}
	\begin{aligned}		
		\mu \|\uu^n \|_{L^\infty(0, T;\, L^2(\RR^2))}^2+
		 \|\tau^n \|_{L^\infty(0, T;\, L^2(\RR^2))}^2
		 &+a\|\tau^n \|_{L^2(0, T;\, L^2(\RR^2))}^2+\\&+
		\nu \mu \|\nabla \uu^n \|_{L^2(0, T;\, L^2(\RR^2))}^2
		\leq 
		\mu \| \uu_0\|_{L^2(\RR^2)}^2+ \| \tau_0\|_{L^2(\RR^2)}^2.
	\end{aligned}
	\end{equation}
\end{prop}
\noindent Furthermore, one can remark that $\uu^n$ belongs to a more refined functional space, namely:
\begin{prop}\label{prop:bid2}
	For any integer $n\in \NN$ the solution  $\uu^n$ belongs to $\tilde L^\infty_{\rm loc}(\RR_+, L^2(\RR^2))$. 
	Furthermore, if $\mu=0$ then
	\begin{equation*}
	\begin{aligned}
		\|\uu^n \|_{\tilde L^\infty(0, T; L^2(\RR^2))}
		\leq 
		\| \uu_0  \|_{L^2(\RR^2)}
			+
		C\left(
			\nu^{-1}\| \uu_0  \|_{L^2(\RR^2)}^2
			+
			\| \tau_0  \|_{L^2(\RR^2)}
			\sqrt{\frac{1-e^{-2aT}}{2a\nu }}
			+
			\nu^{-1}
			\| \tau_0  \|_{L^2(\RR^2)}^2
			\frac{1-e^{-2aT}}{2a\nu}
		\right)
		\end{aligned}
	\end{equation*}
	otherwise, when $\mu>0$,
	\begin{equation*}
	\begin{aligned}
		\|\uu^n \|_{\tilde L^\infty(0, T; L^2(\RR^2))}
		\leq 
		\| \uu_0  \|_{L^2(\RR^2)}
			+
		C\bigg\{
			\nu^{-1}\| \uu_0  \|_{L^2(\RR^2)}^2
			&+
			\mu^{-1} \nu^{-1}
			\| \tau_0 \|_{L^2(\RR^2)}^2
			+\\&+
			\Big(
				\| \tau_0  \|_{L^2(\RR^2)}^2+
				\mu \| \uu_0  \|_{L^2(\RR^2)}^2
			\Big)^\frac{1}{2}
			\sqrt{\frac{1-e^{-2aT}}{2a\nu }}
		\bigg\}.
	\end{aligned}
	\end{equation*}
\end{prop}
\begin{proof}
	Thanks to the mild formulation \eqref{mild-form-bid}, the velocity field $\uu^{n}$ can be decomposed into three terms 
	$\uu^{n} =\uu^L\,+\, \uu_1^{n} \,+\,\uu^{n}_2$.
	The heat kernel allows to estimate the initial term $\uu^L$ as follows:
	\begin{equation*}
		\|\,\uu^L\,\|_{\tilde L^\infty(0, T, L^2(\RR^2))}
		\,\leq\, \|\,\uu_0\,\|_{L^2(\RR^2)}.
	\end{equation*}
	Next, we address $\uu^{n}_1$ in $\tilde L^\infty(0, T; L^2(\RR^2))$. We apply the dyadic block $\Dd_q$ to $\uu^n_1$, for a fixed integer $q\in \ZZ$. We 
	first remark that
	\begin{equation*}
		\|\,\Dd_q \uu_1^{n}\,\|_{L^2(\RR^2)}
		\,\lesssim\,
		\int_0^t 2^q {\rm e}^{-c(t-s)\nu 2^{2q}} \|\,\Dd_q (\uu^n(s) \otimes \uu^n (s))\,\|_{L^2(\RR^2)} \dd s,
	\end{equation*}
	so that the Young inequality applied to the last convolution leads to
	\begin{equation*}
	\begin{aligned}
		\|\,\Dd_q \uu_1^{n}\,\|_{L^\infty(0, T;  L^2(\RR^2))}
		\,&\leq\,
		C\nu^{-\frac{1}{2}}
		\|\,\Dd_q (\uu^n \otimes \uu^n )\,\|_{L^2(0, T; L^2(\RR^2))}\\
		&\leq\,	
		C	\nu^{-\frac 12}\big\| \, \|\,\uu^n(t) \,\|_{L^4(\RR^2)} ^2 \big\|_{L^2(0, T)} \\
		&\leq\,
		C
		\|\,\uu^n \,\|_{L^\infty(0, T; L^2(\RR^2))}\nu^{-\frac 12}\|\,\nabla \uu^n \,\|_{L^2(0, T; L^2(\RR^2))}
		\\&\leq\,
		\frac{1}{\nu}
		\Big(
		\|\,\uu^n \,\|_{L^\infty(0, T; L^2(\RR^2))}^2
		+
		\nu\|\,\nabla \uu^n \,\|_{L^2(0, T; L^2(\RR^2))}^2
		\Big).
	\end{aligned}
	\end{equation*}
	Taking the supremum with respect to the parameter $q\in \ZZ$, we eventually deduce that
	\begin{equation*}
	\begin{alignedat}{4}
		 \|\, \uu_1^{n}\,\|_{\tilde L^\infty(0, T;  L^2(\RR^2))}
		\,&\leq\,
		C\nu^{-1}
		\left(
			\,\|\, \uu_0\,\|_{L^2(\RR^2)}^2\,+\, \|\, \tau_0\,\|_{L^2(\RR^2)}^2
			\frac{1-e^{-2aT}}{2a\nu}
		\right)
		\qquad&&\text{if}\quad \mu = 0,\\
		\|\, \uu_1^{n}\,\|_{\tilde L^\infty(0, T;  L^2(\RR^2))}
		\,&\leq\,
		C\nu^{-1}
		\left(
			\,\|\, \uu_0\,\|_{L^2(\RR^2)}^2\,+\,\mu^{-1} \|\, \tau_0\,\|_{L^2(\RR^2)}^2
		\right)
		\qquad&&\text{if}\quad \mu > 0,
	\end{alignedat}
	\end{equation*}	
	for any integer $n\in \NN$ and for a suitable positive constant $C$. Next we deal with $\uu^{n}_2$ and we remark that
	\begin{equation*}
	\begin{aligned}
		\|\, \Dd_q \uu^{n}_2 \,\|_{L^2(\RR^2)}
		\,&\lesssim\,
		\int_0^t  2^q {\rm e}^{-c\nu (t-s)2^{2q}} \| \Dd_q \tau^{n}(s) \|_{L^2(\RR^2)}\dd s, \\
		&\lesssim\, 
		\nu^{-\frac 12}\| \,\Dd_q \tau^{n} \,\|_{L^2(0, T; L^2(\RR^2))}\,\leq\,C\nu^{-\frac 12} \| \,\tau^{n} \,\|_{L^2(0, T; L^2(\RR^2))},
	\end{aligned}
	\end{equation*}
	therefore
	\begin{equation}\label{ineq_LinftyL2_of_u2}
	\begin{alignedat}{4}
		\|\, \uu^{n}_2 \,\|_{\tilde L^\infty(0, T; L^2(\RR^2))}
		\,&\leq\,
		C\|\,\tau_0\,\|_{L^2(\RR^2)}\sqrt{\frac{1-e^{-2aT}}{2a\nu }}
		&&\qquad\text{if}\quad \mu = 0,\\
		\|\, \uu^{n}_2 \,\|_{\tilde L^\infty(0, T; L^2(\RR^2))}
		\,&\leq\,
		C
		\Big(\mu\|\uu_0\|_{L^2(\RR^2)}^2+\|\,\tau_0\,\|_{L^2(\RR^2)}^2
		\Big)^\frac{1}{2}
		\sqrt{\frac{1-e^{-2aT}}{2a\nu }}
		&&\qquad\text{if}\quad \mu > 0,
		\end{alignedat}
	\end{equation}
	and this concludes the proof of the Proposition.
\end{proof}

\begin{lemma}\label{lemma:un-un}
	For any positive integer $n\in\NN$ and for any positive time $T>0$, the velocity field $\uu^n$ satisfies the class affinity
	\begin{equation*}
			\uu^n\otimes \uu^n  \in \tilde L^\frac{4}{3}(0, T;\BB_{2, 1}^{\frac{1}{2}}).
	\end{equation*}
	Furthermore, when $\mu=0$, the following bound holds true
	\begin{equation*}
		\left\|\, \uu^n\otimes \uu^n \,\right\|_{ \tilde L^\frac{4}{3}(0, T;\,\BB_{2, 1}^{\frac{1}{2}})}
		\,\leq\,
		C\nu^{-\frac 34}
		\Big(
			\| \uu_0  \|_{L^2(\RR^2)}
			+
			\nu^{-1}\| \uu_0  \|_{L^2(\RR^2)}^2
			+
			\| \tau_0  \|_{L^2(\RR^2)}
			\sqrt{
			\frac{1-e^{-2aT}}{2a\nu}
			}
			+
			\nu^{-1}
			\| \tau_0  \|_{L^2(\RR^2)}^2
			\frac{1-e^{-2aT}}{2a\nu}
		\Big)^2,
	\end{equation*}
	while as $\mu>0$
	\begin{equation*}
	\begin{aligned}
		\left\|\uu^n\otimes \uu^n \right\|_{ \tilde L^\frac{4}{3}(0,T;\,\BB_{2, 1}^{\frac{1}{2}})}
		\leq
		C\nu^{-\frac 34}
		\Big\{
			\Big(1&+\mu^{\frac{1}{2}}\sqrt{\frac{1-e^{-2aT}}{2a\nu}}\Big)
			\| \uu_0  \|_{L^2(\RR^2)}
			+
			\nu^{-1}\| \uu_0  \|_{L^2(\RR^2)}^2
			+\\
			&+
			\Big(\mu^{-\frac{1}{2}}
			+
			\sqrt{\frac{1-e^{-2aT}}{2a\nu}}
			\Big)
			\| \tau_0  \|_{L^2(\RR^2)}
			+
			\nu^{-1}
			\| \tau_0  \|_{L^2(\RR^2)}^2
			\frac{1-e^{-2aT}}{2a\nu}
		\Big\}^2,
	\end{aligned}
	\end{equation*}
	for a positive constant $C$ that does not depend on the index $n\in \NN$.
\end{lemma}
\begin{proof}
	We first claim that for any integer $n\in\NN$ the approximate velocity field $\uu^n$ belongs to $\tilde L^\frac{8}{3}(0, T; \dot H^\frac{4}{3})$, for any time $T>0$. Thanks to Propositions \ref{prop:bid1} and \ref{prop:bid2},  $\uu^n\in \tilde L^\infty(0, T; L^2(\RR^2))$, with $\nabla \uu^n\in L^2(0, T; L^2(\RR^2))$. Furthermore, a standard interpolation yields
	\begin{equation*}
		2^{\frac{3}{4}q}\| \,\Dd_q\uu^n\,\|_{L^\frac{8}{3}(0, T; L^2(\RR^2))}
		\,\lesssim\,
		\| \,\Dd_q\uu^n\,\|_{L^\infty(0, T; L^2(\RR^2))}^\frac{1}{4}
		\left(
		2^{q}\| \,\Dd_q\uu^n\,\|_{L^2(0, T; L^2(\RR^2))}
		\right)^\frac{3}{4},
	\end{equation*}
	for any dyadic block of index $q\in\ZZ$. Taking the square of the above identity,  applying a Cauchy-Schwartz inequality and taking the sum for $q\in \ZZ$, we hence deduce that
	\begin{equation}\label{lemma-bid-8/3-ineq1}
	\begin{aligned}
		\|\,\uu^n\,\|_{\tilde L^\frac{8}{3}(0, T, \dot H^\frac{3}{4}(\RR^2))}^2\,
		&=\,
		\sum_{\q\in\ZZ}
		2^{\frac{3}{2}q}\| \,\Dd_q\uu^n\,\|_{L^\frac{8}{3}(0, T, L^2(\RR^2))}^2\\
		&\lesssim\,
		\nu^{ - \frac 38 }\sum_{\q\in\ZZ}
		\| \,\Dd_q\uu^n\,\|_{L^\infty(0, T, L^2(\RR^2))}^\frac{1}{4}
		\left(
		\nu^{\frac 12} 2^{q}\| \,\Dd_q\uu^n\,\|_{L^2(0, T, L^2(\RR^2))}
		\right)^\frac{3}{4}\\
		&\lesssim\,
		\nu^{-\frac 38}
		\left(
		 \|\,\uu^n\,\|_{\tilde L^\infty(0, T, L^2(\RR^2))}^2
		\,+\,
		\nu \|\,\nabla \uu^n\,\|_{L^2(0, T, L^2(\RR^2))}^2 
		\right),
	\end{aligned}
	\end{equation}
	therefore, applying Proposition \ref{prop:bid1} and Proposition \ref{prop:bid2}
	\begin{equation*}
	\begin{aligned}
		\|\uu^n&\|_{\tilde L^\frac{8}{3}(0, T, \dot H^\frac{3}{4}(\RR^2))}
		\leq\\
		&\leq
		\left\{
		\begin{alignedat}{4}
		&C\nu^{-\frac{3}{8}}
		\Big(
			\| \uu_0  \|_{L^2(\RR^2)}
			+
			\nu^{-1}\| \uu_0  \|_{L^2(\RR^2)}^2
			+
			\| \tau_0  \|_{L^2(\RR^2)}
			\sqrt{
			\frac{1-e^{-2aT}}{2a\nu}
			}
			+
			\nu^{-1}
			\| \tau_0  \|_{L^2(\RR^2)}^2
			\frac{1-e^{-2aT}}{2a\nu}
		\Big)\quad
		&&\text{if}\quad \mu = 0,\\
		&C\nu^{-\frac{3}{8}}
		\Big(
			\Big(1+\mu^{\frac{1}{2}}\sqrt{\frac{1-e^{-2aT}}{2a\nu}}\Big)
			\| \uu_0  \|_{L^2(\RR^2)}
			+
			\nu^{-1}\| \uu_0  \|_{L^2(\RR^2)}^2
			+\\
			&\hspace{2cm}		
			+
			\Big(\mu^{-\frac{1}{2}}
			+
			\sqrt{\frac{1-e^{-2aT}}{2a\nu}}
			\Big)
			\| \tau_0  \|_{L^2(\RR^2)}
			+
			\nu^{-1}
			\| \tau_0  \|_{L^2(\RR^2)}^2
			\frac{1-e^{-2aT}}{2a\nu}
		\Big)\quad
		&&\text{if}\quad \mu > 0.
	\end{alignedat}
	\right.
	\end{aligned}
	\end{equation*}
	We hence claim that $\uu^n\otimes \uu^n$ belongs to $\tilde L^\frac{4}{3}(0, T; \BB_{2,1}^\frac{1}{2})$. Indeed, 
	making use of the Bony decomposition
	\begin{equation}\label{lemma-bid-8/3-ineq2}
	\begin{aligned}
		\|\, \uu^n \otimes\uu^n\|_{\tilde L^\frac{4}{3}(0, T; \BB_{2,1}^\frac{1}{2})}
		\,&=\,
		\sum_q 2^{\frac q 2} \|\, \Dd_q ( \uu^n \otimes\uu^n)\,\|_{ L^\frac{4}{3}(0, T;L^2(\RR^2))} \\
		&\leq
		2\sum_q 2^{\frac q 2 }
		\underbrace{ \|\, \Dd_q ( \dot T_{\uu^n \otimes} \uu^n)\,\|_{ L^\frac{4}{3}(0, T;L^2(\RR^2))}}_{\Aa_q}
		\,+\,
		\sum_q
		2^{\frac 1 2 q} 
		\underbrace{ \|\, \Dd_q ( \dot R({\uu^n \otimes} ,\,\uu^n)\,\|_{ L^\frac{4}{3}(0, T;L^2(\RR^2))}}_{\Bb_q}.
	\end{aligned}
	\end{equation}
	We then control the first term by
	\begin{equation*}
		\begin{aligned}
			2^{\frac q2 }\Aa_q
			& \lesssim \sum_{| j - q | \leq 5}2^{\frac q 2 }\| \,\Sd_{j-1} \uu^n \|_{L^\frac{8}{3}(0, T; L^\infty(\RR^2))}
			\| \,\Dd_j \uu^n \|_{L^\frac{8}{3}(0, T; L^2(\RR^2))}\\
			&\lesssim \,
			\sum_{| j - q | \leq 5}2^{\frac q 2 }\| \,\Sd_{j-1} \uu^n \|_{L^\frac{8}{3}(0, T; L^\infty(\RR^2))}
			\| \,\Dd_j \uu^n \|_{L^\frac{8}{3}(0, T; L^2(\RR^2))}\\
			&\lesssim \,
			\sum_{| j - q | \leq 5}2^{-\frac j 4 }\| \,\Sd_{j-1} \uu^n \|_{L^\frac{8}{3}(0, T; L^\infty(\RR^2))}
			2^{\frac 34 j}\| \,\Dd_j \uu^n \|_{L^\frac{8}{3}(0, T; L^2(\RR^2))},
		\end{aligned}
	\end{equation*}
	Hence, taking the sum as $q\in \ZZ$ we gather that
	\begin{equation*}
	\begin{aligned}
		\sum_{q\in\ZZ}	2^{\frac q2 }\Aa_q
		&\lesssim\,
		\left(
			\sum_{j\in\ZZ}
				2^{- \frac j 2 }\| \,\Sd_{j-1} \uu^n \|_{L^\frac{8}{3}(0, T; L^\infty(\RR^2))}^2
		\right)^{\frac{1}{2}}
		\|\, \uu^n\,\|_{\tilde L^\frac{8}{3}(0, T; \dot H^\frac{3}{4}(\RR^2))}\\
		&\lesssim
		\left(
			\sum_{j\in\ZZ}
			\left|\,
			\sum_{\substack{k\in \ZZ\\ k\leq j-2}}
				2^{ \frac{k- j}{4} }2^{\frac{3}{4}k}\| \,\Dd_{k} \uu^n \|_{L^\frac{8}{3}(0, T; L^2(\RR^2))}
		\right|^2
		\right)^{\frac{1}{2}}
		\|\, \uu^n\,\|_{\tilde L^\frac{8}{3}(0, T; \dot H^\frac{3}{4}(\RR^2))},
	\end{aligned}
	\end{equation*}
	and  the Young inequality eventually leads to
	\begin{equation}\label{lemma-bid-8/3-ineq3}
		\sum_{q\in\ZZ}	2^{\frac q2 }\Aa_q
		\,\lesssim\,
		\|\, \uu^n\,\|_{\tilde L^\frac{8}{3}(0, T; \dot H^\frac{3}{4}(\RR^2))}^2.
	\end{equation}
	It remains to handle the term $\Bb_q$ of the homogeneous reminder. We proceed as follows:
	\begin{equation*}
		\begin{aligned}
			2^{\frac q2 }\Bb_q
			\,&\lesssim \,
			2^\frac{q}{2} \sum_{\substack{ j \geq q-5\\ |\eta|<1}}
			2^q
			\|\,\Dd_q( \Dd_j \uu^n\otimes\Dd_{j+\eta}\uu^n)\,\|_{L^\frac{4}{3}(0, T; L^1(\RR^2))}\\
			&\lesssim\,
			 \sum_{\substack{ j \geq q-5\\ |\eta|<1}}
			 2^{\frac{3}{2}(q-j)}
			2^{\frac{3}{4}j}\|\,\Dd_j \uu^n\,\|_{L^\frac{8}{3}(0, T; L^2(\RR^2))}
			2^{\frac{3}{4}(j+\eta)}\|\,\Dd_{j+\eta}\uu^n\,\|_{L^\frac{8}{3}(0, T; L^2(\RR^2))}.
		\end{aligned}
	\end{equation*}
	Applying again the Young inequality we finally deduce that
	\begin{equation}\label{lemma-bid-8/3-ineq4}
		\sum_{q\in\ZZ}2^{\frac q2 }\Bb_q
		\,\lesssim\,
		\|\, \uu^n\,\|_{\tilde L^\frac{8}{3}(0, T; \dot H^\frac{3}{4}(\RR^2))}^2.
	\end{equation}
	The lemma is then proven plugging inequalities \eqref{lemma-bid-8/3-ineq4} and \eqref{lemma-bid-8/3-ineq3}, 
	into \eqref{lemma-bid-8/3-ineq1} together with \eqref{lemma-bid-8/3-ineq2}.	
\end{proof}
\begin{remark}\label{bid-rmk}
	Recalling that  $\uu^{n}_1$ satisfies the following mild formulation
	\begin{equation*}
		\uu^{n}_1(t) := \int_0^t \Div\Pp {\rm e}^{\nu(t-s)\Delta}J_n(\uu^n(s)\otimes \uu^n(s))\dd s,
		\qquad 
		\text{with}\quad 
		\uu^n\otimes \uu^n \in \tilde L^\frac{4}{3}(0,T;\BB_{2,1}^\frac{1}{2})
	\end{equation*}
	we deduce that $\uu^{n}_1$ belongs also to 
	$\tilde L^\infty(0, T; \BB_{2, 1}^0)\cap \tilde L^\rho(0, T; \BB_{2, 1}^{2/\rho})$, for any 
	integer $n\in \NN$ and any $\rho \in [4/3,\infty)$. Furthermore, the following estimates holds true 
	\begin{equation*}
	\begin{aligned}
		\nu^{\frac 1 4}\|\,\uu^{n}_1\,&\|_{ \tilde L^\infty(0, T; \BB_{2, 1}^0)}
		\,+\, 	
		\nu^{\frac{1}{4}+\frac{1}{\rho}}
		\|\,\uu^{n}_1\,\|_{ \tilde L^\rho(0, T; \BB_{2, 1}^\frac{2}{\rho})} 		
		\,\lesssim\,
		\left\|\, \Div(\,\uu^n\otimes \uu^n\,) \,\right\|_{ \tilde L^\frac{4}{3}(0, T;  \BB_{2, 1}^{-\frac{1}{2}})}
		\\
		&\lesssim
		\nu^{-\frac 34}
		\left(
			\| \,\uu_0 \, \|_{L^2(\RR^2)}
			\,+\,
			\nu^{-1}\| \,\uu_0 \, \|_{L^2(\RR^2)}^2
			\,+\,
			\sqrt{\frac{1-e^{-2aT}}{2a\nu}}
			\| \,\tau_0 \, \|_{L^2(\RR^2)}
			\,+\,
			\nu^{-1}\frac{1-e^{-2aT}}{2a\nu}
			\| \,\tau_0 \, \|_{L^2(\RR^2)}^2
		\right)^2,
	\end{aligned}
	\end{equation*}
	for $\mu=0$, while 
	\begin{equation*}
	\begin{aligned}
		\nu^{\frac 1 4}\|\,\uu^{n}_1\,\|_{ \tilde L^\infty(0, T; \BB_{2, 1}^0)}
		\,&+\, 	
		\nu^{\frac{1}{4}+\frac{1}{\rho}}
		\|\,\uu^{n}_1\,\|_{ \tilde L^\rho(0, T; \BB_{2, 1}^\frac{2}{\rho})}
		\leq
		C\nu^{-\frac 34}
		\Big\{
			(1+\frac{\frac{1-e^{-2aT}}{2a\nu}}\mu^\frac{1}{2})
			\| \uu_0  \|_{L^2(\RR^2)}
			+\\ &+ 
			\nu^{-1}\| \uu_0  \|_{L^2(\RR^2)}^2		
		+
			(\mu^{-\frac{1}{2}}
			+
			+\frac{\frac{1-e^{-2aT}}{2a\nu}}
			\| \tau_0  \|_{L^2(\RR^2)}
			+
			\nu^{-1}
			+\frac{1-e^{-2aT}}{2a\nu}
			\| \tau_0  \|_{L^2(\RR^2)}^2
		\Big\}^2,
	\end{aligned}
	\end{equation*}
	for $\mu>0$.
\end{remark}
\noindent We now perform some suitable bounds for the component of the velocity field $\uu^{n}_2(t)$, which is defined as
\begin{equation*}
		\uu_2^{n}(t) =  \int_0^t \Div\,\Pp \,{\rm e}^{\nu (t-s) \Delta }J_n \tau^n(s)  \dd s.
\end{equation*}
\begin{lemma}\label{control-of-u2} For any positive integer $n\in\NN$ the following class affinity holds true
\begin{equation*}
\begin{alignedat}{16}
		&\uu^n_2 \,\in \,\tilde L^\infty(0, T; \BB_{\infty, 2}^0)&&\cap \tilde L^1(0, T; \BB_{\infty, 2}^2), \\
		&\uu^n_2 \,\in \,\tilde L^\infty(0, T; \BB_{2, 1}^0)&&\cap \tilde L^\frac{4}{3}(0, T; \BB_{2, 1}^\frac{1}{2}).
\end{alignedat}
\end{equation*}
Furthermore
\begin{equation*} 
\begin{alignedat}{16}
			&\|\,\uu_2^{n}\,\|_{\tilde L^\infty(0, T; \BB_{\infty, 2}^0)}
			\,&&+\,
			\nu 
			&&&&\|\,\uu_2^{n}\,\|_{\tilde L^1(0, T; \BB_{\infty, 2}^2)}
			\,&&&&&&&&\leq\,
			C\| \,\tau^n\,\|_{L^1(0, T; L^2(\RR^2))},\\
			&\|\,\uu_2^{n}\,\|_{\tilde L^\infty(0, T; \BB_{2, 1}^0)}
			\,&&+\,
			\nu^{\frac{3}{4}}
			&&&&\|\, \uu^n_2\, \|_{\tilde L^\frac{4}{3}(0, T; \BB_{2, 1}^\frac{3}{2})}
			\,&&&&&&&&\leq\,
			C\|\, \tau^n \,\|_{L^1(0, T; \BB_{\infty, 1}^0)},
\end{alignedat}
\end{equation*}
for a suitable positive constant $C$.
\end{lemma}
\begin{proof}
	We restrict ourselves in proving the first statement, that is $\uu_2^{n}$ belongs to 	$\tilde L^\infty(0, T; \BB_{\infty, 2}^0)\cap \tilde L^1(0, T; \BB_{\infty, 2}^2)$ and satisfies
	\begin{equation*}
		\|\,\uu_2^{n}(t)\,\|_{\tilde L^\infty(0, T; \BB_{\infty, 2}^0)}
		\,+\,
		\nu 
		\|\,\uu_2^{n}\,\|_{\tilde L^1(0, T; \BB_{\infty, 2}^2)}
		\lesssim
		\| \,\tau^n\,\|_{L^1(0, T; L^2(\RR^2))}.
	\end{equation*}
	The second part of the Lemma can indeed be achieved with a similar procedure. Applying the dyadic bloc $\Dd_q$ on $\uu^{n}_2(t)$ and taking the $L^\infty(\RR^2)$ norm, one has
	\begin{equation*}
	\begin{aligned}
		\|\,\Dd_q u_2^{n}\,\|_{L^\infty(0, T; L^\infty(\RR^2))}
		\,\lesssim\,
			\int_0^T e^{-c(t-s)2^{2q}}2^q \| \,\Dd_q \tau^n(s)\,\|_{L^\infty(\RR^2)}\dd s \lesssim
			\int_0^T  \| \,\Dd_q \tau^n(s)\,\|_{L^2(\RR^2)}\dd s
	\end{aligned}
	\end{equation*}
	from which
	\begin{equation*}
	\begin{aligned}
		\|\,\uu_2^{n}\,\|_{\tilde L^\infty(0, T; \BB_{\infty, 2}^0)}
		\,=\,
		\left(\sum_{q\in \ZZ}  \| \,\Dd_q\uu_2^{n}(t)\,\|_{L^\infty(0, T; L^2(\RR^2))}^2\right)	^\frac{1}{2}
		\,\lesssim\,
		\|\,\tau^{n}\,\|_{\tilde L^1(0, T; \BB_{2, 2}^0)}
		\,\lesssim\,
		\|\,\tau^{n}\,\|_{L^1(0, T; L^2(\RR^2))}.
	\end{aligned}
	\end{equation*}
	Furthermore
	\begin{equation*}
	\begin{aligned}
		2^{2q}\|\,\Dd_q \uu_2^{n}\,\|_{L^\infty(\RR^2)}
		\,\lesssim\,
		\int_0^t 2^{2q} e^{-c\nu(t-s)2^{2q}}2^q \| \,\Dd_q \tau^n(s)\,\|_{L^\infty(\RR^2)}\dd s
		\,\lesssim \,
		\int_0^t 2^{2q} e^{-c\nu(t-s)2^{2q}}\| \,\Dd_q \tau^n(s)\,\|_{L^2(\RR^2)}\dd s,		
	\end{aligned}
	\end{equation*}
	thus applying the Young inequality we gather that
	\begin{equation*}
		2^{2q}\|\,\Dd_q \uu_2^{n}\,\|_{L^1(0, T; L^\infty(\RR^2))}
		\lesssim
		\nu^{-1}
		\| \,\Dd_q \tau^n \|_{L^1(0, T; L^2(\RR^2))}
	\end{equation*}
	and taking the sum as $q\in \ZZ$
	\begin{equation*}
		\nu \|\,\uu_2^{n}\,\|_{\tilde L^1(0, T; \BB^2_{\infty, 2})}
		\lesssim 
		\|\, \tau^n \|_{L^1(0, T; L^2(\RR^2))}.
	\end{equation*}
\end{proof}

\begin{prop}\label{prop-almost-there}
	For any positive integer $n\in\NN$ and for any positive time $T>0$, the velocity field $\uu^n$ satisfies the class affinity
	\begin{equation*}
		\Div\,(\,\uu^n\otimes \uu^n\,)  \in L^1(0, T; \BB_{\infty, 1}^{-1}).
	\end{equation*}
	Furthermore, the following bound holds true
	\begin{equation*}
	\begin{aligned}
		\|\, \Div(\,\uu^n\otimes \uu^n \,)\,\|_{ L^1(0, T; \BB_{\infty, 1}^{-1})}
		\,\leq\,
		\Psi_{1,\nu,\,\mu}(T,\,\uu_0,\,\tau_0)
		\,	+\,
		\Psi_{2,\,\nu\,\mu}(T,\,\uu_0,\,\tau_0)
		\left\|
		\,
			\tau^n
		\,
		\right\|_{L^1(0, T; \BB_{\infty, 1}^0)},
		\end{aligned}
	\end{equation*}
	for a suitable positive constant $C$, where the smooth-in-time functions $\Psi_{1,\nu}(T,\,\uu_0,\,\tau_0)$ and $\Psi_{2,\nu}(T,\,\uu_0,\,\tau_0)$ are defined by
	\begin{equation*}
	\begin{aligned}
			\Phi_{\nu,\mu}(T, \uu_0,\tau_0)&= \\
			&\hspace{-2cm}=
			\begin{cases}
		\left(
			\| \,\uu_0 \, \|_{L^2(\RR^2)}
			\,+\,
			\nu^{-1}\| \,\uu_0 \, \|_{L^2(\RR^2)}^2
			\,+\,
			\| \,\tau_0 \, \|_{L^2(\RR^2)}
			\sqrt{\frac{1-e^{-2aT}}{2a\nu}}
			\,+\,
			\nu^{-1}
			\| \,\tau_0 \, \|_{L^2(\RR^2)}^2
			\frac{1-e^{-2aT}}{2a\nu}
		\right)^2\quad &\text{if}\quad \mu = 0,\\
		\Big\{
			\big(1+\sqrt{\frac{1-e^{-2aT}}{2a\nu}}\mu^\frac{1}{2}\big)
			\| \uu_0  \|_{L^2(\RR^2)}
			+ 
			\nu^{-1}\| \uu_0  \|_{L^2(\RR^2)}^2		
			+ \\ \hspace{2.3cm}+
			\big(\mu^{-\frac{1}{2}}
			+
			\sqrt{\frac{1-e^{-2aT}}{2a\nu}}
			\big)
			\| \tau_0  \|_{L^2(\RR^2)}
			+
			\nu^{-1}
			\| \tau_0  \|_{L^2(\RR^2)}^2
			\frac{1-e^{-2aT}}{2a\nu}
		\Big\}^2	
		 &\text{if}\quad \mu > 0,
			\end{cases}\\
			\Psi_{1,\nu,\,\mu}(T,\,\uu_0,\,\tau_0)
			\,&=\,
			C
		\bigg\{\,
			\nu^{-\frac 32}
			\Phi_\nu(T, \,\uu_0, \,\tau_0\,)^2
			\,+\,
			\nu^{-\frac{5}{4}}
			\Phi_\nu(T, \,\uu_0, \,\tau_0\,)
			\|\,\uu_0\,\|_{L^2(\RR^2)}
		\bigg\},			\\
		\Psi_{2,\,\nu,\,\mu}(T,\,\uu_0,\,\tau_0)
		\,&=\,
		C
		\left\{
			\sqrt{\frac{1-e^{-2aT}}{2a\nu}}
			\big(\mu\|\uu_0\|_{L^2(\RR^2)}^2+\|\,\tau_0\,\|_{L^2(\RR^2)}^2
			\big)^\frac{1}{2}
			\,+\,
			\nu^{-\frac{5}{4}}\Phi_{\nu,\mu}(T, \,\uu_0, \,\tau_0\,)
			\,+\,
			\nu^{-1}
			\|\,\uu_0\,\|_{L^2(\RR^2)}
		\right\},
	\end{aligned}
	\end{equation*}
	for a suitable positive constant $C$, when $\mu = 0$.
\end{prop}	
\begin{proof}
	 We are now in the position to deal with $\Div(\,\uu^n\otimes\uu^n\,)$ in the 
	functional space $L^1(0, T; \BB_{\infty, 1}^{-1})$. We keep on using the standard decomposition 
	$\uu^n= \uu^n_1\,+\,\uu^n_2\,+\, \uu^L$, thus our estimate reduces to	
	\begin{equation*}
	\begin{aligned}
		\left\| 
			\,
			\Div(\,\uu^n\otimes\uu^n\,)
			\,
		\right\|_{L^1(0, T; \BB_{\infty, 1}^{-1})}
		\leq
		\left\| 
			\,
			\Div(\,\uu^n_1\otimes\uu^n_1\,)
			\,
		\right\|_{L^1(0, T; \BB_{\infty, 1}^{-1})}
		+
		\left\| 
			\,
			\uu^L\cdot \nabla \uu^n_1
			\,
		\right\|_{L^1(0, T; \BB_{\infty, 1}^{-1})}
		+
		\left\| 
			\,
			\uu^n_1\cdot \nabla \uu^L
			\,
		\right\|_{L^1(0, T; \BB_{\infty, 1}^{-1})} + \\		
		+\,
			\left\| 
			\,
			\uu^n_1\cdot \nabla \uu^n_2
			\,
		\right\|_{L^1(0, T; \BB_{\infty, 1}^{-1})}
		\,+\,
			\left\| 
			\,
			\uu^n_2\cdot \nabla \uu^n_1\,
			\,
		\right\|_{L^1(0, T; \BB_{\infty, 1}^{-1})}
		\,+\,
		\left\| 
			\,
			\Div(\,\uu^n_2\otimes\uu^n_2\,)
			\,
		\right\|_{L^1(0, T; \BB_{\infty, 1}^{-1})}
		\,+\\
		+\,
		\left\| 
			\,
			\uu^L\cdot \nabla \uu^n_2
			\,
		\right\|_{L^1(0, T; \BB_{\infty, 1}^{-1})}
		\,+\,
		\left\| 
			\,
			\uu^n_2\cdot \nabla \uu^L
			\,
		\right\|_{L^1(0, T; \BB_{\infty, 1}^{-1})}
		\,+\,
		\left\| 
			\,
			\uu^L\cdot \nabla \uu^L
			\,
		\right\|_{L^1(0, T; \BB_{\infty, 1}^{-1})}.
	\end{aligned}
	\end{equation*}
	We hence control any term on the right-hand side of the above inequality. Thanks to the embedding 
	$\BB_{2,1}^0\hookrightarrow \BB_{\infty,1}^{-1}$, we first remark that
	\begin{equation*}
		\left\| 
			\,
			\Div(\,\uu^n_1\otimes\uu^n_1\,)
			\,
		\right\|_{L^1(0, T; \BB_{\infty, 1}^{-1})}
		\,\lesssim\,
		\left\| 
			\,
			\Div(\,\uu^n_1\otimes\uu^n_1\,)
			\,
		\right\|_{L^1(0, T; \BB_{2, 1}^{0})}
		\,\lesssim\,
		\left\| 
			\,
			\uu^n_1\otimes\uu^n_1
			\,
		\right\|_{L^1(0, T; \BB_{2, 1}^{1})}
		\,\lesssim\,
		\left\| 
			\,
			\uu^n_1
			\,
		\right\|_{L^2(0, T; \BB_{2, 1}^{1})}^2,
	\end{equation*}
	hence thanks to the Remark \ref{bid-rmk}, we obtain that for $\mu = 0$
	\begin{equation*}
			\left\| 			
			\Div(\,\uu^n_1\otimes\uu^n_1\,)			
		\right\|_{L^1(0, T; \BB_{\infty, 1}^{-1})}
		\lesssim
		\nu^{-\frac 32}
		\Big(
			\| \uu_0  \|_{L^2(\RR^2)}
			+
			\nu^{-1}\| \uu_0  \|_{L^2(\RR^2)}^2
			+\sqrt{\frac{1-e^{-2aT}}{2a\nu}}
			\| \,\tau_0 \, \|_{L^2(\RR^2)}
			+
			\frac{1-e^{-2aT}}{2a\nu^2}
			\| \tau_0  \|_{L^2(\RR^2)}^2
		\Big)^4
	\end{equation*}
	while for $\mu>0$
	\begin{equation*}
	\begin{aligned}
		\left\| 
			\,
			\Div(\,\uu^n_1\otimes\uu^n_1\,)
			\,
		\right\|_{L^1(0, T; \BB_{\infty, 1}^{-1})}&
		\,\lesssim\,
		\nu^{-\frac 32}
		\Big\{
			(1+\sqrt{\frac{1-e^{-2aT}}{2a\nu}}\mu^\frac{1}{2})
			\| \uu_0  \|_{L^2(\RR^2)}
			+\\ &+ 
			\nu^{-1}\| \uu_0  \|_{L^2(\RR^2)}^2		
		+
			(\mu^{-\frac{1}{2}}
			+
			\sqrt{\frac{1-e^{-2aT}}{2a\nu}})
			\| \tau_0  \|_{L^2(\RR^2)}
			+
			\nu^{-1}\frac{1-e^{-2aT}}{2a\nu}
			\| \tau_0  \|_{L^2(\RR^2)}^2
		\Big\}^4.
	\end{aligned}
	\end{equation*}
	Now, recalling the embedding $\BB_{2, 1}^{0}\hookrightarrow \BB_{\infty, 1}^{-1}$ in dimension two, 
	together with the continuity of the product within 
	\begin{equation}
		\tilde L^4(0, T; \BB_{2, 2}^{\frac{1}{2}})
		\,\times\,
		\tilde L^\frac{4}{3}(0, T; \BB_{2, 2}^{\frac{1}{2}})
		\,\rightarrow\,
		L^1(0, T; \BB_{2, 1}^{0}),
	\end{equation}	 
	we gather for $\mu = 0$ that
	\begin{equation*}
		\begin{aligned}
			\Big\|
			\uu^n_1\cdot \nabla \uu^L
		&\Big\|_{L^1(0, T; \BB_{\infty, 1}^{-1})} 
		\lesssim
		\left\| 
			\uu^n_1\cdot \nabla \uu^L
		\right\|_{L^1(0, T; \BB_{2, 1}^{0})} 
		\,\lesssim\,
		\|\uu_1^n\|_{\tilde L^4(0, T; \BB_{2, 2}^{\frac{1}{2}})}
		\|\nabla \uu^L\|_{\tilde L^\frac{4}{3}(0, T; \BB_{2, 2}^{\frac{1}{2}})}
		\lesssim
		\|\uu_1^n\|_{\tilde L^4(0, T; \BB_{2, 1}^{\frac{1}{2}})}
		\|\uu^L\|_{\tilde L^\frac{4}{3}(0, T; \BB_{2, 2}^{\frac{3}{2}})}\\
		&\lesssim\,
		\nu^{-\frac 54}
		\left(
			\|\uu_0 \|_{L^2(\RR^2)}
			+
			\nu^{-1}\| \uu_0  \|_{L^2(\RR^2)}^2
			+
			\sqrt{\frac{1-e^{-2aT}}{2a\nu}}
			\| \tau_0  \|_{L^2(\RR^2)}
			+
			\nu^{-1}\frac{1-e^{-2aT}}{2a\nu}
			\| \tau_0  \|_{L^2(\RR^2)}^2
		\right)^2
		\|\uu_0\|_{L^2(\RR^2)},
		\end{aligned}
	\end{equation*}
	while for $\mu>0$
	\begin{equation*}
		\begin{aligned}
			\left\|
			\uu^n_1\cdot \nabla \uu^L
		\right\|_{L^1(0, T; \BB_{\infty, 1}^{-1})} 
		&\lesssim\,
		\nu^{-\frac 54}
		\Big\{
			(1+\sqrt{\frac{1-e^{-2aT}}{2a\nu}}\mu^\frac{1}{2})
			\| \uu_0  \|_{L^2(\RR^2)}
			+\\ &+ 
			\nu^{-1}\| \uu_0  \|_{L^2(\RR^2)}^2		
		+
			(\mu^{-\frac{1}{2}}
			+
			\sqrt{\frac{1-e^{-2aT}}{2a\nu}}
			\| \tau_0  \|_{L^2(\RR^2)}
			+
			\nu^{-1}
			\| \tau_0  \|_{L^2(\RR^2)}^2
			\frac{1-e^{-2aT}}{2a\nu}
		\Big\}^2
		\|\uu_0\|_{L^2(\RR^2)}.
		\end{aligned}
	\end{equation*}
	Similarly, the following bound holds true for $\mu = 0$
	\begin{equation*}
		\begin{aligned}
			\Big\| 
			\,
			&\uu^L\cdot \nabla \uu^n_1
			\,
		\Big\|_{L^1(0, T; \BB_{\infty, 1}^{-1})}
		\,\lesssim\,
		\left\| 
			\,
			\uu^L\cdot \nabla \uu^n_1
			\,
		\right\|_{L^1(0, T; \BB_{2, 1}^{0})}
		\,\lesssim\,
		\| \, \uu^L\,\|_{\tilde L^4(0, T; \BB_{2, 2}^\frac{1}{2})}
		\|\, \nabla \uu^n_1\,\|_{\tilde L^\frac{4}{3}(0, T; \BB_{2, 2}^\frac{1}{2})}\\
		\,&\lesssim\,
		\| \, \uu^L\,\|_{\tilde L^4(0, T; \BB_{2, 1}^\frac{1}{2})}
		\|\, \nabla \uu^n_1\,\|_{\tilde L^\frac{4}{3}(0, T; \BB_{2, 1}^\frac{1}{2})}\\
		&\lesssim\,
		\nu^{-\frac{5}{4}}\|\,\uu_0\,\|_{L^2(\RR^2)}
		\left(
			\|\uu_0 \|_{L^2(\RR^2)}
			+
			\nu^{-1}\| \uu_0  \|_{L^2(\RR^2)}^2
			+
			\sqrt{\frac{1-e^{-2aT}}{2a\nu}}
			\| \tau_0  \|_{L^2(\RR^2)}
			+
			\nu^{-1}\frac{1-e^{-2aT}}{2a\nu}
			\| \tau_0  \|_{L^2(\RR^2)}^2
		\right)^2,
		\end{aligned}
	\end{equation*}
	while for $\mu>0$
	\begin{equation*}
		\begin{aligned}
			\left\| 
			\,
			\uu^L\cdot \nabla \uu^n_1
			\,
		\right\|_{L^1(0, T; \BB_{\infty, 1}^{-1})}
		&\lesssim\,
		\nu^{-\frac{5}{4}}\|\,\uu_0\,\|_{L^2(\RR^2)}
		\Big\{
			(1+\sqrt{\frac{1-e^{-2aT}}{2a\nu}}\mu^\frac{1}{2})
			\| \uu_0  \|_{L^2(\RR^2)}
			+\\ &+ 
			\nu^{-1}\| \uu_0  \|_{L^2(\RR^2)}^2		
		+
			(\mu^{-\frac{1}{2}}
			+
			\sqrt{\frac{1-e^{-2aT}}{2a\nu}}
			\| \tau_0  \|_{L^2(\RR^2)}
			+
			\nu^{-1}
			\| \tau_0  \|_{L^2(\RR^2)}^2
			\frac{1-e^{-2aT}}{2a\nu}
		\Big\}^2,
		\end{aligned}
	\end{equation*}
	We now take into account 
	\begin{equation*}
		\|\,\Div(\uu^n_2\otimes\uu^n_2)\,\|_{L^1(0, T; \BB^{-1}_{\infty, 1})}
		\,\lesssim\,
		\|\,\uu^n_2\otimes\uu^n_2\,\|_{L^1(0, T; \BB^{0}_{\infty, 1})},
	\end{equation*}
	hence, applying Lemma \ref{lemma-product-f^2}, together with inequality \eqref{ineq_LinftyL2_of_u2},  we gather
	\begin{equation*}
	\begin{aligned}
		\|\Div(\uu^n_2\otimes\uu^n_2)\|_{L^1(0, T; \BB^{-1}_{\infty, 1})}
		&\lesssim 
		\|\uu^n_2\|_{\tilde L^\infty (0, T; L^2(\RR^2))} \| \uu^n_2 \|_{L^1(0, T; \BB_{\infty, 1}^1)}\\ 		
		&\lesssim 
		\begin{cases}
		\sqrt{\frac{1-e^{-2aT}}{2a\nu}} \|\,\tau_0\,\|_{L^2(\RR^2)}\| \tau^n \|_{L^1(0, T; \BB_{\infty, 1}^0)}\qquad
		&\mu = 0,
		\\
		\sqrt{\frac{1-e^{-2aT}}{2a\nu}}
		\big(\mu\|\uu_0\|_{L^2(\RR^2)}^2+\|\,\tau_0\,\|_{L^2(\RR^2)}^2
		\big)^\frac{1}{2}
		\| \tau^n \|_{L^1(0, T; \BB_{\infty, 1}^0)}
		&\mu > 0.
		\end{cases}
	\end{aligned}
	\end{equation*}
	Next, we estimate the following term
	\begin{equation*}
	\begin{aligned}
	\left\| 
			\uu^n_1\cdot \nabla \uu^n_2
	\right\|_{L^1(0, T; \BB_{\infty, 1}^{-1})}
	\,\lesssim\,
	\left\| 
			\,
			\uu^n_1\cdot \nabla \uu^n_2
			\,
	\right\|_{L^1(0, T; \BB_{2, 1}^{0})}
	\,\lesssim\,
	\|\,\uu^n_1\,\|_{\tilde L^4(0, T; \BB_{2, 2}^{\frac{1}{2}})}
	\|\,\nabla \uu^n_2\,\|_{\tilde L^\frac{4}{3}(0, T; \BB_{2, 2}^{\frac{1}{2}})}	
	\end{aligned}
	\end{equation*}
	Hence, applying Lemma \ref{control-of-u2} and Remark \ref{bid-rmk}
	\begin{equation*}		
	\begin{aligned}
	\| 
			\,
			\uu^n_1\cdot \nabla& \uu^n_2
			\,
	\|_{L^1(0, T; \BB_{\infty, 1}^{-1})}
	\,\lesssim\,
	\nu^{-\frac{5}{4}}
	\Big(
			\| \,\uu_0 \, \|_{L^2(\RR^2)}
			\,+\,\\
			&+
			\nu^{-1}\| \,\uu_0 \, \|_{L^2(\RR^2)}^2
			\,+\,
			\sqrt{\frac{1-e^{-2aT}}{2a\nu}}
			\| \,\tau_0 \, \|_{L^2(\RR^2)}
			\,+\,
			\nu^{-1}
			\frac{1-e^{-2aT}}{2a\nu}
			\| \,\tau_0 \, \|_{L^2(\RR^2)}^2
		\Big)^2
	\|\, \tau^n \,\|_{L^1(0, T; \BB_{\infty, 1}^0)},
	\end{aligned}
	\end{equation*}
	for $\mu = 0$, while 
	\begin{equation*}		
	\begin{aligned}
	\| 
			\,
			\uu^n_1\cdot \nabla& \uu^n_2
			\,
	\|_{L^1(0, T; \BB_{\infty, 1}^{-1})}
	\lesssim\,
	\nu^{-\frac{5}{4}}
	\Big\{
			(1+\sqrt{\frac{1-e^{-2aT}}{2a\nu}}\mu^\frac{1}{2})
			\| \uu_0  \|_{L^2(\RR^2)}
			+\\ &+ 
			\nu^{-1}\| \uu_0  \|_{L^2(\RR^2)}^2		
		+
			(\mu^{-\frac{1}{2}}
			+
			\sqrt{\frac{1-e^{-2aT}}{2a\nu}}
			)
			\| \tau_0  \|_{L^2(\RR^2)}
			+
			\nu^{-1}
			\frac{1-e^{-2aT}}{2a\nu}
			\| \tau_0  \|_{L^2(\RR^2)}^2
		\Big\}^2
	\|\, \tau^n \,\|_{L^1(0, T; \BB_{\infty, 1}^0)},
	\end{aligned}
	\end{equation*}
	for $\mu>0$. A similar estimate is also valid for $\uu^n_2\cdot \nabla \uu^n_1$. 
	We now estimate the following term
	\begin{equation*}
	\begin{aligned}
	\left\| 
			\,
			\uu^L\cdot \nabla \uu^n_2
			\,
	\right\|_{L^1(0, T; \BB_{\infty, 1}^{-1})}
	\,\lesssim\,
	\left\| 
			\,
			\uu^L\cdot \nabla \uu^n_2
			\,
	\right\|_{L^1(0, T; \BB_{2, 1}^{0})}
	\,\lesssim\,
	\|\,\uu^L\,\|_{\tilde L^4(0, T; \BB_{2, 2}^{\frac{1}{2}})}
	\|\,\nabla \uu^n_2\,\|_{\tilde L^\frac{4}{3}(0, T; \BB_{2, 2}^{\frac{1}{2}})}
	\end{aligned}
	\end{equation*}
	thus, applying Lemma \ref{control-of-u2}
	\begin{equation*}
	\left\| 
			\,
			\uu^L\cdot \nabla \uu^n_2
			\,
	\right\|_{L^1(0, T; \BB_{\infty, 1}^{-1})}
	\,\lesssim\,
	\nu^{-1}\|\,\uu_0\,\|_{L^2(\RR^2)}
	\|\, \tau^n \,\|_{L^1(0, T; \BB_{\infty, 1}^0)}.
	\end{equation*}
	Similarly
	\begin{equation*}
	\begin{aligned}
		\left\| 
			\,
			\uu^n_2\cdot \nabla \uu^L
			\,
	\right\|_{L^1(0, T; \BB_{\infty, 1}^{-1})}
	\,&\lesssim\,
	\left\| 
			\,
			\uu^n_2\cdot \nabla \uu^L
			\,
		\right\|_{L^1(0, T; \BB_{2, 1}^{0})}\\
		\,&\lesssim\,
		\| \, \uu^n_2\,\|_{\tilde L^4(0, T; \BB_{2, 1}^\frac{1}{2})}
		\|\, \nabla \uu^L\,\|_{\tilde L^\frac{4}{3}(0, T; \BB_{2, 2}^\frac{1}{2})}\\
		&\lesssim
		\,\nu^{-1}
		\|\,\uu_0\,\|_{L^2(\RR^2)}
		\|\, \tau^n \,\|_{L^1(0, T; \BB_{\infty, 1}^0)}
	\end{aligned}
	\end{equation*}
	Finally, we have that
	\begin{equation*}
			\left\| 
			\,
			\uu^L\cdot \nabla \uu^L
			\,
			\right\|_{L^1(0, T; \BB_{\infty, 1}^{-1})}
		\,\lesssim\,	
		\| \, \uu^L\,\|_{\tilde L^4(0, T; \BB_{2, 1}^\frac{1}{2})}
		\|\, \nabla \uu^L\,\|_{\tilde L^\frac{4}{3}(0, T; \BB_{2, 2}^\frac{1}{2})}
		\,\lesssim\,
		\nu^{-1}\|\,\uu_0\,\|_{L^2}^2.
	\end{equation*}
	Summarizing the above inequalities, we can then conclude the proof of the proposition.
\end{proof}
\noindent
We are now in the condition to prove that the velocity field is Lipschitz in space.
\begin{theorem}\label{thm-Idontknow}
		For any index $n\in\NN$ the velocity field $\uu^n$ belongs to the functional space $L^\infty(0, T; \BB_{\infty, 1}^{-1})\cap L^1(0, T; \BB_{\infty, 1}^1)$, for any time $T>0$. Furthermore the following inequalities hold true
		\begin{equation*}
		\begin{aligned}
			\nu
			\|\,\uu^n\,\|_{L^1(0, T;\BB_{\infty, 1}^{1} )}
			\,&\leq\,
			\Upsilon^1_{\nu,\mu}(T,\, \uu_0,\,\tau_0),\\
			\|\,\uu^n\,\|_{L^\infty(0, T;\BB_{\infty, 1}^{-1} )}
			\,&\leq\,
			\Upsilon^2_{\nu,\mu}(T,\, \uu_0,\,\tau_0).
		\end{aligned}
		\end{equation*}
\end{theorem}

\begin{proof}
	Recalling the mild formulation of the velocity field $\uu^{n}(t)$
	\begin{equation*}
			\uu^{n}(t)\,=\,{\rm e}^{\nu t \Delta}J_n\uu_0
			\,+\,
			\int_0^t \mathcal{P}J_n{\rm e}^{\nu(t-s)\Delta}
						\Div(\,\uu^n(s)\otimes\uu^n(s)\,)\dd s
			\,+\,
			\int_0^t \mathcal{P}J_n{\rm e}^{\nu(t-s)\Delta}
						\Div\,\tau^{n}(s)\dd s,
	\end{equation*}		
	hence
	\begin{equation*}
	\begin{aligned}
		\nu  \|\,\uu^{n}\,\|_{L^1(0,T; \BB_{\infty, 1}^{1})}
		\,\leq\,
		\|\,\uu_0\,\|_{\BB_{\infty,1}^{-1}}
		\,+\,
		C
		\|\,\Div(\,\uu^n\otimes \uu^n )\,\|_{L^1(0, T; \BB_{\infty, 1}^{-1})}
		\,+\,
		C
		\|\,\Div\, \tau^n \,\|_{L^1(0, T; \BB_{\infty, 1}^{-1})}.
	\end{aligned}
	\end{equation*}
	Next, applying Proposition \ref{prop-almost-there} we deduce that
	\begin{equation*}
	\begin{aligned}
		\|\,\uu^{n}\,\|_{L^\infty(0,T; \BB_{\infty, 1}^{-1})}
		+
		\nu
		\|\,\uu^{n}\,\|_{L^1(0,T; \BB_{\infty, 1}^{1})}
		\,\leq\,
		\|\,\uu_0\,\|_{\BB_{\infty,1}^{-1}}
		+
		\Psi_{1,\nu,\,\mu}(T,\,\uu_0,\,\tau_0)
		+
		\left(\,
		\Psi_{2,\,\nu\,\mu}(T,\,\uu_0,\,\tau_0)
		+
		C
		\right)
		\left\|
		\,
			\tau^n
		\,
		\right\|_{L^1(0, T; \BB_{\infty, 1}^0)}.
	\end{aligned}
	\end{equation*}
	Hence, Lemma \ref{lemma:bound_of_tau-linear_on_nablau} allows us to gather
	\begin{equation*}	
	\begin{aligned}
		\|\uu^{n}\|_{L^\infty(0,T; \BB_{\infty, 1}^{-1})}
		+
		\nu
		\|\uu^{n}\|_{L^1(0,T; \BB_{\infty, 1}^{1})}
		\,\leq\,
		\|\uu_0\|_{\BB_{\infty,1}^{-1}}
		+
		\Psi_{1,\nu,\mu}(T,\uu_0,\tau_0)
		+
		C
		\Big(
		\Psi_{2,\,\nu,\,\mu}(T,\uu_0,\tau_0)
		+
		1
		\Big)
		\|\tau_0\|_{ \BB_{\infty,1}^0}
		\frac{1-e^{-aT}}{a}+\\
		+
		C
		\Psi_{2,\,\nu,\,\mu}(T,\,\uu_0,\,\tau_0)
		\int_0^T
		\big(\mu + 
		\|\tau_0\|_{ \BB_{\infty,1}^0}e^{-at}
		\big)
		\|\uu\|_{L^1(0, t; \BB_{\infty, 1}^{1})}
		\dd t+
		C\mu 
		\Psi_{2,\,\nu,\,\mu}(T,\,\uu_0,\,\tau_0)
		\int_0^T
		\|\uu\|_{L^1(0, t; \BB_{\infty, 1}^{1})}^2
		\dd t.
	\end{aligned}
	\end{equation*}
	Applying the Gronwall inequality we hence deduce that
	\begin{equation*}
		\nu
		\|\,\uu^{n}\,\|_{L^1(0,T; \BB_{\infty, 1}^{1})}
		\,\leq\,
		\Upsilon_{\nu}^1(T,\, \uu_0,\,\tau_0),\qquad 
		\text{for any}\quad T\in [0,T_{\rm max}),
	\end{equation*}
	where the smooth function $\Upsilon_{\nu}^1(T,\, \uu_0,\,\tau_0)$ is defined in Definition \eqref{def-Upsilon} together with $T_{\rm max}$. The lifespan $T_{max}$ is $+\infty$ when $\mu = 0$.
	
\end{proof}

\subsection{\bf Propagation of $\BB_{p,1}^{\frac{2}{p}-1}\times \BB_{p,1}^{\frac{2}{p}}$-regularity with index $1\leq p< 2$}.

\noindent
In the previous section we establish the propagation of a Lipschitz-in-space regularity for any approximate velocity fields $\uu^n$. We now use this criterion to propagate higher regularity given by the initial condition
\begin{equation}\label{initial-condtiion-Bp1}
	\uu_0\in \BB_{p,1}^{\frac{2}{p}-1}\quad\quad\text{and}\quad\quad \tau_0\in\BB_{p,1}^{\frac{2}{p}},
\end{equation}
where $p$ is assumed in $[1, \infty)$.

\begin{theorem}
	For any index $n\in \NN$ and for any time $t\in [0,T)$ the following estimate for the conformation tensor $\tau^n$ holds true
	\begin{equation}\label{ineq-tauBp}
			\|\,\tau^n(t)\,\|_{\BB_{p,1}^{\frac{2}{p}}}
			\,\leq\,
			\|\,\tau_0\,\|_{\BB_{p,1}^{\frac{2}{p}}}\exp\bigg\{\,C\nu^{-1}\,\Upsilon_{1,\nu}(T, \,\uu_0,\,\tau_0)\,\bigg\},
	\end{equation}		
	for a suitable positive constant $C>0$. Furthermore, the approximate velocity field $\uu^n$ satisfies 
	\begin{equation}\label{ineq-uuBp}
		\|\,\uu^n(t)\,\|_{\BB_{p,1}^{\frac{2}{p}-1}}
		\,+\,
		\int_0^t
		\|\,\uu^n(s)\,\|_{\BB_{p,1}^{\frac{2}{p}+1}}
		\dd s
		\,
		\,\leq\,
		\|\,\uu_0\,\|_{\BB_{p,1}^{\frac{2}{p}-1}}
		\exp\bigg\{C\nu^{-2}\Upsilon^1_\nu(T,\,\uu_0,\,\tau_0)\Upsilon^2_\nu(T,\,\uu_0,\,\tau_0)\bigg\}.
	\end{equation}
\end{theorem}
\begin{proof}
	We recall that $\tau^n$ satisfies the equation
	\begin{equation*}
		\partial_t \tau^n\,+\,\uu^n\cdot \nabla \tau^n - \omega^n\tau^n\,+\,\tau^n\omega^n\,+a\tau^n =\,0,
	\end{equation*}
	hence, applying Lemma \ref{lemma:bound_of_tau-exp_on_nablau}, we deduce that
	\begin{equation*}
		e^{a t}\|\,\tau^n(t)\,\|_{\BB_{p,1}^{\frac{2}{p}}}
		\,\leq\,
		\|\,\tau_0\,\|_{\BB_{p,1}^{\frac{2}{p}}}\exp\bigg\{ C\|\,\uu\,\|_{L^1(0,t;\BB_{\infty,1}^{1})}\bigg\}.
	\end{equation*}
	We hence apply the first inequality of Theorem \ref{thm-Idontknow}, to eventually gather inequality \eqref{ineq-tauBp}.
	
	\noindent Similarly, we remark that $\uu^n$ satisfies the Stokes equation
	\begin{equation*}
		\begin{cases}
			\partial_t \uu^n\,-\nu \Delta \uu^n\,+\,\nabla \pre^n \,=\,-J^n(\uu^n\cdot \nabla \uu^n)\,+\,\Div J^n \tau^n,\\
			\Div\,\uu^n\,=\,0,
		\end{cases}
	\end{equation*}
	from which we gather that
	\begin{equation}\label{ineq118b}
		\|\,\uu^n(t)\,\|_{\BB_{p,1}^{\frac{2}{p}-1}}
		\,+\,
		\nu\int_0^t\|\,\uu^n(s)\,\|_{\BB_{p,1}^{\frac{2}{p}+1}}\dd s
		\,\leq\,
		\int_0^t\|\,\uu^n(s)\otimes \uu^n(s)\,\|_{\BB_{p,1}^{\frac{2}{p}}}\dd s
		\,+\,
		\int_0^t\|\,\tau^n(s)\,\|_{\BB_{p,1}^{\frac{2}{p}}}\dd s.
	\end{equation}
	Now, applying Proposition \ref{prop:reminder-und-paraproduct}, we remark that $\uu^n\otimes\uu^n$ can be recast as follows
	\begin{equation*}
	\begin{aligned}
		\|\,\uu^n\otimes\uu^n\,\|_{\BB_{p,1}^\frac{\dd}{p}}
		\,&\leq\,
		2\|\dot T_{\uu^n\otimes}\uu^n\,\|_{\BB_{p,1}^\frac{\dd}{p}}
		\,+\,
		\|\,R(\uu^n\otimes,\,\uu^n)\,\|_{\BB_{p,1}^\frac{\dd}{p}}\\
		\,&\leq\,
		C\|\,\uu^n\,\|_{L^\infty}\|\,\uu^n\,\|_{\BB_{p,1}^{\frac{\dd}{p}}}\,+\,\|\,\uu^n\,\|_{\BB_{\infty, \infty}^0}\|\,\uu^n\,\|_{\BB_{p,1}^{\frac{\dd}{p}}}
		\,\leq\,
		C\|\,\uu^n\,\|_{\BB_{p,1}^{\frac{2}{p}}}\|\,\uu^n\,\|_{\BB_{\infty, 1}^0}.
	\end{aligned}
	\end{equation*}
	Hence, by interpolation we get
	\begin{equation*}
	\begin{aligned}
		\|\,\uu^n\otimes \uu^n\,\|_{\BB_{p,1}^{\frac{2}{p}}}
		\,\leq\,C\|\,\uu^n\,\|_{\BB_{p,1}^{\frac{2}{p}}}\|\,\uu^n\,\|_{\BB_{\infty, 1}^0}
		\,\leq\,C\|\,\uu^n\,\|_{\BB_{p,1}^{\frac{2}{p}-1}}^{\frac{1}{2}}\|\,\uu^n\,\|_{\BB_{p,1}^{\frac{2}{p}+1}}^{\frac{1}{2}}
		\|\,\uu^n\,\|_{\BB_{\infty, 1}^{-1}}^{\frac{1}{2}}
		\|\,\uu^n\,\|_{\BB_{\infty, 1}^1}^{\frac{1}{2}}\\
		\leq\,
		\frac{\nu}{2}
		\|\,\uu^n\,\|_{\BB_{p,1}^{\frac{2}{p}+1}}
		\,+\,
		C\nu^{-1}
		\|\,\uu^n\,\|_{\BB_{\infty, 1}^{-1}}
		\|\,\uu^n\,\|_{\BB_{\infty, 1}^1}
		\|\,\uu^n\,\|_{\BB_{p,1}^{\frac{2}{p}-1}}.
	\end{aligned}
	\end{equation*}
	Replacing the above inequality into \eqref{ineq118b}, we can then apply the Gronwall inequality to eventually gather that
	\begin{equation*}
	\begin{aligned}
		\|\,\uu^n(t)\,&\|_{\BB_{p,1}^{\frac{2}{p}-1}}
		\,+\,
		\int_0^t
		\|\,\uu^n(s)\,\|_{\BB_{p,1}^{\frac{2}{p}+1}}
		\dd s
		\,
		\leq \\
		\,&\leq\,
		\Big(
			\|\,\uu_0\,\|_{\BB_{p,1}^{\frac{2}{p}-1}}
			\,+\,
			\frac{1-e^{-aT}}{a}\|\,\tau_0\,\|_{\BB_{p,1}^{\frac{2}{p}-1}}
		\Big)
		\exp\bigg\{C(\nu^{-1}\|\,\uu^n\,\|_{L^\infty(0,T;\BB_{\infty, 1}^{-1})}+1)
		\|\,\uu^n\,\|_{L^1(0,T;\BB_{\infty, 1})}\bigg\}
		\\
		\,&\leq\,
	\Big(
			\|\,\uu_0\,\|_{\BB_{p,1}^{\frac{2}{p}-1}}
			\,+\,
			\frac{1-e^{-aT}}{a}\|\,\tau_0\,\|_{\BB_{p,1}^{\frac{2}{p}}}
		\Big)
		\exp\bigg\{C\nu^{-1}\Upsilon^1_\nu(T,\,\uu_0,\,\tau_0)\big(\nu^{-1}\Upsilon^2_\nu(T,\,\uu_0,\,\tau_0)+1\big)\bigg\}.
	\end{aligned}
	\end{equation*}
\end{proof}

\subsection{\bf Passage to the limit}\label{sec:limits-pas}$\,$

\noindent To conclude the proof of Theorem \ref{main_thm0}, it remains to pass to the limit in system \eqref{main_system_appx1}, as $n$ goes to infinity. We first recall that $(\uu^n)_\NN$ is uniformly bounded in 
\begin{equation}\label{cond-on-u}
		L^\infty_{\rm loc}(\RR_+,\,\BB_{p,1}^{\frac{2}{p}-1})\cap L^1_{\rm loc}(\RR_+,\,\BB_{p,1}^{\frac{2}{p}+1}) 
		\quad\quad\text{and also}\quad\quad
		L^\infty_{\rm loc}(\RR_+,L^2(\RR^2))\cap L^2_{\rm loc}(\RR_+,\,\dot H^1(\RR^2)),
\end{equation}
while the sequence $(\tau^n)_\NN$ is uniformly bounded into
\begin{equation*}
		L^\infty_{\rm loc}(\RR_+,\,\BB_{p,1}^{\frac{2}{p}})\cap
		L^\infty_{\rm loc}(\RR_+,L^2(\RR^2)).
\end{equation*}
We hence proceed by extrapolating a convergent subsequence within a suitable functional space. To this end we consider the classical Aubin-Lions Lemma.
\begin{lemma}\label{Aubin-Lions-Lemma}
	Consider three Banach Spaces $X_0$, $X$ and $X_1$ such that $X_0\subseteq X \subseteq X_1$. Assume that $X_0$ is compactly embedded in $X$ and that $X$ is continuously embedded in $X_1$. For $1\leq s,r\leq \infty$, let 
	$W$ be defined by
	\begin{equation*}
		W\, =\, \left\{\,v  \in L^s(0,T; X_0)\quad\text{such that}\quad \partial_t v \in L^r(0,T;X_1)\,\right\}.
	\end{equation*}
	Then if $s<+\infty$, the embedding of $W$ into $L^s(0,T; X)$ is compact, while if $s=+\infty$ and $r>1$ then the embedding of $W$ into $\mathcal{C}([0, T],\,X)$ is compact.
\end{lemma}
\noindent The lemma being stated, we can focalize ourselves to the choice of the functional spaces $X_0$, $X$ and $X_1$: we fix a compact set $K$ in $\RR^2$ and we define
\begin{equation*}
				X_0\,:=\, B_{q, 1}^{\frac{2}{q}-1}(K),\quad\quad X\,=\,X_1\,:=B_{q, 1}^{\frac{2}{q}-2}(K),
\end{equation*}
where $q\geq p $ and $q>2$. The compactness embedding $X_0\hookrightarrow\hookrightarrow X$ of Lemma \ref{Aubin-Lions-Lemma} is satisfied thanks to the following Proposition (cf. \cite{BCD}, Corollary $2.96$).
\begin{prop}
	For any $(s',\,s)$ in $\RR^2$ such that $s'<s$ and any compact set $K$ of $\RR^\dd$, the space $B_{q,\infty}^s(K)$ is compactly embedded in $B_{q,1}^{s'}(K)$.
\end{prop}

\noindent Since $H^1(\RR^2)$ is continuously embedded into $L^q(\RR^2)$, we gather that $(\uu^n)_\NN$ is uniformly bounded into $L^2_{\rm loc}(\RR_+,\,B_{q,1}^{2/q})$, where  $B_{q,1}^{2/q}$ stands for the {\it non-homogeneous} Besov space. Furthermore, since $\BB_{q,1}^{2/q}$ is continuously embedded into $L^\infty(\RR^2)$, the sequence $(\tau^n)_\NN$ is uniformly bounded into $L^\infty_{\rm loc}(\RR_+,L^q(\RR^2))$ and so into $L^\infty_{\rm loc}(\RR_+,\,B_{q,1}^{2/q})$ which is embedded into $L^2_{\rm loc}(\RR_+,B_{q,1}^{2/q-1})$. This allows us to conclude that 
\begin{equation*}
		\|\,\Div\,\tau^n\,\|_{L^2_{\rm loc}(\RR_+,B_{q,1}^{2/q-2}(K))} \leq C,
\end{equation*}
for a suitable constant $C$ that does not depend on the index $n\in \NN$.  Now, we claim that $\Div\,j^n(\uu^n\otimes \uu^n)$ is uniformly bounded in $L^r(0,T; B_{q,1}^{2/q-2}(K))$, for a suitable positive index $r>1$.  We first remark that
\begin{equation*}
\begin{aligned}
	\|\,\Div\,j^n(\uu^n\otimes \uu^n)\,\|_{B_{q,1}^{2/q-2}(K)}
	\,&\leq\,
	\|\,j^n(\uu^n\otimes \uu^n)\,\|_{B_{q,1}^{2/q-1}(K)}
	\,\leq\,
	\|\,j^n(\uu^n\otimes \uu^n)\,\|_{B_{q,1}^{2/q-1}}\\
	&\lesssim\,
	\|\,\Sd_{-1} j^n(\uu^n\otimes \uu^n)\,\|_{L^q(\RR^2)}
	\,+\,
	\|\,(\Id -\Sd_{-1}) j^n(\uu^n\otimes \uu^n)\,\|_{\BB_{q,1}^{2/q-1}}\\
	&\lesssim\,
	\|\,j^n(\uu^n\otimes \uu^n)\,\|_{L^q(\RR^2)}
	\,+\,
	\|\,(\Id -\Sd_{-1}) j^n(\uu^n\otimes \uu^n)\,\|_{\BB_{q,1}^{2/q-\ee}},
\end{aligned}
\end{equation*}
where $\ee\in(0,1)$ such that $2/q-\ee>0$.  Hence we eventually gather that
\begin{equation*}
	\|\,\Div\,j^n(\uu^n\otimes \uu^n)\,\|_{B_{q,1}^{2/q-2}(K)}
	\,\leq\,
	C
	\left(
	\|\,(\uu^n\otimes \uu^n)\,\|_{L^q(\RR^2)}
	\,+\,
	\|\,(\uu^n\otimes \uu^n)\,\|_{\BB_{q,1}^{2/q-\ee}}
	\right).
\end{equation*}
Now, making use of the continuity of the product between the functional spaces
\begin{equation*}
	L^\frac{2}{1-\ee}_{\rm loc}(\RR_+, \BB_{q,1}^{\frac{2}{q}-\ee})\times L^2_{\rm loc}(\RR_+, \BB_{q,1}^{\frac{2}{q}})
	\rightarrow
	L^\frac{1}{1-\ee}_{\rm loc}(\RR_+, \BB_{q,1}^{\frac{2}{q}-\ee}),
\end{equation*}
the term $\,\uu^n\otimes \uu^n\,$ is uniformly bounded into $L^{1/(1-\ee)}_{\rm loc}(\RR_+, \BB_{q,1}^{\frac{2}{q}-1})$. Furthermore, we can bound the $L^q(\RR^2)$ norm through the following interpolation:
\begin{equation*}
	\|\,\uu^n\otimes \uu^n\,\|_{L^q(\RR^2)}\leq C\|\,\uu^n\,\|_{L^{2q}(\RR^2)}^2\leq C\|\,\uu^n\,\|_{L^2(\RR^2)}^{\frac{2}{q}}\|\,\nabla \uu^n\,\|_{L^{2}(\RR^2)}^{2-\frac{2}{q}} \in L_{\rm loc}^{\frac{q}{q-1}}(\RR_+).
\end{equation*} 
Denoting by $r\,=\,\min\{2,1/(1-\ee), q/(q-1)\}>1$, we gather that the sequence $(\partial_t \uu^n)_\NN$ satisfying
\begin{equation*}
	\partial_t \uu^n\,=\,\Delta\uu^n\,-J^n(\uu^n\cdot \nabla \uu^n)\,+\,\nabla \pre^n\,+\,\Div\,\tau^n 
\end{equation*}
is uniformly bounded in $L^r_{\rm loc}(\RR_+,\,B_{p,1}^{2/p-2}(K))$. Hence, the Aubins-Lion Lemma \ref{Aubin-Lions-Lemma} and the generality of the compact set $K$ allow us to extract a convergent subsequence $(\uu^{n_k})_\NN\subset (\uu^n)_\NN$ such that
\begin{equation*}
	\uu^{n_k}\,\rightarrow\,\uu\quad\quad\text{in}\quad\quad L^\infty\big(0,T;\,(B_{p,1}^{\frac{2}{p}-2})_{\rm loc}\big).
\end{equation*} 
Since $(\uu^n)_\NN$ satisfies estimates \eqref{ineq-uuBp} and \eqref{energy-estimate}, also the limit $\uu$ belongs to the functional space given by \eqref{cond-on-u}.

\smallskip\noindent 
We now claim that $(\partial_t \tau^n)_\NN$ is uniformly bounded in the \textit{non-homogeneous} functional space $L^r_{\rm loc}(\RR_+,\,B_{p,1}^{2/p-1}(K))$. First we recall that the conformation tensor satisfies
\begin{equation*}
	\partial_t\tau^n\,=\,-\uu^n\cdot \nabla \tau^n\,+\,\omega^n\tau^n\,-\,\tau^n\omega^n,
\end{equation*}
therefore
\begin{equation*}
	\|\,\partial_t \tau^n\,\|_{\BB_{p,1}^{\frac 2p -1}}\,\lesssim\,\|\,\uu^n\,\|_{\BB_{p,1}^{\frac 2p }}\|\,\nabla  \tau^n\,\|_{\BB_{p,1}^{\frac{2}{p}-1}}
	\,+\,\|\,\nabla \uu^n\,\|_{\BB_{p,1}^{\frac{2}{p}-1}}\|\,\tau^n\,\|_{\BB_{p,1}^{\frac{2}{p}-1}}
	\,\in\,L^2_{\rm loc}(\RR_+).
\end{equation*}
Furthermore
\begin{equation*}
\begin{aligned}
	\|\,	-\omega^n\tau^n + \tau^n\omega^n	\,\|_{L^p(\RR^2)}
	&\,\lesssim\,
	\|\,\nabla \uu^n\,\|_{L^{2p}(\RR^2)}
	\|\,\tau^n\,\|_{L^{2p}(\RR^2)}\\
	&\,\lesssim\,
	\|\,\nabla \uu^n\,\|_{L^2(\RR^2)}^{\frac 1p}
	\|\,\nabla \uu^n\,\|_{L^\infty(\RR^2)}^{1-\frac 1p}
	\|\,\tau^n\,\|_{L^2(\RR^2)}^{\frac 1p}
	\|\,\tau^n\,\|_{L^\infty(\RR^2)}^{1-\frac 1p}\\
	&\lesssim\,
	\|\,\nabla \uu^n\,\|_{L^2(\RR^2)}^{\frac 1p}
	\|\,\nabla \uu^n\,\|_{\BB_{p,1}^{\frac 2p}}^{1-\frac 1p}
	\|\,\tau^n\,\|_{L^2(\RR^2)}^{\frac 1p}
	\|\,\tau^n\,\|_{\BB_{p,1}^{\frac 2p}}^{1-\frac 1p}
	\in L^{\frac{p}{p-1}}_{\rm loc}(\RR_+).
\end{aligned}
\end{equation*}
Finally, one has
\begin{equation*}
\begin{aligned}
	\|\,\Delta_{-1}\Div(\uu^n\otimes \tau^n)\,\|_{L^p(\RR^2)}
	\,&\lesssim\,
	\|\,\uu^n\otimes \tau^n\,\|_{L^p(\RR^2)}
	\,\lesssim\,
	\|\,\uu^n\,\|_{L^{2p}(\RR^2)}
	\|\,\tau^n\,\|_{L^{2p}(\RR^2)}\\
	\,&\lesssim\,
	\|\, \uu^n\,\|_{L^2(\RR^2)}^{\frac 1p}
	\|\, \uu^n\,\|_{L^\infty(\RR^2)}^{1-\frac 1p}
	\|\,\tau^n\,\|_{L^2(\RR^2)}^{\frac 1p}
	\|\,\tau^n\,\|_{L^\infty(\RR^2)}^{1-\frac 1p}\\
	\,&\lesssim\,
	\|\, \uu^n\,\|_{L^2(\RR^2)}^{\frac 1p}
	\|\, \uu^n\,\|_{\BB_{p,1}^{\frac 2p}}^{1-\frac 1p}
	\|\, \tau^n\,\|_{L^2(\RR^2)}^{\frac 1p}
	\|\, \tau^n\,\|_{\BB_{p,1}^{\frac 2p}}^{1-\frac 1p}
	\,\in\,
	L^2_{\rm loc}(\RR^2),
\end{aligned}
\end{equation*}
which allows us to conclude that $(\partial_t \tau^n)_\NN$ is uniformly bounded in $L^r_{\rm loc}(\RR_+,\,B_{p,1}^{2/p-1})$. Thus, 
the Aubins-Lion lemma allows us to extract a convergent subsequence $(\tau^{n_k})_\NN\subset (\tau^n)_\NN$ such that
\begin{equation*}
	\tau^{n_k}\,\rightarrow\,\tau\quad\quad\text{in}\quad\quad L^\infty\big(0,T;\,(B_{p,1}^{\frac{2}{p}-1})_{\rm loc}\big).
\end{equation*} 
Since $(\tau^n)_\NN$ satisfies estimates \eqref{ineq-tauBp} and \eqref{energy-estimate}, also the limit $\tau$ belongs to the functional space given by \eqref{cond-on-u}.

These properties allow to pass to the limit to system \eqref{main_system_appx} and thus to show that $(\uu,\,\tau)$ is a global-in-time solution of system \eqref{main_system}. This concludes the proof of the existence part of Theorem \ref{main_thm0}.

\subsection{\bf Uniqueness of classical solutions in dimension two: The case $p\in [2,4)$.}$\,$

\noindent  This section is devoted to the proof of the uniqueness of solutions determined by Theorem \ref{main_thm0}. We consider two solutions $(\uu_1,\,\tau_1)$ and $(\uu_2,\,\tau_2)$ with same initial data and satisfying the condition of Theorem \ref{main_thm0}:
\begin{equation*}
		\begin{aligned}
			(\uu_1,\,\uu_2)\,&\in\, L^\infty(0,T; L^2(\RR^2)\cap \BB_{p,1}^{\frac{2}{p}-1})\cap L^2(0,T;\dot H^1(\RR^2))\cap  L^1(0,T;\BB_{p,1}^{\frac{2}{p}+1}),\\
			(\tau_1,\,\tau_2) \,&\in\, L^\infty(0,T; L^2(\RR^2)\cap \BB_{p,1}^\frac{2}{p}),
		\end{aligned}
\end{equation*}
for any time $T>0$. Defining by $\delta \uu := \uu_1 - \uu_2$ and by  $\delta \tau := \tau_1 - \tau_2$, we aim in controlling $(\delta \uu,\,\delta \tau)$ in the following Chemin-Lerner spaces:
\begin{equation*}
				\tilde L^\infty(0,T; \BB_{p, 1}^{\frac{2}{p}-2})\cap \tilde L^1(0,T; \BB_{p, 1}^{\frac{2}{p}}) \times \tilde L^\infty(0,T; \BB_{p, 1}^{\frac{2}{p}-1}).
\end{equation*}
We begin with introducing the system driving the evolution of $(\delta \uu,\,\delta \tau)$:
\begin{equation*}
\left\{\hspace{0.2cm}
	\begin{alignedat}{2}
		&\,\partial_t\delta \tau\,+\,\uu_1\cdot \nabla \delta \tau 	- \,\omega_1\, \delta \tau \,+\,\delta \tau\,\omega_1 \,=\,-\delta \uu\cdot \nabla \tau_2 \,+\,\delta \omega \tau_2 \,-\, \tau_2 \delta \omega
		\hspace{3cm}
		&&\RR_+\times \RR^\dd, \vspace{0.1cm}	\\		
		&\,\partial_t \delta \uu  \, - \, \nu \Delta \delta \uu  
		\,+\, \nabla \delta \pre \, =\,
		- \delta\uu \cdot \nabla \uu_1\,-\,
		\uu_2\cdot \nabla\delta \uu \,+\,
		\Div\,\delta \tau
									\hspace{3cm}									&& \RR_+\times \RR^\dd, \vspace{0.1cm}\\
		&\,\Div\,\delta \uu\,=\,0			
		&&\RR_+\times \RR^\dd, \vspace{0.1cm}	\\			
		& (\delta \uu,\,\delta \tau)_{|t=0}	\,=\,(0,\,0)			
		&&\hspace{1.02cm} \RR^\dd.\vspace{0.1cm}																							
	\end{alignedat}
	\right.
\end{equation*}
Applying Remark \eqref{rmk:Stokes-bound} to the equation of $\delta \uu$, we gather that
\begin{equation}\label{ineq118}
	\| \,\delta \uu \,\|_{\tilde L^\infty(0,T; \BB_{p, 1}^{\frac{2}{p}-2})}
	\,+\,	
	\| \,\delta \uu \,\|_{\tilde L^1(0,T; \BB_{p, 1}^{\frac{2}{p}})}	
	\,\lesssim\,
	\| \,  \delta\uu \cdot \nabla \uu_1\,+\, \uu_2\cdot \nabla\delta \uu\,\|_{\tilde L^1(0,T; \BB_{p, 1}^{\frac{2}{p}-2})}
	\,+\,\| \,  \Div\,\delta\tau\,\|_{\tilde L^1(0,T; \BB_{p, 1}^{\frac{2}{p}-2})}.
\end{equation}
Hence, we observe first that thanks to the free-divergence condition on $\delta \uu$ and the continuity of the product between
\begin{equation*}
	\BB_{p,1}^{\frac{2}{p}-1}\times \BB_{p,1}^{\frac{2}{p}}\rightarrow \BB_{p,1}^{\frac{2}{p}-1},\quad\quad p\in [1,4),
\end{equation*}
the following inequality holds true:
\begin{equation*}
\begin{aligned}
	\| \,  \delta\uu \cdot \nabla \uu_1\,+\, \uu_2\cdot \nabla\delta \uu\,\|_{\tilde L^1(0,T; \BB_{p, 1}^{\frac{2}{p}-2})}
	\,&\leq\,
	C\| \,  \delta\uu \otimes (\uu_1,\,\uu_2)\,\|_{\tilde L^1(0,T; \BB_{p,1}^{\frac{2}{p}-1})}\\
	\,&\leq\,
	C\| \,  \delta\uu \,\|_{\tilde L^2(0,T; \BB_{p,1}^{\frac{2}{p}-1})} \|\, (\uu_1,\,\uu_2)\,\|_{\tilde L^2(0,T; \BB_{p, 1}^{\frac{2}{p}})}\\
	\,&\leq\,
	C\| \,  \delta\uu \,\|_{\tilde L^\infty(0,T; \BB_{p, 1}^{\frac{2}{p}-2})}^{\frac{1}{2}}\| \,  \nabla \delta\uu \,\|_{\tilde L^1(0,T; \BB_{p, 1}^{\frac{2}{p}-1})}^{\frac{1}{2}}
	 \|\, (\uu,\,\uu_2)\,\|_{L^2(0,T; \BB_{p, 1}^{\frac{2}{p}})}
	\\
	\,&\leq\,
	C	 \|\, (\uu_1,\,\uu_2)\,\|_{L^2(0,T; \BB_{p, 1}^{\frac{2}{p}})}
	\| \,  \delta\uu \,\|_{\tilde L^\infty(0,T; \BB_{p, 1}^{\frac{2}{p}-2})}
	 \,+\,
	 \frac{\nu}{100}
	 \| \,  \nabla \delta\uu \,\|_{\tilde L^1(0,T; \BB_{p,1}^{\frac{2}{p}-1})}
\end{aligned}
\end{equation*}
We assume $T$ sufficiently small satisfying 
\begin{equation*}
	C	 \|\, (\uu_1,\,\uu_2)\,\|_{L^2(0,T; \BB_{p, 1}^{\frac{2}{p}})}\leq \frac{1}{2},
\end{equation*}
hence, we can replace the last inequality into \eqref{ineq118} and absorb any terms by the left-hand, to gather
\begin{equation}\label{ineq120}
	\| \,\delta \uu \,\|_{\tilde L^\infty(0,T; \BB_{p, 1}^{\frac{2}{p}-2})}
	\,+\,	
	\| \,\delta \uu \,\|_{\tilde L^1(0,T; \BB_{p, 1}^{\frac{2}{p}})}	
	\,\lesssim\,
	\| \,  \delta\tau\,\|_{\tilde L^1(0,T; \BB_{p, 1}^{\frac{2}{p}-1})}
\end{equation}
We now apply Lemma \ref{lemma:bound_of_tau-exp_on_nablau} to the $\delta \tau$-equation to gather:
\begin{equation}\label{ineq119}
	\| \,  \delta\tau\,\|_{\tilde L^\infty(0,T; \BB_{p, 1}^{\frac{2}{p}-1})}
	\,\leq\,
	C
	\Big\{
		\, \| \, \delta \uu\cdot \nabla \tau_2\,\|_{\tilde L^1(0,T; \BB_{p,1}^{\frac{2}{p}-1})}\,+\,
		 \| \, \delta \omega \tau_2 \,-\, \tau_2 \delta \omega\,\|_{\tilde L^1(0,T; \BB_{p, 1}^{\frac{2}{p}-1})}
		 \,
	\Big\}
	\exp\left\{ C\int_0^T \| \, \uu_1(s)\,\|_{\BB_{p,1}^{\frac{2}{p}+1}}\dd s\,\right\}.
\end{equation}
First, we have
\begin{equation*}
\begin{aligned}
	 \| \, \delta \uu\cdot \nabla \tau_2\,\|_{\tilde L^1(0,T; \BB_{p, 1}^{\frac{2}{p}-1})}
	  \,&\leq\,
	C
	 \| \, \delta \uu\otimes \tau_2\,\|_{\tilde L^1(0,T; \BB_{p, 1}^{\frac{2}{p}})}\\
	\,&\leq\,
	C
	 \|\,\delta \uu \,\|_{\tilde L^1(0,T; \BB_{p, 1}^{\frac{2}{p}})}
	 \|\, \tau_2\,\|_{L^\infty(0,T; \BB_{p, 1}^{\frac{2}{p}})}.\\
\end{aligned}
\end{equation*}
Furthermore,
\begin{equation*}
	\begin{aligned}
		 \| \, \delta \omega \tau_2 \,-\, \tau_2 \delta \omega\,\|_{\tilde L^1(0,T; \BB_{p, 1}^{\frac{2}{p}-1})}
		 \,&\leq\,
		 C\|\, \delta \omega \|_{\tilde L^1(0,T;\BB_{p, 1}^{\frac{2}{p}-1}) }
		  \|\, \tau_2  \|_{L^\infty(0,T;\BB_{p, 1}^{\frac{2}{p}}) }\\
		  \,&\leq\,
		 C\|\, \delta \uu \|_{\tilde L^1(0,T;\BB_{p, 1}^{\frac{2}{p}}) }
		  \|\, \tau_2  \|_{L^\infty(0,T;\BB_{p, 1}^{\frac{2}{p}}) }.
	\end{aligned}	
\end{equation*}
Replacing the above inequalities into \eqref{ineq119}, we eventually that
\begin{equation}\label{ineq156}
	\| \,  \delta\tau\,\|_{\tilde L^\infty(0,T; \BB_{p, 1}^{\frac{2}{p}-1})}
	\,\leq\,
	C
	\|\, \delta \uu \|_{\tilde L^1(0,T;\BB_{p, 1}^{\frac{2}{p}}) },
\end{equation}
for a suitable positive constant $C$ that depends also on the norm of the initial data $(\uu_0,\,\tau_0)$ in $L^2(\RR^2)\cap \BB_{p,1}^{2/p-1}\times L^2(\RR^2)\cap \BB_{p,1}^{2/p}$. Hence, denoting by 
\begin{equation*}
	\delta \mathcal{U}(T) :=\| \,\delta \uu \,\|_{\tilde L^\infty(0,T; \BB_{p, 1}^{\frac{2}{p}-2})}
	\,+\,	
	\| \,\delta \uu \,\|_{\tilde L^1(0,T; \BB_{p, 1}^{\frac{2}{p}})},
\end{equation*}
and replacing the above inequalities into \eqref{ineq120}, we gather that
\begin{equation*}
	\delta \mathcal{U}(T) \leq C \int_0^T \delta \mathcal{U}(t)\dd t,
\end{equation*}
for a sufficiently small $T>0$. The Gronwall inequality yields that $\delta \mathcal{U}(t) =0$ for any $t\in [0,T]$, for which $\uu_1(t)\equiv \uu_2(t)$. Similarly, making use of \eqref{ineq156} also the conformations tensors coincide $\tau_1(t)\equiv \tau_2(t)$ for any time $t\in [0,T]$. Since the time $T$ depends uniquely on the norm of the initial data $\uu_0$ and $\tau_0$, a standard boodstrap method allow to propagate the uniqueness, globally in time. This concludes the proof of Theorem \ref{main_thm0} in the case $p\in [1,4)$.
\subsection{\bf Uniqueness of classical solutions in dimension two: The case $p=4$.}$\,$
In order to obtain the uniqueness in the critical case $p=4$ we have to control the difference between two solutions $(\delta u, \delta \tau)$ in the space
$$
				\tilde L^\infty(0,T; \BB_{4, \infty}^{-\frac{3}{2}})\cap \tilde L^1(0,T; \BB_{4, \infty}^{\frac{1}{2}}) \times \tilde L^\infty(0,T; \BB_{4, \infty}^{-\frac{1}{2}}). 
$$
Using the system verified by $(\delta u, \delta \tau)$ we get the estimate
\begin{equation}
	\| \,\delta \uu \,\|_{\tilde L^\infty(0,T; \BB_{4, \infty}^{-\frac{3}{2}})}
	\,+\,	
	\| \,\delta \uu \,\|_{\tilde L^1(0,T; \BB_{4, \infty}^{\frac{1}{2}})}	
	\,\lesssim\,
	\| \,  \delta\uu \cdot \nabla \uu_1\,+\, \uu_2\cdot \nabla\delta \uu\,\|_{\tilde L^1(0,T; \BB_{4, \infty}^{-\frac{3}{2}})}
	\,+\,\| \,  \Div\,\delta\tau\,\|_{\tilde L^1(0,T; \BB_{4, \infty}^{-\frac{3}{2}})}.
\end{equation}
Hence, using the product between
$$
	\BB_{4,\infty}^{-\frac{1}{2}}\times \BB_{4,1}^{\frac{1}{2}}\rightarrow \BB_{4,\infty}^{-\frac{1}{2}},
$$
and the same type of estimates as previously, the following inequality holds true:
\begin{equation}\label{ineq120b}
	\| \,\delta \uu \,\|_{\tilde L^\infty(0,T; \BB_{4, \infty}^{-\frac{3}{2}})}
	\,+\,	
	\| \,\delta \uu \,\|_{\tilde L^1(0,T; \BB_{4, \infty}^{\frac{1}{2}})}	
	\,\lesssim\,
	\| \,  \delta\tau\,\|_{\tilde L^1(0,T; \BB_{4, \infty}^{-\frac{1}{2}})}.
\end{equation}
Using the equation on $\delta\tau$ we obtain easily 
\begin{equation}\label{ineq119-}
	\| \,  \delta\tau\,\|_{\tilde L^\infty(0,T; \BB_{4, \infty}^{-\frac{1}{2}})}
	\,\leq\,
	C
	\Big\{
		\, \| \, \delta \uu\cdot \nabla \tau_2\,\|_{\tilde L^1(0,T; \BB_{4,\infty}^{-\frac{1}{2}})}\,+\,
		 \| \, \delta \omega \tau_2 \,-\, \tau_2 \delta \omega\,\|_{\tilde L^1(0,T; \BB_{4, \infty}^{-\frac{1}{2}})}
		 \,
	\Big\}
	\exp\left\{ C\int_0^T \| \, \uu_1(s)\,\|_{\BB_{4,1}^{\frac{3}{2}}}\dd s\,\right\}.
\end{equation}
First, we have
\begin{equation*}
\begin{aligned}
	 \| \, \delta \uu\cdot \nabla \tau_2\,\|_{\tilde L^1(0,T; \BB_{4, 1}^{-\frac{1}{2}})}
	\,&\leq\,
	C
	 \|\,\delta \uu \,\|_{\tilde L^1(0,T; \BB_{4, 1}^{\frac{1}{2}})}
	 \|\, \tau_2\,\|_{L^\infty(0,T; \BB_{4, 1}^{\frac{1}{2}})}.\\
\end{aligned}
\end{equation*}
Furthermore,
\begin{equation*}
	\begin{aligned}
		 \| \, \delta \omega \tau_2 \,-\, \tau_2 \delta \omega\,\|_{\tilde L^1(0,T; \BB_{4, \infty}^{-\frac{1}{2}})}
		 \,&\leq\,
		 C\|\, \delta \uu \|_{\tilde L^1(0,T;\BB_{4, 1}^{\frac{1}{2}}) }
		  \|\, \tau_2  \|_{L^\infty(0,T;\BB_{4, 1}^{\frac{1}{2}}) }.
	\end{aligned}	
\end{equation*}
Inserting  the above inequalities into \eqref{ineq119-}, we eventually deduce that
\begin{equation}\label{ineq156-}
	\| \,  \delta\tau\,\|_{\tilde L^\infty(0,T; \BB_{4, \infty}^{-\frac{1}{2}})}
	\,\leq\,
	C
	\|\, \delta \uu \|_{\tilde L^1(0,T;\BB_{4, 1}^{\frac{1}{2}}) }.
\end{equation}
Now we use the following standard logarithmical estimate
\begin{equation*}
\begin{aligned}	\| \,  \delta u\,\|_{\tilde L^1(0,T; \BB_{4, 1}^{\frac{1}{2}})}\leq 	\| \,  \delta u\,\|_{\tilde L^1(0,T; \BB_{4, \infty}^{\frac{1}{2}})}\ln \bigg(e+\frac{\| \,  \delta u\,\|_{\tilde L^1(0,T; \BB_{4, \infty}^{-\frac{1}{2}})}+\| \,  \delta u\,\|_{\tilde L^1(0,T; \BB_{4, 1}^{\frac{3}{2}})}}{\| \,  \delta u\,\|_{\tilde L^1(0,T; \BB_{4, \infty}^{\frac{1}{2}})}}\bigg)\\
\leq  	\| \,  \delta u\,\|_{\tilde L^1(0,T; \BB_{4, \infty}^{\frac{1}{2}})}\ln \bigg(e+\frac{\| \,  (u^1, u^2)\,\|_{ L^1(0,T; L^2)}+\| \,  (u_1,u_2)\,\|_{\tilde L^1(0,T; B^{\frac 32}_{4,1})}}{\| \, \delta u\,\|_{\tilde L^1(0,T; \BB_{4, \infty}^{\frac{1}{2}})}}\bigg).
\end{aligned}
\end{equation*}
Denoting by 
\begin{equation*}
	\delta \mathcal{U}(T) :=\| \,\delta \uu \,\|_{\tilde L^\infty(0,T; \BB_{4, \infty}^{-\frac{1}{2}})}
	\,+\,	
	\| \,\delta \uu \,\|_{\tilde L^1(0,T; \BB_{4, \infty}^{\frac{1}{4}})},
\end{equation*}
and replacing the above inequalities into \eqref{ineq120b}, we gather that
\begin{equation*}
	\delta \mathcal{U}(T) \leq C \int_0^T \delta \mathcal{U}(t)\ln(e+\frac{C}{\delta\mathcal U(t)})\dd t.
\end{equation*}
The uniquness of our solutions is then an application of the classical Osgood lemma.

\section{Non-corotational setting with positive $\mu$}\label{sec:non-cor}

\noindent This section is devoted to the proof of Theorem \ref{main_thm0b}. The main idea is to show that the solution remains close the one of the corotational setting with $\mu=0$, as long as the main parameters of System \eqref{main_system} are sufficiently small. Indeed, we first denote by $(v,\sigma)$ the solution of 
\begin{equation}
		\,\partial_t \sigma\,+\,v\cdot \nabla \sigma +a \sigma - \,W\, \sigma \,+\,\sigma\,W \,=\,0,
		\hspace{1.5cm}
		\,\partial_t v  + v\cdot \nabla  v\, - \, \nu \Delta  v  
		\,+\, \nabla q  =
		\Div\,\sigma,
\end{equation}
together with $\Div\, v = 0$, $W = (\nabla v -\,^t \nabla v)/2$ and the initial condition $(v,\,\sigma)_{|t=0} = (\uu_0,\,\tau_0)$. The existence and uniqueness of this solution is given by Theorem \ref{main_thm0} we have proven in the previous sections. Furthermore, $(v,\sigma)$ satisfies  
\begin{equation*}
	\begin{alignedat}{16}
		v	\,	&\in	\,&& L^\infty(\RR_+;  \,L^2(\RR^2))&&&&\cap 
		L^2(\RR_+,\dot H^1(\RR^2)),\quad\quad
		&&&&&&&&\sigma	\,	\in	\, L^\infty(\RR_+,\,L^2(\RR^2)),\\
		v		\,	&\in	\,&& L^\infty(\RR_+;  \,\BB_{p,1}^{\frac{2}{p}-1})
		&&&&\cap 
		L^1(\RR_+,\BB^{\frac{2}{p}+1}_{p,1}),\quad\quad
		&&&&&&&&\sigma	\,	\in	\, L^\infty(\RR_+,\,\BB_{p,1}^{\frac{2}{p}}).
	\end{alignedat}
	\end{equation*}
Thanks to the initial condition $\tau_0 \in \BB_{p,1}^\frac{2}{p}\cap \BB_{p,1}^{\frac{2}{p}+1}$ and Proposition \ref{lemma:bound_of_tau-exp_on_nablau}, the following bound for $\sigma$ holds true 
\begin{equation*}
\begin{aligned}
	\sigma \in L^\infty(\mathbb{R}_+,\BB_{p,1}^{\frac{2}{p}+1}) 
	\quad
	\text{with}
	\quad
	e^{at}\| \sigma (t) \|_{\BB_{p,1}^{\frac{2}{p}+1}}
	&\leq 
	C
	\| \tau_0 \|_{\BB_{p,1}^\frac{2}{p}\cap \BB_{p,1}^{\frac{2}{p}+1}} 
	\exp
	\Big\{ 
		C\int_0^t  \| \nabla v \|_{\BB_{\infty,1}^0}
	\Big\}	\\
	&\leq 
	C
	\| \tau_0 \|_{\BB_{p,1}^\frac{2}{p}\cap \BB_{p,1}^{\frac{2}{p}+1}} 
	\exp
	\Big\{ 
		C\nu^{-1}\Upsilon_{\nu,a}^1(T,\, \uu_0,\,\tau_0).
	\Big\}	
	\end{aligned}
\end{equation*}
We recall that the operator $\Upsilon_{\nu,a}^1$ has been introduced in Definition \ref{def-Upsilon} (with $\mu=0$), it depends uniquely on $\| \uu_0 \|_{\BB_{\infty,1}^{-1}}$ and $\| \tau_0 \|_{\BB_{\infty,1}^{0}}$. Additionally, the function
$T\in (0,\infty)\rightarrow \Upsilon_{\nu,a}^1(T,\, \uu_0,\,\tau_0)$ is uniformly bounded, since the damping term $a>0$ is positive.  

\noindent We now introduced the local solution $(\uu,\,\tau)$ of our original problem \eqref{main_system}. Since we are in the local setting, such a solution can be achieved by a rather standard  Friedrichs-type scheme. Furthermore $(\uu,\tau)$ belongs to the same functional space of $(v,\sigma)$, locally in time. Moreover, being $T^*>0$ its lifespan, if $T^*$ is finite then a blow up must occur for the following norm:
\begin{equation*}
	\int_0^{T^*} \| \nabla \uu(s) \|_{\BB_{p,1}^\frac{2}{p}}ds=+\infty.
\end{equation*}
Indeed, the solution would own a Lipschitz regularity as long as the above norm is finite. This would allow to extend the solution beyond the finite lifespan $T^*$.  

\noindent
We consider the difference $(\delta \uu,\delta \tau):= (\uu,\tau)-(v,\sigma)$, which is solution of the following equations  
\begin{equation}
\begin{cases}
		\,\partial_t \delta \tau \,+\,(v+\delta \uu)\cdot \nabla \delta \tau +a \delta \tau  - \,W \, \delta \tau \,+\,\delta \tau\,W 			\,=
		\,\mu \mathbb{D}-\delta u \cdot \nabla \sigma + \delta \omega \, \tau -\tau \,\delta \omega
		-b \big(  \mathbb{D} \tau +   \mathbb{D}\tau  \big)
		\\
		\,\partial_t \delta \uu   -  \nu \Delta  \delta \uu  
		\,+\, \nabla \delta p  =
		\Div\,\delta \tau  - \delta \uu \cdot \nabla \delta \uu -\Div(\delta \uu \otimes v+v\otimes \delta u),
\end{cases}
\end{equation}
together with $\Div\, \delta u = 0$, $\delta \omega = (\nabla \delta \uu -\,^t \nabla \delta \uu)/2$, $D =  (\nabla v +\,^t \nabla v)/2$ and the initial condition $(\delta u,\,\delta \tau)_{|t=0} = (0,\,0)$. Since $v$ is bounded in $L^1(\mathbb{R}_+,\BB_{p,1}^{2/p})$, the following blow up  must occurs if $T^*<\infty$:
\begin{equation*}
	\int_0^{T^*} \| \nabla \delta \uu(s) \|_{\BB_{p,1}^\frac{2}{p}}\dd s=+\infty.
\end{equation*}

\noindent 
The momentum equation leads to the following inequality:
\begin{equation}\label{glob_exist_pos_mu:est3}
\begin{aligned}
	\| \delta \uu \|_{L^\infty(0,t;\BB_{p,1}^{\frac{2}{p}-1})}&+ 
	\nu \int_0^t \| \nabla \delta \uu(s) \|_{\BB_{p,1}^{\frac{2}{p}}}\dd s
	\leq 
	\tilde C \int_0^t \| \delta \tau(s) \|_{\BB_{p,1}^\frac{2}{p}}\dd s + \\ 
	&+
	\tilde C \| \delta \uu \|_{L^\infty(0,t;\BB_{p,1}^{\frac{2}{p}-1})}
	\int_0^t \| \nabla \delta \uu(s) \|_{\BB_{p,1}^{\frac{2}{p}}}\dd s
	+
	\tilde C 
	\int_0^t \| \delta \uu \|_{L^\infty(0,s;\BB_{p,1}^{\frac{2}{p}-1})}\|  v(s) \|_{\BB_{p,1}^{\frac{2}{p}}}^2\dd s.
\end{aligned}
\end{equation}
We thus introduce the maximum time $\tilde T\in (0,T^*)$ such that for any $t\in (0,\tilde T)$ one has  
$$\|\delta \uu \|_{L^\infty(0,t;\BB_{p,1}^{\frac{2}{p}-1})}+\nu \int_0^t \| \nabla \delta \uu(s) \|_{\BB_{p,1}^{\frac{2}{p}}}\dd s \leq \frac{1}{4\tilde C }.$$ 
Absorbing the bilinear term in $\delta \uu$ of \eqref{glob_exist_pos_mu:est3} and applying the Gronwall's lemma
\begin{equation}\label{glob_exist_pos_mu:est2}
	\| \delta \uu \|_{L^\infty(0,t;\BB_{p,1}^{\frac{2}{p}-1})}+ 
	\nu \int_0^t \| \nabla \delta \uu(s) \|_{\BB_{p,1}^{\frac{2}{p}}}\dd s
	\leq 
	\tilde 
	C
	\int_0^t \| \delta \tau(s) \|_{\BB_{p,1}^\frac{2}{p}}\dd s
	\exp
	\Big\{
	\tilde C\int_0^t \| v(s) \|_{\BB_{p,1}^\frac{2}{p}}^2\dd s 
	\Big\}.
\end{equation}
This time value $\tilde T$ exists since $\delta u\in \mathcal{C}([0,T^*), \BB_{p,1}^{2/p-1})$ and $\delta u(0)=0$.
Next, we apply Proposition \ref{lemma:bound_of_tau-exp_on_nablau} together with the Gronwall inequality to gather that
\begin{equation}\label{glob_exist_pos_mu:est1}
\begin{aligned}
	\| \delta \tau(t) \|_{\BB_{p,1}^\frac{2}{p}}
	&\leq 
	C\Big(
	\mu \int_0^te^{a(s-t)} \| \nabla \delta \uu \|_{\BB_{p,1}^\frac{2}{p}}ds+
	\mu \int_0^te^{a(s-t)} \| \nabla v \|_{\BB_{p,1}^\frac{2}{p}}ds
	+
	\int_0^te^{a(s-t)} \| \delta \uu \|_{\BB_{p,1}^\frac{2}{p}}\| \nabla \sigma \|_{\BB_{p,1}^\frac{2}{p}}ds+\\
	&+
	(1+|b|)\int_0^te^{a(s-t)} \| \nabla  \delta \uu \|_{\BB_{p,1}^\frac{2}{p}}\|  \sigma \|_{\BB_{p,1}^\frac{2}{p}}ds
	\Big)
	\exp\Big\{C(1+|b|)\int_0^t \| \nabla v \|_{\BB_{p,1}^\frac{2}{p}} + 
	C(1+|b|)\frac{1}{4\tilde C\nu
	} \Big\}.
\end{aligned}
\end{equation}
We plug estimate \eqref{glob_exist_pos_mu:est1} into the inequality \eqref{glob_exist_pos_mu:est1} and we apply the Gronwall's lemma:
\begin{equation*}
	\begin{aligned}
		\| \delta \uu \|_{L^\infty(0,t;\BB_{p,1}^{\frac{2}{p}-1})}+ 
	\nu \int_0^t \| \nabla \delta \uu \|_{\BB_{p,1}^{\frac{2}{p}}}
	\leq 	
	C
	\Big\{
	\mu \int_0^t \| \nabla \delta \uu \|_{\BB_{p,1}^{\frac{2}{p}}} + 
	\mu \mathcal{G}_{a,\nu}(\uu_0,\tau_0)
	+
	\frac{|b|}{a}
	\mathcal{G}_{a,\nu}(\uu_0,\tau_0)
	\int_0^t \| \nabla \delta \uu \|_{\BB_{p,1}^{\frac{2}{p}}}
	\Big\}\times \\
	\times 
	\exp\Big\{C(1+|b|)\mathcal{G}_{a,\nu}(\uu_0,\tau_0) + 
	C(1+\frac{1}{\nu}+|b|)\frac{1}{4\tilde C\nu }\Big\}
	\end{aligned}
\end{equation*} 
where
\begin{equation*}
\begin{aligned}
		\mathcal{G}_{a,\nu}(\uu_0,\tau_0):=
		\max\bigg\{
		\max_{t\in\mathbb{R}}
		\Big(
			\|\,\uu_0\,\|_{\BB_{p,1}^{\frac{2}{p}-1}}
			\,+\,
			\frac{1-e^{-at}}{a}\|\,\tau_0\,\|_{\BB_{p,1}^{\frac{2}{p}}}
		\Big)
		\exp\bigg\{C\nu^{-1}\Upsilon^1_\nu(t,\,\uu_0,\,\tau_0)\big(\nu^{-1}\Upsilon^2_\nu(t,\,\uu_0,\,\tau_0)+1\big)
		\bigg\},\\	
		\, 
		\|\,\tau_0\,\|_{\BB_{p,1}^{\frac{2}{p}}\cap \BB_{p,1}^{\frac{2}{p}+1}}\exp\bigg\{\,C\nu^{-1}\,\Upsilon_{1,\nu}(t, \,\uu_0,\,\tau_0)\,\bigg\}
		\bigg\}.
\end{aligned}
\end{equation*}
Hence for a sufficiently small $\mu>0$ and a small parameter $|b|>0$ (depending to the initial data $(\uu_0,\tau_0)$), we finally achieve that
\begin{equation*}
	\| \delta \uu \|_{L^\infty(0,t;\BB_{p,1}^{\frac{2}{p}-1})}+ 
	\nu \int_0^t \| \nabla \delta \uu \|_{\BB_{p,1}^{\frac{2}{p}}}
	\leq 	\frac{1}{8\tilde C \nu},
\end{equation*}
for any time $t\in (0,\tilde T)$. This implies that $T^*=\tilde T$, hence $T^*$ can not be finite, since the above norms are not blowing up.

\section{Local-in-time solutions for any initial data}\label{sec:loc-sol}
\noindent
This section is devoted to the proof of the results of local-in-time existence of classical solutions for the Johnson-Segalman model \eqref{main_system} in dimension $\dd\geq 3$, aiming to prove Theorem \ref{main_thm1}. We adopt a standard strategy, which can be  summarized as follows:
\begin{itemize}
	\item construction of global-in-time approximate solutions,
	\item uniform estimates on a suitable fixed (small) interval of time,  
	\item convergence of the sequences to a solution in such an interval,
	\item uniqueness of the solutions.
\end{itemize}

\subsection{\bf Global-in-time approximate solutions}$\,$
 
\noindent We first introduce the solution $\uu_L=\uu_L(t,x)$ of the following linear non-stationary Stokes system:
\begin{equation}\label{Stokes-eq}
	\left\{\hspace{0.2cm}
	\begin{alignedat}{2}
		&\,\partial_t \uu_L - \nu \Delta \uu_L\, +\, \nabla \pre_L\,=\, 0
		\hspace{3cm}
		&&\RR_+\times \RR^\dd, \vspace{0.1cm}\\
		&\,\Div\,\uu_L=\, 0
		\hspace{3cm}
		&&\RR_+\times \RR^\dd, \vspace{0.1cm}\\
		&\,
		\uu_{L\,|t=0}	\,=\,\uu_0			
		&&\hspace{1.02cm} \RR^\dd.\vspace{0.1cm}																							
	\end{alignedat}
	\right.
\end{equation} 
Since $\uu_0$ belongs to $\BB_{p, 1}^{\frac{\dd}{p}-1}$ a standard approach allows us to conclude (cf. \cite{PD}, Remark 3.12) that the solution $\uu^L(t)$ belongs to the functional framework
\begin{equation*}
	\uu_L \in \mathcal{C}_{\rm b}(\RR_+,\BB_{p, 1}^{\frac{\dd}{p}-1})	\cap L^1_{\rm loc}(\RR_+, \BB_{p, 1}^{\frac{\dd}{p}+1})
	\quad\quad
	\text{and}
	\quad\quad
	\nabla \pre_L \in L^1_{\rm loc}(\RR_+, \BB_{p, 1}^{\frac{\dd}{p}-1}),
\end{equation*}
satisfying the inequality
\begin{equation*}
	\|\, \uu_L(t)\,\|_{\BB_{p, 1}^{\frac{\dd}{p}-1}} \,+\, \nu \int_0^t \|\, \uu_L(s)\,\|_{\BB_{p, 1}^{\frac{\dd}{p}+1}}\dd s
	\,\leq\,
	\|\,\uu_0\,\|_{\BB_{p, 1}^{\frac{\dd}{p}-1}}.
\end{equation*}
We now define the functions $(\bar \uu^0(t,x),\, \tau^0(t,x)) \,:=\,(0,\, \tau_0(x))$ and we apply an iterative method to solve the following sequences of linear equations:
\begin{equation}\label{main_system_appx}
\left\{\hspace{0.2cm}
	\begin{alignedat}{2}
		&\,\partial_t \tau^{n+1}\,+\,\uu^n\cdot \nabla \tau^{n+1} 	- \,\omega^n \tau^{n+1} \,+\,\tau^{n+1}\omega^n \,=\,0
		\hspace{3cm}
		&&\RR_+\times \RR^\dd, \vspace{0.1cm}	\\		
		&\,\partial_t\bar  \uu^{n+1}\,   - \, \nu \Delta  \bar \uu^{n+1}  
		\,+\, \nabla \bar \pre^{n+1}  =- \uu^n\cdot \nabla \uu^n
		\,+\,
		\Div\,\tau^{n+1}
									\hspace{3cm}									&& \RR_+\times \RR^\dd, \vspace{0.1cm}\\
		&\,\Div\, \bar  \uu^{n+1}\,=\,0			
		&&\RR_+\times \RR^\dd, \vspace{0.1cm}	\\			
		& (\bar \uu^{n+1},\,\tau)_{|t=0}	\,=\,(0,\,\tau_0)			
		&&\hspace{1.02cm} \RR^\dd,\vspace{0.1cm}																							
	\end{alignedat}
	\right.
\end{equation}
where $\uu^n$ stands for $\uu^L\,+\,\bar \uu^n$. Being a linear system of PDE's, the above equations are globally solvable in time. Furthermore, we claim that an induction method implies the sequence of solutions $(\,\bar \uu^n,\, \tau^n\,)$ to belong to the functional framework
\begin{equation}\label{barUn_ind}
	\bar{\uu}^{n}\,\in\, \mathcal{C}_{\rm b}(\RR_+,\BB_{p, 1}^{\frac{\dd}{p}-1}\,)\cap L^1_{\rm loc}(\RR_+,\BB_{p, 1}^{\frac{\dd}{p}+1}),
	\quad\quad
	\tau^n\,\in\,\mathcal{C}(\RR_+, \BB_{p, 1}^{\frac{\dd}{p}}).
\end{equation}
The basic case of $n=0$ is automatically satisfied by the definition of $(\bar \uu^0,\,\tau^0)$, thus we focus on the induction step. Thanks to Lemma \ref{lemma:bound_of_tau-exp_on_nablau}, the following bound for the conformation tensor $\tau^{n+1}$ holds true
\begin{equation*}
\|\,\tau^{n+1}\,\|_{L^\infty_{\rm loc}(\RR_+,\, \BB_{p, 1}^{\frac{\dd}{p}})}
	\,\leq\, 
	\|\, \tau_0\, \|_{\BB_{p, 1}^{\frac{\dd}{p}}}
	\exp
	\Big\{\,
		C \| \, \uu^n\,\|_{L^1_{\rm loc}(\RR_+,\, \BB_{p, 1}^{\frac{\dd}{p}+1})}
	\,\Big\}.
\end{equation*}
Moreover, applying Remark \ref{rmk:Stokes-bound} to the mild formulation
\begin{equation*}
		\bar \uu^{n+1}(t)\,=\,\int_0^t \Div \Pp{\rm e}^{(t-s)\Delta}\big(-\uu^n(s)\otimes \uu^n(s)\, +\, \tau^{n+1}(s)\big)\dd s,
\end{equation*}
we gather that $\uu^{n+1}$ satisfies 
\begin{equation*}
\begin{aligned}
	\|\,\bar \uu^{n+1}\,\|_{L^\infty_{\rm loc}(\RR_+,\,\BB_{p, 1}^{\frac{\dd}{p}-1})}\,+\,\nu\|\,\uu^{n+1}\,\|_{L^1_{\rm loc}(\RR_+,\,\BB_{p, 1}^{\frac{\dd}{p}+1})},
	\,&\leq\,
	C\Big\{ \,\| \,\uu^n\,\|_{L^2_{\rm loc}(\RR_+,\BB_{p, 1}^{\frac{\dd}{p}})}^2\, +\| \,\tau^{n+1}\,\|_{L^1_{\rm loc}(\RR_+,\BB_{p, 1}^{\frac{\dd}{p}})} \, \Big\},
\end{aligned}
\end{equation*}
thus the condition \eqref{barUn_ind} is satisfied by induction.

\subsection{\bf Uniform estimates on a fixed small interval}\label{sec:uniform-est-fixed-interval}$\,$  

\noindent In this section we deal with some suitable estimates of our approximate solutions $(\uu^n,\,\tau^n)_\NN$. We show that for a sufficiently small time $T>0$, the norms of the solutions $(\uu^n,\,\tau^n)_\NN$ within the functional framework
\begin{equation}\label{def-XT}
	\bar{\uu}^{n}\,\in\, \mathcal{C}([0,T], \BB_{p, 1}^{\frac{\dd}{p}-1})\cap L^1(0,T;\BB_{p, 1}^{\frac{\dd}{p}+1}),
	\quad\quad
	\tau^n\,\in\,\mathcal{C}([0,T], \BB_{p, 1}^{\frac{\dd}{p}})
\end{equation}
are bounded uniformly in $n\in \NN$. 

\noindent 
Thanks to  Lemma \ref{lemma:bound_of_tau-exp_on_nablau}, we have that for any time $T\geq 0$ the following inequality holds true:
\begin{equation*}
	\|\,\tau^{n+1}\,\|_{L^\infty(0,\, T;\,\BB_{p, 1}^{\frac{\dd}{p}})}
	\,\leq \,
	\|\, \tau_0\, \|_{\BB_{p, 1}^{\frac{\dd}{p}}}
	\exp
	\Big\{
		 C\| \,\uu^n\,\|_{L^1(0,\, T;\, \BB_{p, 1}^{\frac{\dd}{p}+1})}
	\Big\}.
\end{equation*}
Next, we define the function $\bar{\mathcal{U}}^n(T)$, depending on time $T\geq 0$, by means of 
\begin{equation*}
	\bar{\mathcal{U}}^n(T)\:=\, 
	\|\,\bar \uu^n\,\|_{L^\infty(0,\,T\,; \BB_{p, 1}^{\frac{\dd}{p}-1})}
	\,+\,
	\nu \|\,\bar \uu^n\,\|_{L^1(0,\,T\,; \BB_{p, 1}^{\frac{\dd}{p}+1})}
	\,+\,
	\|\,\nabla \bar \pre^n\,\|_{L^1(0,\,T\,;\BB_{p, 1}^{\frac{\dd}{p}-1})}.
\end{equation*}
We thus have
\begin{equation*}
\begin{aligned}
	\bar{\mathcal{U}}^{n+1}(T)
	\,&\lesssim\,
	\|\,\uu^L\|_{L^2(0,\, T\,; \BB_{p, 1}^{\frac{\dd}{p}})}^2
	\,+\,
	\|\,\bar \uu^{n}\|_{L^2(0,\, T\,; \BB_{p, 1}^{\frac{\dd}{p}})}^2
	\,+\,
	\|\,\tau^{n+1}\,\|_{L^1(0,\, T;\,\BB_{p, 1}^{\frac{\dd}{p}})}\\
	&\leq
	C\|\,\uu^L\|_{L^2(0,\, T;\BB_{p, 1}^{\frac{\dd}{p}})}^2
	\,+\,
	C\nu^{-1}
	\bar{\mathcal{U}}^{n}(T)^2
	\,+\,
	C
	T 
	\|\, \tau_0\, \|_{\BB_{p, 1}^{\frac{\dd}{p}}}
	\exp
	\Big\{
		 C\,\bar{\mathcal{U}}^{n}(T)
	\Big\}.
\end{aligned}
\end{equation*}
Observing that $\bar{\mathcal{U}}^0\equiv 0$, one can easily show by induction that $\bar{\mathcal{U}}^n(T) \leq \nu\ee /(2C)$, for any integer $n\in \NN$ and a small parameter $\ee>0$, provided that $T>0$ is sufficiently small, for instance satisfying
\begin{equation}\label{bound-for-T}
	C\|\,\uu^L\|_{L^2(0,\, T;\BB_{p, 1}^{\frac{\dd}{p}})}^2
	\,+\, 
	C
	T
	\|\, \tau_0\, \|_{\BB_{p, 1}^{\frac{\dd}{p}}}
	\exp
	\Big\{
		\frac{\nu}{2}
	\Big\}	
	\leq 
	\frac{\nu\ee }{100 C}.
\end{equation}
Denoting by $T$ the supremum of the time $t$ satisfying the above inequality, we can finally deduce that the sequence $(\uu^n,\,\tau^n,\,\nabla \pre^n)$ is uniformly bounded in the functional space given by \eqref{def-XT}. 

\subsection{\bf Convergence in small norms}\label{sec:conv_in_small_norm}$\,$

\noindent Denoting by $\bar \tau^n \,:=\,\tau^n\,-\,\tau_0$, we claim that the sequence $(\bar \uu^n,\,\bar \tau^n,\,\nabla \bar \pre^n)$ is a Cauchy sequence in the functional setting given by
\begin{equation}\label{YT}
		(\uu^n,\,\tau^n,\,\nabla \pre^n)\,\in\,\Y_T\quad\Rightarrow\quad 
		\left\{
			\begin{alignedat}{8}
						&\bar \uu^n	\,	&&\in	\,\tilde  
											L^\infty(0,\, T; 	  \,&&&&\BB_{p, 1}^{\frac{\dd}{p}-2}\,)\cap \tilde L^1(0, T,\BB_{p, 1}^{\frac{\dd}{p}}),\\
						&\bar \tau^n	\,	&&\in	\, \tilde L^\infty(0, T,\,	&&&&\BB_{p, 1}^{\frac{\dd}{p}-1}\,),\\
				\nabla	&\bar \pre^n	\,	&&\in 	\,	\tilde L^1(0, T,				&&&&\BB_{p, 1}^{\frac{\dd}{p}-2}\,).
			\end{alignedat}
		\right.	
\end{equation}
We will prove that this condition holds true, up to sufficiently decreasing the life span $T>0$. It is worth to remark that the considered Besov spaces are homogeneous, hence the fact that $(\bar \uu^n,\,\bar \tau^n,\,\nabla \bar \pre^n)$ belongs to the functional space given by \eqref{def-XT} does not automatically imply that the same functions $(\bar \uu^n,\,\bar \tau^n,\,\nabla \bar \pre^n)$ belong also to $\Y_T$. Nevertheless, system \eqref{main_system_appx} together with Proposition \ref{prop:reminder-und-paraproduct} yields that $\partial_t \bar \uu^n$ and $\partial_t \tau^n$ belongs to $\tilde L^1(0,T; \BB_{p,1}^{\dd/p-2})$ and $\tilde L^1(0,T; \BB_{p,1}^{\dd/p-1})$, respectively. This remark together with the initial condition $\bar \uu^n(0)\,=\,0$ and $\bar \tau^n(0)=0$ justifies the fact that $(\bar \uu^n,\,\bar \tau^n,\,\nabla \bar \pre^n)$ belongs to the functional space given by $\Y_T$.

\smallskip
\noindent We now introduce the notation $\delta \tau^n\,=\,\tau^{n+1}\,-\,\tau^n$, $\delta \uu^{n}\,=\,\uu^{n+1}\,-\,\uu^n$, $\delta \omega^n\,=\,\omega^{n+1}\,-\,\omega^n$ and $\delta \pre^n \,=\,\pre^{n+1}-\pre^n$. 
Hence, we first remark from the equation of the conformation tensor in \eqref{main_system_appx} that $\delta \tau^n$ satisfies
\begin{equation*}
	\partial_t \delta \tau^n \,+\,\uu^{n-1}\cdot \nabla \delta \tau^n\,-\,\omega^{n-1} \delta \tau^n\,+\,\delta \tau^n \omega^{n-1}\,=\,
	-\Div\left(\,\delta \uu^{n-1}\otimes \tau^n\,\right) \,+\,\delta \omega^{n-1}\tau^n \,-\,\tau^n\,\delta \omega^{n-1}.
\end{equation*}
We can then apply Lemma \ref{lemma:bound_of_tau-exp_on_nablau} about the propagation of Besov regularity of index $0$, insuring that
\begin{equation*}
\begin{aligned}
	\|\,\delta \tau^n\,\|_{\tilde L^\infty(0, t; \BB_{p,1}^{\frac{\dd}{p}-1})}\,\leq \,
	C\exp\left\{\,\int_0^t\|\,\nabla \uu^{n-1}(s)\,\|_{\BB_{p,1}^\frac{\dd}{p}}\dd s\right\}
	\bigg(
		& \|\,\delta \uu^{n-1}(s)\otimes \tau^n(s)\,\|_{\tilde L^1(0,t;\,\BB_{p,1}^\frac{\dd}{p})}
		\,+\\&+\|\,\delta \omega^{n-1}(s)\tau^n(s) \,-\,\tau^n(s)\,\delta \omega^{n-1}(s)\,\|_{\tilde L^1(0, t;\, \BB_{p,1}^{\frac{\dd}{p}-1})}
	\bigg),
\end{aligned}
\end{equation*}
for a suitable positive constant $C$. 
Next, from Proposition \ref{prop:reminder-und-paraproduct} we gather
\begin{equation*}
	\|\,\delta \tau^n\,\|_{\tilde L^\infty(0,t; \BB_{p,1}^{\frac{\dd}{p}-1})}\,\leq \,C\exp\left\{\,\int_0^t\|\,\nabla \uu^{n-1}(s)\,\|_{\BB_{\infty,1}^0}\dd s\right\}
		\|\,\delta \uu^{n-1}(s)\,\|_{\tilde L^1(0,t;\BB_{p, 1}^\frac{\dd}{p})}\|\,\tau^n(s)\,\|_{\tilde L^\infty(0,t; \BB_{p,1}^\frac{\dd}{p})}.
\end{equation*}
Hence, there exists a smooth increasing function $\chi\,=\,\chi(T)$ which is null in $T=0$ and such that the following bound holds true:
\begin{equation}\label{Cauchy-converg:est1}
		\|\,\delta \tau^n\,\|_{\tilde L^\infty(0,t;\BB_{p,1}^{\frac{\dd}{p}-1})}\,\leq\,\chi(T) \|\,\delta \uu^{n-1}\,\|_{\tilde L^1(0,T;\BB_{p, 1}^\frac{\dd}{p})}
		\,\leq\,\chi(T)\delta U^{n-1}(t).
\end{equation}
Here, the function $\chi$ depends just on $T$ and not on the index $n$ of the approximate solution. 

\noindent
Next, since $\delta \uu^n$ is a classical solution of the equation
\begin{equation*}
	\partial_t \delta \uu^n\,-\,\nu \Delta \delta \uu^n\,+\,\nabla \delta \pre^n\,=\,-\Div\left(\,\uu^{n-1}\otimes \delta \uu^{n-1}\,\right)
	-\Div\left(\,\delta \uu^{n-1}\otimes \uu^{n}\,\right)\,+\,\Div\,\delta \tau^n,
\end{equation*}
Then, the Remark \ref{rmk:Stokes-bound} concerning suitable bounds for the Stokes operator implies that
\begin{equation}\label{Cauchy-converg:est2}
\begin{aligned}
	\delta \mathcal{U}^n(t)\,&:=\,\|\,\delta \uu^n\,\|_{\tilde L^\infty(0,t;\,\BB_{p, 1}^{\frac{\dd}{p}-2})}\,+\,\nu \|\,\delta \uu^n\,\|_{\tilde L^1(0,t\,;\,\BB_{p, 1}^{\frac{\dd}{p}})}
	\,+\,\|\,\nabla \delta \pre^n\,\|_{\tilde L^1(0,t\,;\,B_{p, 1}^{\frac{\dd}{p}-2})}\\
	&\leq\,
	C\left(
		 		\|\,\uu^{n-1}\otimes \delta \uu^{n-1}\,\|_{\tilde L^1(0,t;\,\BB_{p, 1}^{\frac{\dd}{p}-1})}
				\,+\, \|\,\delta \uu^{n-1}\otimes \uu^{n}\,\|_{\tilde L^1(0,t;\,\BB_{p, 1}^{\frac{\dd}{p}-1})}
				\,+\, \|\,\delta \tau^n \,\|_{\tilde L^1(0,t;\,\BB_{p, 1}^{\frac{\dd}{p}-1})}
	\right).
\end{aligned}
\end{equation}
Hence, recalling that  $p\in [1, 2\dd)$, the product is a continuous function between the functional spaces
\begin{equation*}
	\tilde L^2(0,t;\,\BB_{p, 1}^{\frac{\dd}{p}-1})\times \tilde L^2(0,t;\,\BB_{p, 1}^\frac{\dd}{p})\,\rightarrow 
	\tilde L^1(0,t;\,\BB_{p,1}^{\frac{\dd}{p}-1}),
\end{equation*}
we combine \eqref{Cauchy-converg:est2} together with \eqref{Cauchy-converg:est1}, to get
\begin{equation*}
	\delta \mathcal{U}^n(t)\,\leq\,
	C\left(
		\|\,\uu^{n-1}\,\|_{\tilde L^2(0,t\,;\,\BB_{p, 1}^\frac{\dd}{p})}
		\,+\,
		\|\,\uu^n\,\|_{\tilde L^2(0,t\,;\,\BB_{p, 1}^\frac{\dd}{p})}
	\right)
	\nu^{-1/2}\delta \mathcal{U}^{n-1}(t)
	\,+\,
	\nu^{-1}\chi(t)\delta  \mathcal{U}^{n-1}(t).
\end{equation*}
From \eqref{bound-for-T}, there exists a small parameter $c$, we can assume small enough, such that
\begin{equation*}
		\|\,\uu^{n-1}\,\|_{L^2(0,t\,;\,\BB_{p, 1}^\frac{\dd}{p})}
		\,+\,
		\|\,\uu^n\,\|_{L^2(0,t\,;\,\BB_{p, 1}^\frac{\dd}{p})}
		\,\leq\,
		c\sqrt{\nu}.
\end{equation*}
We hence conclude that for $T$ sufficiently small, we have
\begin{equation}\label{iperiper}
	\delta \mathcal{U}^n(t)\,\leq\,\frac 12 \delta \mathcal{U}^{n-1}(t),
	\quad\quad
	\forall\,t\,\in\,[0,T]
	\quad\quad
	\text{and}
	\quad\quad
	\forall\,
	n\in\NN.
\end{equation}
The considered sequence $(\bar \uu^n,\,\bar \tau^n,\,\nabla\bar  \pre^n)_\NN$ is then a Cauchy sequence in $\Y_T$.

\subsection{\bf End of the proof of the local existence} $\,$

\noindent
On the one hand, we have achieved that the sequence $(\bar \uu^n,\,\bar \tau^n,\,\nabla\bar  \pre^n)_\NN$ converges towards a function $(\bar \uu,\,\bar \tau,\,\nabla\bar  \pre)_\NN$ in the functional space $\Y_T$. On the other hand the uniform estimates of Section \ref{sec:uniform-est-fixed-interval} allow us to state that $(\uu^n,\,\tau^n,\,\nabla \pre^n\,)_\NN$ converges to a solution $(\uu,\,\tau,\,\nabla \pre)$ of the co-rotational Oldroyd-B model \eqref{main_system}:
\begin{equation*}
	\uu\,=\,\uu_L\,+\,\bar{\uu},\quad\quad \tau\,=\,\tau_0\,+\,\bar{\tau},\quad\quad \nabla \pre \,=\, \nabla \pre_L\,+\,\nabla \bar \pre.
\end{equation*}
This solution belongs to the functional space given by
\begin{equation}\label{functional_space_sol}
	L^\infty(0,T;\BB_{p, 1}^{\frac{\dd}{p}-1})\cap L^1(0,T;\BB_{p, 1}^{\frac{\dd}{p}+1})\,\times L^\infty(0,T;\BB_{p, 1}^\frac{\dd}{p})\,\times \, L^1(0,T;\BB_{p, 1}^{\frac{\dd}{p}-1}).
\end{equation}
The continuity in time of the solution is determined by Lemma \ref{lemma:bound_of_tau-exp_on_nablau} and Remark \ref{rmk:Stokes-bound}. Indeed \eqref{functional_space_sol} yields that $\,\uu\cdot \nabla \tau -\omega \tau +\tau \omega $ and $\uu\cdot \nabla \uu\,+\,\Div\,\tau$ belongs to $L^1(0,T;\BB_{p, 1}^\frac{\dd}{p})$ and $L^1(0,T;\BB_{p, 1}^{\frac{\dd}{p}-1})$, respectively.

\subsection{\bf Uniqueness}$\,$

\noindent
The uniqueness of Theorem \ref{main_thm1} can be achieved with similar procedures as the one used in Section \ref{sec:conv_in_small_norm}: considering two solutions $(\uu^1,\,\tau^1)$ and  $(\uu^2,\,\tau^2)$ satisfying the condition of Theorem \ref{main_thm1}, we define the difference $\delta \uu := \uu^1-\uu^2$ and $\delta \tau := \tau^1-\tau^2$. Hence, for $p\in [1,2d)$ we can estimate $(\delta \uu,\,\delta\tau)$ in the functional space defined by $\Y_T$ as in \eqref{YT}: denoting by $\delta \mathcal{U}$ the functional
\begin{equation*}
	\delta \mathcal{U}(T)\,:=\,\|\,\delta \uu\,\|_{\tilde L^\infty(0,T\,\BB_{p,1}^{\frac{\dd}{p}-2})}\,+\,\nu \|\,\delta \uu^n\,\|_{\tilde L^1(0,T,\,\BB_{p,1}^{\frac{\dd}{p}})},
\end{equation*} 
we proceed similarly as for proving \eqref{iperiper}.  For $p=2d$ we control
\begin{equation*}
	\delta \mathcal{U}(T)\,:=\,\|\,\delta \uu\,\|_{\tilde L^\infty(0,T\,\BB_{p,\infty}^{\frac{\dd}{p}-2})}\,+\,\nu \|\,\delta \uu^n\,\|_{\tilde L^1(0,T,\,\BB_{p,\infty}^{\frac{\dd}{p}})},
\end{equation*} 
and we use as in the corresponding section in $2D$, the follwoing logarithmic estimate
\begin{equation*}
\begin{aligned}	\| \,  \delta u\,\|_{\tilde L^1(0,T; \BB_{p, 1}^{\frac{2}{p}})}\leq 	\| \,  \delta u\,\|_{\tilde L^1(0,T; \BB_{p, \infty}^{\frac{d}{p}})}\ln \bigg(e+\frac{\| \,  \delta u\,\|_{\tilde L^1(0,T; \BB_{p, \infty}^{\frac{d}{p}-1})}+\| \,  \delta u\,\|_{\tilde L^1(0,T; \BB_{p, 1}^{\frac{d}{p}+1})}}{\| \,  \delta u\,\|_{\tilde L^1(0,T; \BB_{p, \infty}^{\frac{2}{p}})}}\bigg)\\
\leq  	\| \,  \delta u\,\|_{\tilde L^1(0,T; \BB_{p, \infty}^{\frac{2}{p}})}\ln \bigg(e+\frac{\| \,  (u^1, u^2)\,\|_{ L^1(0,T; L^{d,w})}+\| \,  (u_1,u_2)\,\|_{\tilde L^1(0,T; B^{\frac dp+1}_{p,1})}}{\| \, \delta u\,\|_{\tilde L^1(0,T; \BB_{p, \infty}^{\frac{d}{p}})}}\bigg).
\end{aligned}
\end{equation*}
to obtain that
\begin{equation*}
	\delta \mathcal{U}(T) \leq C \int_0^T \delta \mathcal{U}(t)\ln(e+\frac{C}{\delta\mathcal U(t)})\dd t,
\end{equation*}
and to conclude the uniquness as an application of the classical Osgood lemma.
We hence gather that  $\delta \mathcal{U}(T) =0$, which insures the uniqueness of the solution.

\subsection{\bf Extension criterion}$\,$

\noindent
Since the functions $\tau$ and $\uu$ belongs to $L^\infty(0,T^*;\BB_{p, 1}^{\frac{\dd}{p}})$ and  $L^\infty(0,T^*;\BB_{p, 1}^{\frac{\dd}{p}-1})\cap L^1(0,T^*;\BB_{p, 1}^{\frac{\dd}{p}-1})$.  Thanks to \eqref{bound-for-T}, for any $t\in [0,T^*)$ there exists a local solution with initial data $(\uu(t),\,\tau(t))$ on a time interval $[t,t+T]$, with $T>0$ not depending on $t$. Combining such a property with the previous uniqueness result, we can the extend the solution $(\uu,\,\tau)$ for larger time $t$ than $T^*$.

\smallskip
\noindent This concludes the proof of Theorem \ref{main_thm1}.

\section{Global-in-time solutions for small initial data: the case $p\in [1, 2\dd)$}\label{sec:globl-sol}

\noindent
This section is devoted to the proof of Theorem \ref{main_thm2} assuming that $p\in [1, 2\dd)$. The existence of a global in time solution is based on certain suitable estimates. We will begin considering the local solution given by Theorem \ref{main_thm1}, and we will then show that the additional assumptions of Theorem \ref{main_thm2} allow these solutions to be uniformly-in-time bounded in suitable Lorentz spaces.  Hence, under a smallness condition on such a norms, we will propagate the Lipschitz regularity of the velocity field, globally in time. The flow being Lipschitz, we will then be able to propagate the suitable Besov (positive) regularity of Theorem \ref{main_thm2}.

\subsection{\bf Propagation of Lorentz regularities}$\,$

\noindent
Since the initial data $\uu_0$ and $\tau_0$ belongs to $\BB_{p,1}^{\frac{\dd}{p}-1}$, for a $p\in [1, 2\dd)$, thanks to Theorem \ref{main_thm1} there exists a time $T$ and a unique solution $(\uu,\,\tau)$ of the corotational  Johnson-Segalman system \eqref{main_system} in the functional space
\begin{equation*}
	\uu\,\in\, \mathcal{C}([0,T], \BB_{p, 1}^{\frac{\dd}{p}-1})\cap L^1(0,T;\BB_{p, 1}^{\frac{\dd}{p}+1}),
	\quad\quad
	\tau\,\in\,\mathcal{C}([0,T], \BB_{p, 1}^{\frac{\dd}{p}})
\end{equation*}
We first remark that the velocity field $\uu$ satisfies the mild formulation
\begin{equation}\label{mild-formulationn}
 	\uu(t)\,=\,e^{\nu t\Delta}\uu_0	\,+\,\int_0^t\Div\,\Pp e^{\nu(t-s)\Delta}(\uu(s)\otimes \uu(s))\dd s\,+\,\int_0^t\Div\,\Pp e^{\nu(t-s)\Delta}\tau(s) \dd s
\end{equation}
Thus, thanks to Lemma \ref{lemma_Lorentz_bound} and since $\|\,\uu\otimes \uu\,\|_{L^{\dd/2, \infty}_x}\leq \|\,\uu\,\|_{L^{\dd, \infty}_x}^2$, we gather that
\begin{equation*}
	\|\,\uu\,\|_{L^\infty(0, t; L^{\dd, \infty}_x)}
	\,\leq \,
	C
	\left(
		\|\,\uu_0\,\|_{L^{\dd, \infty}}
		\,+\,
		\nu^{-1}
		\|\,\uu\,\|_{L^\infty(0,t; \,L^{\dd, \infty}_x)}^2
		\,+\,
		\nu^{-1}
		\|\,\tau\,\|_{L^{\dd, \infty}_x}
	\right),
\end{equation*}
for any time $t\in [0, T]$. Furthermore, thanks to Lemma \ref{lemma:L^pbound_of_tau}, $\tau$ satisfies
\begin{equation*}
	\|\,\tau(t)\,\|_{L^{\dd, \infty}}
	\,=\,
	\|\,\tau_0\,\|_{L^{\dd, \infty}}.
\end{equation*}
We thus conclude that the Lorentz norm of $\uu$ is uniformly small in time
\begin{equation}\label{small-cond-lorentz-u}
	\|\,\uu\,\|_{L^\infty(0,t;\, L^{\dd, \infty}_x)}\,\leq\,
	\ee \nu ,
\end{equation}
provided that
\begin{equation}\label{small-cond-proof}
	\nu^{-1}\|\,\tau_0\,\|_{L^{\dd, \infty}}
	\,+\,
	\|\,\uu_0\,\|_{L^{\dd,\,\infty}}
	\,\leq\,
	\ee\nu
\end{equation}
for a suitable small parameter $\ee$.

\subsection{\bf Propagation of Lipschitz regularities}$\,$

\noindent We now deal with the Lipschitz regularity of the flow $\uu$ and provide an uniform estimate. We first remark that since $\BB_{p,1}^{\dd/p-1}$, $\BB_{p,1}^{\dd/p}$ and $\BB_{p,1}^{\dd/p+1}$ are continuously embedded into  $\BB_{\infty,1}^{-1}$,  $\BB_{\infty,1}^{0}$ and $\BB_{\infty,1}^{1}$, respectively, the solution $(\uu,\,\tau)$ of Theorem \ref{main_thm1} belongs to
\begin{equation*}
	\uu\,\in\, \mathcal{C}([0,T], \BB_{\infty, 1}^{-1})\cap L^1(0,T;\BB_{\infty, 1}^{1}),
	\quad\quad
	\tau\,\in\,\mathcal{C}([0,T], \BB_{\infty, 1}^{0}),
\end{equation*}
and we can hence define the continuous functional $\mathcal{U}(t)$, $t\in [0,T]$, by means of
\begin{equation*}
	\mathcal{U}(t)\,:=\, \|\,\uu\,\|_{L^\infty(0,t;\,\BB_{\infty, 1}^{-1})}\,+\,\nu\|\,\uu\,\|_{L^1(0,t;\,\BB_{\infty, 1}^{1})}.
\end{equation*}
Next, recasting the equation for $\uu$ into a non stationary linear Stokes problem, we can apply Remark \ref{rmk:Stokes-bound} to gather
\begin{equation}\label{juve1988}
	\mathcal{U}(t)
	\,\leq\,
	C
	\left(
		\|\,\uu_0\,\|_{\BB_{\infty, 1}^{-1}}
		\,+\,
		\|\,\uu\cdot\nabla \uu \,\|_{L^1(0,t;\,\BB_{\infty, 1}^{-1})}
		\,+\,
		\|\,\tau\,\|_{L^1(0,t;\,\BB_{\infty, 1}^{0})}
	\right).
\end{equation}
We hence analyze the nonlinear term $\uu\cdot\nabla \uu = \Div(\uu\otimes\uu)$. First we recast it through the Bony decomposition
\begin{equation*}
		(\uu\cdot \nabla\uu)_j = \partial_i  \dot{T}_{\uu_i}\uu_j\,+\,\partial_i \dot{T}_{\uu_j}\uu_i \,+\,\partial_i \dot R(\uu_i,\,\uu_j),\quad j=1,\dots \dd.
\end{equation*}
Then, we proceed estimating each term on the right hand side. First, we observe that
\begin{equation*}
\begin{aligned}
	\|\,  \dot{T}_{\uu_i}\uu_j \,\|_{\BB_{\infty, 1}^{0}}\,+\,\|\,  \dot{T}_{\uu_j}\uu_i \,\|_{\BB_{\infty, 1}^{0}}\,&=\,\sum_{q\in \ZZ} \sum_{|q-k|\leq 5} \|\,\Dd_q(\Sd_{k-1}\uu_i\Dd_k \uu_j)\, \|_{L^\infty}\,+\,\|\,\Dd_q(\Sd_{k-1}\uu_j\Dd_k \uu_i)\, \|_{L^\infty}\\
	&\lesssim\, \sum_{k\in \ZZ} 2^{-q}\|\,\Sd_{k-1}\uu \,\|_{L^\infty}2^q\|\,\Dd_k \uu\, \|_{L^\infty}\\
	&\lesssim\, \|\,\uu \,\|_{\BB_{\infty, \infty}^{-1}}\|\, \uu\, \|_{\BB_{\infty, 1}^{1}}
	\,\lesssim\,  \|\,\uu \,\|_{L^{\dd, \infty}}\|\, \uu\, \|_{\BB_{\infty, 1}^{1}}.
\end{aligned}
\end{equation*}
Moreover
\begin{equation*}
\begin{aligned}
	\|\, \partial_i \dot R(\uu_i,\,\uu_j)\,\|_{\BB_{\infty, 1}^{0}}
	\,&\leq\,
	\sum_{q\in \ZZ} \sum_{\substack{k\geq q- 5\\|\eta|\leq 1}}\|\,\Dd_q(\Dd_{k}\uu_i\Dd_{k+\eta} \uu_j)\, \|_{L^\infty}\\
	&\leq
	\sum_{q\in \ZZ} \sum_{\substack{k\geq q- 5\\|\eta|\leq 1}}2^{q}2^{-q}\|\,\Dd_q(\Dd_{k}\uu_i\Dd_{k+\eta} \uu_j)\, \|_{L^\infty}\\
	&\leq
	\sum_{q\in \ZZ} \sum_{\substack{k\geq q- 5\\|\eta|\leq 1}}2^{q}\|\,\Dd_{k}\uu_i\Dd_{k+\eta} \uu_j\,\|_{\BB_{\infty, \infty}^{-1}}.
\end{aligned}
\end{equation*}
Since $L^{\dd, \infty}$ is continuously embedded into $\BB_{\infty, \infty}^{-1}$ we hence gather
\begin{equation*}
\begin{aligned}
		\|\, \partial_i \dot R(\uu_i,\,\uu_j)\,\|_{\BB_{\infty, 1}^{0}}
		\,\leq\,
		C\|\,\uu\,\|_{L^{\dd,\infty}}\sum_{q\in \ZZ} 	\sum_{k\in \ZZ} {\bf 1}_{(-\infty, 5]}(q-k)2^{q-k}2^k\|\,\Dd_{k}\uu\,\|_{L^\infty}
		\,\leq\,
		C\|\,\uu\,\|_{L^{\dd,\infty}}
		\|\,\uu\,\|_{\BB_{\infty, 1}^{1}}.
\end{aligned}
\end{equation*}
Summarizing the previous estimates with \eqref{juve1988}, we eventually get
\begin{equation*}
	\mathcal{U}(t)
	\,\leq\,
	C
	\left(
		\|\,\uu_0\,\|_{\BB_{\infty, 1}^{-1}}
		\,+\,
		\|\,\tau\,\|_{L^1(0,t;\,\BB_{\infty, 1}^{0})}
	\right)
	\,+\,
	C\|\,\uu\,|_{L^\infty(0,T; L^{\dd, \infty})}\mathcal{U}(t),
\end{equation*}
thus, recalling the smallness condition \eqref{small-cond-lorentz-u} of the Lorentz norm of $\uu$, we deduce that
\begin{equation*}
	\mathcal{U}(t)
	\,\leq\,
	C
	\left(
		\|\,\uu_0\,\|_{\BB_{\infty, 1}^{-1}}
		\,+\,
		\|\,\tau\,\|_{L^1(0,t;\,\BB_{\infty, 1}^{0})}
	\right).
\end{equation*}
Next, in virtue of Lemma \ref{lemma:bound_of_tau-linear_on_nablau}
\begin{equation*}
	\|\,\tau(t)\,\|_{\BB_{\infty, 1}^{0}}
	\,\leq\,
	C
	\|\,\tau_0\,\|_{\BB_{\infty, 1}^{0}}
	\left(
		1\,+\,\|\,\uu\,\|_{L^1(0,t;\,\BB_{\infty, 1}^{1})}
	\right)
\end{equation*}
from which we obtain
\begin{equation*}
	\mathcal{U}(t)\,\leq\,C\|\,\uu_0\,\|_{\BB_{\infty, 1}^{-1}}
	\,+\,
	Ct\|\,\tau_0\,\|_{\BB_{\infty, 1}^{0}}
	\,+\,
	C\nu^{-1}
	\|\,\tau_0\,\|_{\BB_{\infty, 1}^{0}}
	\int_0^t
	\mathcal{U}(\tau)\dd \tau,
\end{equation*}
hence thanks to the Gronwall Lemma we deduce that
\begin{equation}\label{iperiperiperiper}
	\mathcal{U}(t)\,\leq\, 
	C\|\,\uu_0\,\|_{\BB_{\infty, 1}^{-1}}
	\exp\left\{	Ct\nu^{-1}\|\,\tau_0\,\|_{\BB_{\infty, 1}^0}	\right\}
	\,+\,
	\nu 
	\left(
		\exp\left\{	Ct\nu^{-1}\|\,\tau_0\,\|_{\BB_{\infty, 1}^0}	\right\}
		\,-\,1
	\right),
\end{equation}
for any time $t\in [0, T]$.
\subsection{\bf Proof of global existence}\label{sec:imp}$\,$

\noindent
Let $T^*$ be the largest time of existence of the solution $(\uu,\,\tau)$ determined by Theorem \ref{main_thm1}. We claim that under the condition of Theorem \ref{main_thm2}, the lifespan $T^*$ satisfies $T^*=+\infty$. We proceed by contradiction, assuming $T^*<+\infty$. For any time $T\in (0, T^*)$, $(\uu,\,\tau)$ belongs to 
\begin{equation*}
	\mathcal{C}([0,T],\,\BB_{p,1}^{\frac{\dd}{p}-1})\cap L^1(0,T;\BB_{p,1}^{\frac{\dd}{p}+1})\times \mathcal{C}([0,T],\BB_{p,1}^{\frac{\dd}{p}}),
\end{equation*}
which is continuously embedded into
\begin{equation*}
	\mathcal{C}([0,T],\,\BB_{\infty,1}^{-1})\cap L^1(0,T;\BB_{\infty,1}^{1})\times \mathcal{C}([0,T],\BB_{\infty,1}^{0}).
\end{equation*}
Furthermore, thanks to \eqref{iperiperiperiper}, the following  estimate of the Lipschitz regularity of $\uu$ holds true for any time $T\in (0,T^*)$:
\begin{equation*}
	\|\,\uu\,\|_{L^\infty(0,T;\BB_{\infty, 1}^{-1})}\,+\,\nu\|\,\uu\,\|_{L^1(0,T;\BB_{\infty, 1}^1)}
	\,\leq\,\Theta_\nu(\uu_0,\,\tau_0,\,T),
\end{equation*}
where $\Theta_\nu$ is a growing continuous function depending on time $T$ by
\begin{equation*}
	\Theta_\nu(\uu_0,\,\tau_0,\,T):=
	C\|\,\uu_0\,\|_{\BB_{\infty, 1}^{-1}}
	\exp\left\{	CT\nu^{-1}\|\,\tau_0\,\|_{\BB_{\infty, 1}^0}	\right\}
	\,+\,
	\nu 
	\left(
		\exp\left\{	CT\nu^{-1}\|\,\tau_0\,\|_{\BB_{\infty, 1}^0}	\right\}
		\,-\,1
	\right).
\end{equation*}
Applying Lemma \ref{lemma:bound_of_tau-exp_on_nablau}, we deduce that for any time $T\in (0,T^*)$
\begin{equation*}
\begin{aligned}
	\|\,\tau\,\|_{L^\infty(0,T;\BB_{p,1}^\frac{\dd}{p})}
	\,&\leq\,
	\|\,\tau_0\,\|_{\BB_{p,1}^{\frac{\dd}{p}}}
	\exp
	\left\{
	C\int_0^T
	\Theta_\nu(\uu_0,\,\tau_0,\,t)\dd t
	\right\}\\
	&\leq\,
	\|\,\tau_0\,\|_{\BB_{p,1}^{\frac{\dd}{p}}}
	\exp
	\Big\{
	CT^*
	\Theta_\nu(\uu_0,\,\tau_0,\,T^*)\dd t
	\Big\}
	<+\infty.
\end{aligned}
\end{equation*} 
We hence deduce that $\tau$ belongs to $L^\infty(0,T^*,\,\BB_{p,1}^{\dd/p})$.

\noindent Next, we take into account the velocity field $\uu$ and we remark that for any $t\in (0,T^*)$
\begin{equation*}
	\|\,\uu(t)\,\|_{\BB_{p,1}^{\frac{\dd}{p}-1}}
	\,+\,
	\nu
	\int_0^t
	\|\,\uu(s)\,\|_{\BB_{p,1}^{\frac{\dd}{p}+1}}
	\dd s
	\,\leq\,
	\|\,\uu_0\,\|_{\BB_{p,1}^{\frac{\dd}{p}-1}}
	\,+\,
	\int_0^t
	\|\,\uu(s)\,\|_{\BB_{\infty,1}^{0}}
	\|\,\uu(s)\,\|_{\BB_{p,1}^{\frac{\dd}{p}}}
	\dd s
	\,+\,
	\int_0^t\|\,\tau(s)\,\|_{\BB_{p,1}^{\frac{\dd}{p}}}
	\dd s
\end{equation*}
from which we deduce that
\begin{equation*}
\begin{aligned}
	\|\,\uu(t)\,\|_{\BB_{p,1}^{\frac{\dd}{p}-1}}
	\,+\,
	\nu
	\int_0^t
	\|\,\uu(s)\,\|_{\BB_{p,1}^{\frac{\dd}{p}+1}}
	\leq\,
	\|\,\uu_0\,\|_{\BB_{p,1}^{\frac{\dd}{p}-1}}
	\,&+\,
	C\nu^{-1}\int_0^t
	\|\,\uu(s)\,\|_{\BB_{\infty,1}^{0}}^2
	\|\,\uu(s)\,\|_{\BB_{p,1}^{\frac{\dd}{p}-1}}
	\dd s
	\,+\\&+\,
	CT^*
	\|\,\tau_0\,\|_{\BB_{p,1}^{\frac{\dd}{p}}}
	\exp
	\Big\{
	CT^*
	\Theta_\nu(\uu_0,\,\tau_0,\,T^*)
	\Big\}.
\end{aligned}
\end{equation*}
Finally, applying the Gronwall inequality we gather that for any time $t\in (0,T^*)$,
\begin{equation*}
	\|\,\uu(t)\,\|_{\BB_{p,1}^{\frac{\dd}{p}-1}}
	\,+\,
	\nu
	\int_0^t
	\|\,\uu(s)\,\|_{\BB_{p,1}^{\frac{\dd}{p}+1}}
	\,\leq\,
	\bigg\{
		\|\,\uu_0\,\|_{\BB_{p,1}^{\frac{\dd}{p}-1}}
		\,+\,
		CT^*
	\|\,\tau_0\,\|_{\BB_{p,1}^{\frac{\dd}{p}}}
	{\rm e}^{
	CT^*
	\Theta_\nu(\uu_0,\,\tau_0,\,T^*)
	}
	\bigg\}
	{\rm e}^{C\nu^{-1}\Theta_\nu(\uu_0,\,\tau_0,\,T^*)}
	<+\infty.
\end{equation*}
The above inequality implies that $\uu$ belongs to $L^\infty(0,T^*;\BB_{p,1}^{\frac{\dd}{p}-1})\cap L^1(0,T^*,\BB_{p,1}^{\frac{\dd}{p}+1})$. The prolongation criterion of Theorem \ref{main_thm1} hence allows to extend in time the solution $(\uu,\,\tau)$ above the lifespan $T^*$, which contradicts the maximality of $T^*$, itself. Thus $T^*=+\infty$ and this concludes the proof of Theorem \ref{main_thm2} for $p\in [1,2\dd)$.

\section{Global-in-time solutions for small initial data: the case $p\in [2\dd, \infty)$}\label{sec:globl-sol2}

\noindent
This section is devoted to conclude the proof of Theorem \ref{main_thm2}, namely showing the existence of global-in-time classical solutions when $p\in [2\dd, \infty)$. In this setting, the uniqueness of these solutions is not determined as in the case of the previous section, since the regularity of the velocity field is below the critical negative value $\dd/p-1<-1/2$. 

\noindent 
We begin with regularizing the initial data $(\uu_0,\,\tau_0)$ as follows:
\begin{equation*}
	\uu_0^n:= J^n \uu_0\quad\quad\quad \tau_0^n:= J^n\tau_0. 
\end{equation*}
The regularized initial data $(\uu_0^n,\,\tau_0^n)$ belongs to $\BB_{\dd, 1}^{0}\times \BB_{\dd, 1}^1$ as well as the smallness condition
	\begin{equation*}
	\begin{alignedat}{8}
		&\left\|\,\uu_0^n	\,\right\|_{L^{\dd, \infty}}\,+\,
		\frac{1}{\nu}
		\left\|\,\tau_0^n	\,\right\|_{L^{\frac{\dd}{2}, \infty}}
		\,&&\leq\,
		\frac{\ee}{\nu}
	\end{alignedat}
	\end{equation*}
is still satisfied. Thanks to Theorem \ref{main_thm2}, with $p=\dd \in [1, 2\dd)$, there exists a unique global-in-time solutions $(\uu^n,\,\tau^n)$ of the system
\begin{equation*}
\left\{\hspace{0.2cm}
	\begin{alignedat}{2}
		&\,\partial_t \tau^{n}\,+\,\uu^n\cdot \nabla \tau^{n}	- \omega^n \tau^{n} \,+\,\tau^{n}\omega^n \,=\,0
		\hspace{3cm}
		&&\RR_+\times \RR^2, \vspace{0.1cm}	\\		
		&\,\partial_t \uu^{n} \,+\,\uu^{n}\cdot \nabla \uu^{n}\,  - \, \nu \Delta \uu^{n}  
		\,+\, \nabla  \pre^{n} \, =\,
		\Div\,\tau^{n}
									\hspace{3cm}									&& \RR_+\times \RR^2,\vspace{0.1cm}\\
		&\,\Div\, \uu^{n}\,=\,0			
		&&\RR_+\times \RR^\dd, \vspace{0.1cm}	\\			
		& ( \uu^{n},\,\tau^{n})_{|t=0}	\,=\,(\uu_0^n,\,\tau_0^n)			
		&&\hspace{1.02cm} \RR^2,\vspace{0.1cm}																							
	\end{alignedat}
	\right.
\end{equation*}
satisfying
\begin{equation*}
	\begin{alignedat}{4}
		\uu^n	\,	&\in	\,&& \mathcal{C}\big(\,\RR_+, \,\BB_{\dd, 1}^{0}\cap  L^{\dd, \infty}(\RR^\dd)\,\big)\cap L^1_{\rm loc}(\RR_+,\BB_{\dd, 1}^{2}),\quad\quad
		\tau^n	\,	\in	\, \mathcal{C}(\,\RR_+, \,\BB_{\dd, 1}^{1}\cap L^{\frac{\dd}{2}, \infty}(\RR^\dd)\,).
	\end{alignedat}
	\end{equation*}
Proceeding as in Section \ref{sec:imp}, we get that $(\uu^n,\,\tau^n)$ is uniformly bounded in the functional space
\begin{equation*}
	\begin{alignedat}{4}
		\uu^n		\,	&\in	\,&& \mathcal{C}\big(\,\RR_+, \,\BB_{p, 1}^{\frac{\dd}{p}-1}\cap  L^{\dd, \infty}(\RR^\dd)\,\big)\cap L^1_{\rm loc}(\RR_+,\BB_{p, 1}^{\frac{\dd}{p}+1}),\quad\quad
		\tau^n	\,	\in	\, \mathcal{C}(\,\RR_+, \,\BB_{p, 1}^{\frac{\dd}{p}}\cap L^{\frac{\dd}{2}, \infty}(\RR^\dd)\,).
	\end{alignedat}
\end{equation*}
In order to conclude, we need to pass to the limit as $n$ goes to $\infty$. We proceed similarly as in Section \ref{sec:limits-pas}. We fix a compact set $K$ in $\RR^2$ and we introduce the functional spaces 
\begin{equation*}
				X_0\,:=\, B_{p, 1}^{\frac{2}{p}-1}(K),\quad\quad X\,=\,X_1\,:=B_{p, 1}^{\frac{\dd}{p}-2}(K),
\end{equation*}
for the Aubin-Lions Lemma \ref{Aubin-Lions-Lemma}. Since $(\uu^n)_\NN$ is uniformly bounded in $L^2_{\rm loc}(\RR_+,\,\BB_{p,1}^{\dd/p})$ then, by embedding it is uniformly bounded in $L^2_{\rm loc}(\RR_+,\,L^\infty(\RR^\dd))$. Interpolating this result with the uniform bound of $(\uu^n)_\NN$ in $L^\infty_{\rm loc}(\RR_+, L^{\dd, \infty}(\RR^\dd))$ allows us to conclude that $(\uu^n)_\NN$ is uniformly bounded in  $L^{{2p}/(p-\dd)}_{\rm loc}(\RR_+,\,L^p(\RR^\dd))$ and thus in $L^2_{\rm loc}(\RR_+,\,B_{p,1}^{2/p})$, where $B_{p,1}^{2/p}$ is a \textit{non-homogeneous} Besov space. 
 Furthermore, since $\BB_{p,1}^{\dd/p}$ is continuously embedded into $L^\infty(\RR^\dd)$, the sequence $(\tau^n)_\NN$ is uniformly bounded into $L^\infty_{\rm loc}(\RR_+,L^p(\RR^2))$ and so into $L^\infty_{\rm loc}(\RR_+,\,B_{p,1}^{\dd/p})$ which is embedded into $L^2_{\rm loc}(\RR_+,B_{p,1}^{\dd/p-1})$. This allows us to conclude that 
\begin{equation*}
		\|\,\Div\,\tau^n\,\|_{L^2_{\rm loc}(\RR_+,B_{p,1}^{\frac{\dd}{p}-2}(K))} \leq C,
\end{equation*}
for a suitable constant $C$ that does not depend on the index $n\in \NN$.  Now, we claim that $\Div\,j^n(\uu^n\otimes \uu^n)$ is uniformly bounded in $L^r(0,T; B_{p,1}^{\dd/p-2}(K))$, for a suitable positive index $r>1$.  We first remark that
\begin{equation*}
\begin{aligned}
	\|\,\Div\,j^n(\uu^n\otimes \uu^n)\,\|_{B_{p,1}^{\frac{\dd}{p}-2}(K)}
	\,&\leq\,
	\|\,j^n(\uu^n\otimes \uu^n)\,\|_{B_{p,1}^{\frac{\dd}{p}-1}(K)}
	\,\leq\,
	\|\,j^n(\uu^n\otimes \uu^n)\,\|_{B_{p,1}^{\frac{\dd}{p}-1}}\\
	&\lesssim\,
	\|\,\Sd_{-1} j^n(\uu^n\otimes \uu^n)\,\|_{L^p(\RR^2)}
	\,+\,
	\|\,(\Id -\Sd_{-1}) j^n(\uu^n\otimes \uu^n)\,\|_{\BB_{p,1}^{\frac{\dd}{p}-1}}\\
	&\lesssim\,
	\|\,j^n(\uu^n\otimes \uu^n)\,\|_{L^p(\RR^2)}
	\,+\,
	\|\,(\Id -\Sd_{-1}) j^n(\uu^n\otimes \uu^n)\,\|_{\BB_{p,1}^{\frac{\dd}{p}-\ee}},
\end{aligned}
\end{equation*}
where $\ee\in(0,1)$ such that $\dd/p-\ee>0$.  Hence we eventually gather that
\begin{equation*}
	\|\,\Div\,j^n(\uu^n\otimes \uu^n)\,\|_{B_{p,1}^{\frac{\dd}{p}-2}(K)}
	\,\leq\,
	C
	\left(
	\|\,(\uu^n\otimes \uu^n)\,\|_{L^p(\RR^2)}
	\,+\,
	\|\,(\uu^n\otimes \uu^n)\,\|_{\BB_{p,1}^{\frac{\dd}{p}-\ee}}
	\right).
\end{equation*}
Now, making use of the continuity of the product between the functional spaces
\begin{equation*}
	L^\frac{2}{1-\ee}_{\rm loc}(\RR_+, \BB_{p,1}^{\frac{\dd}{p}-\ee})\times L^2_{\rm loc}(\RR_+, \BB_{p,1}^{\frac{\dd}{p}})
	\rightarrow
	L^\frac{1}{1-\ee}_{\rm loc}(\RR_+, \BB_{p,1}^{\frac{\dd}{p}-\ee}),
\end{equation*}
the term $\,\uu^n\otimes \uu^n\,$ is uniformly bounded into $L^{1/(1-\ee)}_{\rm loc}(\RR_+, \BB_{p,1}^{\frac{2}{p}-1})$. Furthermore, we can bound the $L^p(\RR^2)$ norm through the following interpolation:
\begin{equation*}
	\|\,\uu^n\otimes \uu^n\,\|_{L^p(\RR^2)}\leq C\|\,\uu^n\,\|_{L^{2p}(\RR^2)}^2\leq C\|\,\uu^n\,\|_{L^p(\RR^2)}\|\,\uu^n\,\|_{L^{\infty}(\RR^2)} \in L_{\rm loc}^{\frac{2p}{2p-\dd}}(\RR_+).
\end{equation*} 
Denoting by $r\,=\,\min\{2,1/(1-\ee), 2p/(2p-\dd)\}>1$, we gather that the sequence $(\partial_t \uu^n)_\NN$ satisfying
\begin{equation*}
	\partial_t \uu^n\,=\,\Delta\uu^n\,-J^n(\uu^n\cdot \nabla \uu^n)\,+\,\nabla \pre^n\,+\,\Div\,\tau^n 
\end{equation*}
is uniformly bounded in $L^r_{\rm loc}(\RR_+,\,B_{p,1}^{\dd/p-2}(K))$. Hence, the Aubins-Lion Lemma \ref{Aubin-Lions-Lemma} and the generality of the compact set $K$ allow us to extract a convergent subsequence $(\uu^{n_k})_\NN\subset (\uu^n)_\NN$ such that
\begin{equation*}
	\uu^{n_k}\,\rightarrow\,\uu\quad\quad\text{in}\quad\quad L^\infty\big(0,T;\,(B_{p,1}^{\frac{\dd}{p}-2})_{\rm loc}\big).
\end{equation*} 

\smallskip\noindent 
We now claim that $(\partial_t \tau^n)_\NN$ is uniformly bounded in the \textit{non-homogeneous} functional space $L^r_{\rm loc}(\RR_+,\,B_{p,1}^{\dd/p-1}(K))$. First we recall that the conformation tensor satisfies
\begin{equation*}
	\partial_t\tau^n\,=\,-\uu^n\cdot \nabla \tau^n\,+\,\omega^n\tau^n\,-\,\tau^n\omega^n,
\end{equation*}
therefore
\begin{equation*}
	\|\,\partial_t \tau^n\,\|_{\BB_{p,1}^{\frac \dd p -1}}\,\lesssim\,\|\,\uu^n\,\|_{\BB_{p,1}^{\frac  \dd p }}\|\,\nabla  \tau^n\,\|_{\BB_{p,1}^{\frac{ \dd }{p}-1}}
	\,+\,\|\,\nabla \uu^n\,\|_{\BB_{p,1}^{\frac{ \dd }{p}-1}}\|\,\tau^n\,\|_{\BB_{p,1}^{\frac{ \dd }{p}-1}}
	\,\in\,L^2_{\rm loc}(\RR_+).
\end{equation*}
Furthermore
\begin{equation*}
\begin{aligned}
	\|\,	-\omega^n\tau^n + \tau^n\omega^n	\,\|_{L^p(\RR^2)}
	&\,\lesssim\,
	\|\,\nabla \uu^n\,\|_{L^{p}(\RR^2)}
	\|\,\tau^n\,\|_{L^{\infty}(\RR^2)}
	\,\lesssim\,
	\|\,\nabla \uu^n\,\|_{\BB_{p,1}^0}
	\|\,\tau^n\,\|_{\BB_{p,1}^\frac{\dd}{p}}\\
	&\lesssim\,
	\|\,\uu^n\,\|_{\BB_{p,1}^1}
	\|\,\tau^n\,\|_{\BB_{p,1}^\frac{\dd}{p}}
	\,\lesssim\,
	\|\,\uu^n\,\|_{\BB_{p,1}^{\frac{\dd}{p}-1}}^\frac{\dd}{2p}
	\|\,\uu^n\,\|_{\BB_{p,1}^{\frac{\dd}{p}+1}}^{1-\frac{\dd}{2p}}
	\|\,\tau^n\,\|_{\BB_{p,1}^\frac{\dd}{p}}
	\in L^{\frac{2p}{2p-\dd}}_{\rm loc}(\RR_+).
\end{aligned}
\end{equation*}
Finally, one has
\begin{equation*}
\begin{aligned}
	\|\,\Delta_{-1}\Div(\uu^n\otimes \tau^n)\,\|_{L^p(\RR^2)}
	\,&\lesssim\,
	\|\,\uu^n\otimes \tau^n\,\|_{L^p(\RR^2)}
	\,\lesssim\,
	\|\,\uu^n\,\|_{L^{p}(\RR^2)}
	\|\,\tau^n\,\|_{L^{\infty}(\RR^2)}\\
	\,&\lesssim\,
	\|\, \uu^n\,\|_{L^p(\RR^2)}
	\|\, \tau^n\,\|_{\BB_{p,1}^\frac{\dd}{p}}
	\,\in\,
	L^\frac{2p}{p-\dd}_{\rm loc}(\RR^2),
\end{aligned}
\end{equation*}
which allows us to conclude that $(\partial_t \tau^n)_\NN$ is uniformly bounded in $L^r_{\rm loc}(\RR_+,\,B_{p,1}^{2/p-1}(K))$. Thus, 
the Aubins-Lion lemma together with the arbitrariness of the compact set $K$ allow us to extract a convergent subsequence $(\tau^{n_k})_\NN\subset (\uu^n)_\NN$ such that
\begin{equation*}
	\tau^{n_k}\,\rightarrow\,\tau\quad\quad\text{in}\quad\quad L^\infty\big(0,T;\,(B_{p,1}^{\frac{\dd}{p}-1})_{\rm loc}\big).
\end{equation*} 
These properties allow to pass to the limit and thus to show that $(\uu,\,\tau)$ is a global-in-time solution of system \eqref{main_system}. This concludes the proof of Theorem \ref{main_thm2}.

\section*{Acknowledgment} 

\noindent This manuscript benefited from comments and remarks of the reviewer. We are thankful to the reviewer for the helpful feedback.

{

}

\end{document}